\documentclass[11pt]{article}
\usepackage{amsmath, amsthm, amssymb, mathrsfs}
\usepackage{geometry}
\newtheorem{theorem}{Theorem}[section]
\newtheorem{lemma}[theorem]{Lemma}
\newtheorem{cor}[theorem]{Corollary}
\newtheorem{rem}[theorem]{Remark}
\newtheorem{proposition}[theorem]{Proposition}

\numberwithin{equation}{section}

\usepackage{hyperref}
\usepackage{dsfont}

\usepackage[
    backend=biber,     
    style=numeric,
    maxnames=99,
    giveninits=true,   
    doi=false,         
    isbn=false,        
    url=false,          
    eprint=true
]{biblatex}

\renewbibmacro{in:}{%
  \ifentrytype{article}{}{\printtext{\bibstring{in}\intitlepunct}}}

\DeclareFieldFormat*{title}{#1}

\usepackage{xcolor,todonotes}

\newcommand{\R}{\mathbb{R}}
\newcommand{\N}{\mathbb{N}}
\newcommand{\E}{\mathbb{E}}
\renewcommand{\P}{\mathbb{P}}
\newcommand{\ind}{\mathds{1}}

\newcommand{\eps}{\varepsilon}
\newcommand{\Prob}[1]{\mathbb P\{#1\}}

\newcommand{\Var}{\operatorname{Var}}
\newcommand{\dd}{{\mathrm d}}

\makeatletter
\def\namedlabel#1#2{\begingroup
	#2%
	\def\@currentlabel{#2}%
	\phantomsection\label{#1}\endgroup
}
\makeatother
\allowdisplaybreaks[4]
\title{Dickman Approximation of Randomly Weighted Sums via Stein's Method}
\author{Chinmoy Bhattacharjee \footnote{\raggedright Department of Mathematics, University of Hamburg, Germany;~\href{mailto:chinmoy.bhattacharjee@uni-hamburg.de}{\texttt{chinmoy.bhattacharjee@uni-hamburg.de}}.}}
\date{\today}

\begin{document}
	
	\maketitle
	\begin{abstract}
We introduce a unified framework via Stein's method for bounding the Kolmogorov distance between the generalized Dickman distributions and the distribution of randomly weighted sums of non-negative integer-valued random variables that are conditionally independent given the weights. By utilizing size-bias couplings and a decomposition of the solution to the corresponding Stein equation into distinct bounded and non-increasing components, our approach yields non-asymptotic error bounds governed only by the quality of a discrete-to-continuous coupling. A similar bound is also obtained for the smooth Wasserstein-2 distance. We apply our abstract results to establish concrete, and in some cases optimal, rates of convergence across diverse probabilistic models. These applications range from settings with independent weights, such as randomly weighted sums of log-primes, to models with intricate dependency structures, including geometric sums over the spectra of the Circular Unitary Ensemble (CUE) and random weights generated by independent increments.
	\end{abstract}

\medskip


\medskip


	%
%
	\section{Introduction and main results}\label{sec:mainres}

The generalized Dickman distribution with mean $\theta>0$, denoted by $\mathcal{D}_\theta$, traces its origins to the work of Dickman \cite{Dickman} on smooth numbers, where it describes the asymptotic density of integers having no large prime factors. Since its number-theoretic genesis, it has been shown to arise naturally as the limiting distribution for a wide variety of logarithmic combinatorial structures \cite{ABT,BN11,Pi16b, Pinsky16a}. Classical examples include the scaled size of the largest prime factor of a uniformly chosen integer \cite{Bil,KP76,Ver} and the scaled length of the longest cycle in a random permutation \cite{King,VS77}.

A particularly elegant probabilistic interpretation of the generalized Dickman distributions is given by its characterization as a distributional fixed point. Concretely, $\mathcal{D}_\theta$ is uniquely defined as the non-negative fixed point of the distributional transformation $W \to W^*$ given by
\begin{equation*}\label{def:Dickman}
	W \stackrel{d}{=} U^{1/\theta}(W+1),
\end{equation*}
where $U \sim\mathbb{U}[0,1]$ is a uniform random variable on the interval $[0,1]$  independent of $W$ with $\stackrel{d}{=}$ denoting equality in distribution. Throughout this paper, $D_\theta \sim \mathcal{D}_\theta$ denotes a Dickman distributed random variable with mean $\theta$, and we will simply write $D \equiv D_1$ for a standard Dickman distributed random variable. 

Motivated by its appearance across such varied frameworks, there has been significant recent interest in understanding the general domain of attraction for these generalized Dickman distributions. Many random combinatorial or number theoretic objects that weakly converge to these distributions can often be expressed as a weighted (possibly random) sum of independent random variables. Let $(X_k)_{k \in \N}$ be a sequence of non-negative integer-valued random variables, henceforth called \textit{counts}, and $(W_{n,k})_{k \in [n]}$ be a triangular array of non-negative real-valued random \textit{weights}, where we denote $[n] = \{1,\hdots, n\}$. We study the sum
\begin{equation}\label{eq:S_n}
	S_n = \sum_{k=1}^n W_{n,k} X_k.
\end{equation}

We quantify the Dickman convergence of such sums by establishing non-asymptotic bounds in the Kolmogorov distance -- arguably the most natural metric for real-valued random variables.  For two non-negative random variables $X$ and $Y$, the Kolmogorov distance $d_K$ between (the distributions of) $X$ and $Y$ is defined as
\begin{equation*}\label{eq:dK}
	d_{K}(X,Y) =\sup_{t \ge 0} |\P\, (X \le x)-\P\, (Y \le t)|.
\end{equation*}

Dickman convergence for certain specific deterministically weighted sums has been previously studied in the literature, with optimal bounds in the Kolmogorov distance recently obtained for certain weighted Bernoulli sums in \cite{BS25}. When weights are truly random, \cite{Pi16b} provides certain conditions for weak convergence of independently weighted Bernoulli sums to generalized Dickman distributions. However, the existing literature (e.g., \cite{BG19, BS25}) for quantitative results in the case of random weights generally restricts convergence bounds to the smoother Wasserstein-2 distance $d_{1,1}$, which takes the supremum over twice continuously differentiable test functions with bounded first and second derivatives. Furthermore, to the best of our knowledge, no quantitative bounds exist for settings exhibiting genuine dependence among the weights and the counts. This work bridges this gap.


Motivated by the natural question of identifying the general domain of attraction for generalized Dickman distributions, we provide non-asymptotic bounds for Dickman convergence applicable to a broad class of randomly weighted sums. In particular, we significantly generalize the framework of \cite{BS25} to encompass random sums of the form \eqref{eq:S_n}. Existing works typically restrict their scope to specific classes of count variables (such as Poisson or Binomial) and impose rigid conditions on the random weights. By leveraging a non-increasing function decomposition of the Stein equation solution introduced in \cite{BS25}, and using size-biasing, we construct a unified framework that applies under far broader conditions.

\medskip
Before presenting our main result, we first state the assumptions we need on the weights $\mathbf{W} = (W_{n,k})_{n \in \N, k \in [n]}$ and the random variables $(X_k)_{k \in \N}$. 
\medskip

\begin{itemize}
	\item[\namedlabel{a2}{(\textbf{A1})}]  We assume that the random weights $\mathbf{W} = (W_{n,k})_{n \in \N, k \in [n]}$ are non-negative real-valued random variables which are \textit{ordered}, i.e.,
	$$
	0 = :W_{n,0} \le W_{n,1} \le \dots \le W_{n,n} =: W_{\max}^{(n)};
	$$
	and are \textit{integrable}, i.e., satisfying $\E[W_{\max}^{(n)}] \in [0,\infty)$. Here by convention, we take $W_{n,0} = 0$.
	\item[\namedlabel{a1}{(\textbf{A2})}] Further, we assume that conditionally given $\mathbf{W}$, the sequence $(X_k)_{k \in \N}$ consists of mutually independent, non-negative integer-valued random variables.
\end{itemize}

Note that we do not assume any independence among the weights; in particular, they can be highly dependent, for instance, order statistics of independent random variables (see Section \ref{sec:orderstat}), or the spectra of a Circular Unitary Ensemble (see Section \ref{sec:CUE}). Moreover, we allow dependence between the weights $\mathbf{W}$ and the sequence $(X_k)_{k \in \N}$; see Section \ref{sec:CUE} for an application where this becomes relevant.

\medskip

Writing 
$
\lambda_k(\mathbf{W}) := \E[X_k \mid \mathbf{W}] \in [0,\infty),
$
define for $k \in [n]$ the random variables
\begin{equation*}
	\Delta_{n,k} := \theta (W_{n,k} - W_{n,k-1}), \quad \text{and} \quad D_{n,k} := \sum_{j=1}^k (\lambda_j(\mathbf{W}) W_{n,j} - \Delta_{n,j}).
\end{equation*}
For notational simplicity, denote
\begin{equation}\label{def:delD}
	\delta_n = \max_{k \in [n]} ({W}_{n,k} - {W}_{n,k-1}) = \frac{1}{\theta} \max_{k \in [n]}  \Delta_{n,k}, \quad \text{and} \quad D_n^* = \max_{k \in [n]} |D_{n,k}|.
\end{equation}
Also, let $x_{\min}^{(k)}$ be the smallest value that $X_k$ can take with positive probability, i.e., 
\begin{equation}\label{eq:xmin}
	x_{\min}^{(k)}(\mathbf{W}) := \min \{ x \in \mathbb{N}_0 : \P(X_k = x \mid \mathbf{W}) > 0 \}.
\end{equation}
For each $k$, conditionally given $\mathbf{W}$, let $X_k^*$ be a random variable having the size-biased distribution of $X_k$, defined by the relation 
\begin{equation}\label{eq:sb}
	\E[X_k f(X_k) \mid \mathbf{W}] = \lambda_k(\mathbf{W}) \E[f(X_k^*) \mid \mathbf{W}]
\end{equation}
for any test function $f$ for which the expectations exist. Define the aggregate error
\begin{equation}\label{eq:eps_n_random}
	\eps_n(\mathbf{W}) := \delta_n + \frac{2}{\theta} D_n^* + \frac{2}{\theta}\sum_{k=1}^n \lambda_k(\mathbf{W}) W_{n,k} \,\frac{ d_{TV}(X_k^*, X_k+1 \mid \mathbf{W})}{\P(X_k = x_{\min}^{(k)}(\mathbf{W}) \mid \mathbf{W})} + |W_{\max}^{(n)} - 1|,
\end{equation}
where $d_{TV}(\cdot, \cdot \mid \mathbf{W})$ denotes the conditional total variation distance, given $\mathbf{W}$.

Let $p_\theta$ denote the density of $D_\theta$. We note the fact (see e.g.\ \cite[Section~1]{Pi16b}) that there exists a constant $k_\theta>0$ such that 
$$
p_\theta(x) \le k_\theta \,\text{ for } \; x \ge 1,
$$ 
and in particular, we can take $k_1=e^{-\gamma}$, where $\gamma$ is the Euler-Mascheroni constant. Denote
\begin{equation}\label{eq:C_constants}
	C_1 := 2\theta^2 + (6\theta  + 10 \theta \, \zeta(1+\theta)) \max \left\{\frac{e^{-\theta \gamma}}{\Gamma(\theta)}, k_\theta\right\}, \quad \text{and} \quad C_2 := 3 (2\theta  + 4 \theta \, \zeta(1+\theta))^{\theta}.
\end{equation}

\begin{theorem}\label{thm:randrates_cond}
	Let $S_n = \sum_{k=1}^n W_{n,k} X_k$ be as in \eqref{eq:S_n} satisfying assumptions \ref{a2}--\ref{a1}. Let $C_1$ and $C_2$ be the constants as in \eqref{eq:C_constants} depending only on $\theta$. Then,
	$$
	d_K(S_n, D_\theta) \le \begin{cases} 
		C_1 \E[\eps_n(\mathbf{W})], & \text{if } \theta \ge 1, \\[6pt] 
		2\theta^2 \E[\eps_n(\mathbf{W})] + C_2 \left(\E[\eps_n(\mathbf{W})]\right)^\theta, & \text{if } \theta \in (0,1). 
	\end{cases}
	$$
\end{theorem}

The proof of Theorem \ref{thm:randrates_cond} is carried out in Section \ref{sec:1proofs}. We remark here on the different components in the aggregate error term $\eps_n(\mathbf{W})$. First note that $\lambda_j(\mathbf{W}) W_{n,j}, j \in [n]$ is simply the conditional expectation of the $j$-th summand in $S_n$, given the weights. If we consider the weights $\mathbf{W}$ to be deterministic, the approximation $\lambda_k(\mathbf{W}) W_{n,j} \approx \Delta_{n,j}$ can be interpreted as distributing the expected mass of the $k$-th summand over the continuous interval $(W_{n,j-1}, W_{n,j}]$ with intensity $\theta$. If $\delta_n \to 0$ as $n \to \infty$, then the intervals shrink, thereby yielding a continuous approximation over the interval $[0,W_{\max}^{(n)}]$. When $W_{\max}^{(n)}$ is close to $1$, that is the fourth component of the error $\eps_n(\mathbf{W})$ is small, this discrete allocation yields a coupling to a continuous uniform random variable on $[0,1]$, which via Stein's method yields the Dickman limit. Finally we observe that the third component in $\eps_n(\mathbf{W})$ involving $d_{TV}(X_k^*, X_k+1 \mid \mathbf{W})$ forces, conditionally given $\mathbf{W}$, the random variables $X_k$ to be close to a Poisson distribution.

\medskip

An important special case of Theorem \ref{thm:randrates_cond} arises when the random weight array $\mathbf{W}$ is independent of the sequence $(X_k)_{k \in \N}$. In this setting, the conditional expectations and conditional distributions of $X_k$ reduce to their deterministic marginals. We denote $\lambda_k := \E[X_k]$ and
$$
x_{\min}^{(k)} := \min \{ x \in \mathbb{N}_0 : \P(X_k = x) > 0 \}.
$$
Further, denote
\begin{equation}\label{eq:eps_n_indep}
	\eps_n := \delta_n + \frac{2}{\theta} D_n^* + \frac{2}{\theta}\sum_{k=1}^n \lambda_k W_{n,k} \, \frac{d_{TV}(X_k^*, X_k+1)}{\P(X_k =x_{\min}^{(k)})} + |W_{\max}^{(n)} - 1|.
\end{equation}

This simplification immediately yields the following corollary.

\begin{cor}\label{cor:indep_weights}
	Let $S_n = \sum_{k=1}^n W_{n,k} X_k$ satisfy Assumption \ref{a2}. Assume further that $\mathbf{W}$ is independent of the sequence $(X_k)_{k \in \N}$, and that $(X_k)_{k \in \N}$ consists of mutually independent, non-negative integer-valued random variables. Then, for the constants $C_1, C_2$ defined in \eqref{eq:C_constants},
	$$
	d_K(S_n, D_\theta) \le \begin{cases} 
		C_1 \E[\eps_n], & \text{if } \theta \ge 1, \\[6pt] 
		2\theta^2 \E[\eps_n] + C_2 \left(\E[\eps_n]\right)^\theta, & \text{if } \theta \in (0,1). 
	\end{cases}
	$$
\end{cor}

A different possible scenario is when the weights in \eqref{eq:S_n} are not naturally almost surely ordered as in assumption \ref{a2}, so that Theorem \ref{thm:randrates_cond} cannot be directly applied. However, in this case, conditioning on the weights one can artificially order the weights and their corresponding random counts and then apply Theorem \ref{thm:randrates_cond}.

For any $n \in \N$, let $0 \le W_{n,(1)} \le W_{n,(2)} \le \dots \le W_{n,(n)} =: W_{\max}^{(n)}$ denote the order statistics of the weights $(W_{n,1}, \dots, W_{n,n})$, and by convention set $W_{n,(0)} := 0$. Let $\sigma$ be a permutation of $[n]$ (measurable with respect to $\mathbf{W}$) such that $W_{n,\sigma(k)} = W_{n,(k)}$, breaking any ties arbitrarily. We define the ordered increments and the corresponding sum of discrepancies as
\begin{equation*}
	\Delta_{n,k}^{\text{os}}  := \theta (W_{n,(k)} - W_{n,(k-1)}), \quad \text{and} \quad D_{n,k}^{\text{os}}  := \sum_{j=1}^k (\lambda_{\sigma(j)}(\mathbf{W}) W_{n,(j)} - \Delta_{n,j}^{\text{os}}),
\end{equation*}
yielding the order-statistic equivalents of the maximal gap and maximal discrepancy in \eqref{def:delD} as
\begin{equation*}
	\delta_n^{\text{os}} = \max_{k \in [n]} (W_{n,(k)} - W_{n,(k-1)}), \quad \text{and} \quad D_n^{*,{\text{os}} } = \max_{k \in [n]} |D_{n,k}^{\text{os}}|.
\end{equation*}
Finally, we define the aggregate error term $\eps_n^{\text{os}} (\mathbf{W})$ exactly as in \eqref{eq:eps_n_random}, but replacing the gap and discrepancy terms with their ordered equivalents as
\begin{equation}\label{eq:eps_n_unordered}
	\eps_n^{\text{os}}(\mathbf{W}) := \delta_n^{\text{os}} + \frac{2}{\theta} D_n^{*,{\text{os}}} + \frac{2}{\theta}\sum_{k=1}^n \lambda_k(\mathbf{W}) W_{n,k} \, \frac{d_{TV}(X_k^*, X_k+1 \mid \mathbf{W})}{\P(X_k = x_{\min}^{(k)}(\mathbf{W}) \mid \mathbf{W})} + |W_{\max}^{(n)} - 1|.
\end{equation}

\begin{theorem}\label{thm:randrates_unordered}
	Let $\mathbf{W} = (W_{n,k})_{n \in \N, k \in [n]}$ be an array of non-negative integrable random variables, and suppose $(X_k)_{k \in \N}$ satisfies assumption \ref{a1}. Let $S_n = \sum_{k=1}^n W_{n,k} X_k$, and let $C_1, C_2$ be the constants defined in \eqref{eq:C_constants}. Then,
	$$
	d_K(S_n, D_\theta) \le \begin{cases} 
		C_1 \E[\eps_n^{\text{os}} (\mathbf{W})], & \text{if } \theta \ge 1, \\[6pt] 
		2\theta^2 \E[\eps_n^{\text{os}} (\mathbf{W})] + C_2 \left(\E[\eps_n^{\text{os}} (\mathbf{W})]\right)^\theta, & \text{if } \theta \in (0,1). 
	\end{cases}
	$$
\end{theorem}

The proof of Theorem \ref{thm:randrates_unordered} is presented in Section \ref{sec:1proofs}. In most cases where Dickman convergence arises, there is usually a natural ordering of the weights present. However, in some scenarios, such as our application in Section \ref{sec:app_poisson}, the weights may only be ordered by their expectations. This lack of almost-sure ordering can significantly deteriorate the rate of convergence in the Kolmogorov distance due to a jump discontinuity in the derivative of the solution to the Stein equation. Nonetheless, Theorem \ref{thm:randrates_unordered} recovers good rates of convergence in the Kolmogorov distance under reasonably mild assumptions. For a concrete example and further discussion, see Section \ref{sec:app_poisson}. 

The lack of ordering of weights in assumption \ref{a2} however poses no problem when deploying Stein's method for Dickman approximation in the smooth $d_{1,1}$ distance. The $d_{1,1}$ distance between two random variables $X$ and $Y$ is defined as
		$$
		d_{1,1}(X, Y) := \sup_{h \in \mathcal{H}_{1,1}} |\E[h(X)] - \E[h(Y)]|,
		$$
		where $\mathcal{H}_{\alpha,\beta} := \{ h : \R \to \R \mid h \in \operatorname{Lip}_\alpha,h' \in \operatorname{Lip}_\beta \}$ for $\alpha,\beta>0$, with $\operatorname{Lip}_\alpha$ denoting the class of all Lipschitz functions with Lipschitz constant at most $\alpha$. One fundamental advantage of employing the smoother $d_{1,1}$ distance is that when applying Stein's method, the solutions to the corresponding Stein equation possess bounded second derivatives. This allows one to completely drop the ordering requirement on the weights formulated in assumption \ref{a2}.

		For a fixed $n \in \N$, let $0 = t_0 \le t_1 \le \dots \le t_n$ be any non-decreasing grid (possibly dependent on $\mathbf{W}$, but deterministic conditionally given $\mathbf{W}$). We denote the grid increments by $\Delta t_k := t_k - t_{k-1}$ for $k \in [n]$, and define the corresponding grid discrepancy sum as $D_{n,k}^{(t)} := \sum_{j=1}^k (\lambda_j(\mathbf{W}) W_{n,j} - \theta \Delta t_j)$, with the maximal grid gap and maximal grid discrepancy denoted by
		$$
		\delta_n^{(t)} := \max_{k \in [n]} \Delta t_k, \quad \text{and} \quad D_n^{(t)*} := \max_{k \in [n]} |D_{n,k}^{(t)}|.
		$$
		Let $d_1(X,Y) := \inf \E|X-Y|$ denote the Wasserstein distance, where the infimum is taken over all joint couplings of $X$ and $Y$. We define the aggregate error for the $d_{1,1}$ distance as
		\begin{align}\label{eq:eps_n_d11}
			\eps_n^{1,1}(\mathbf{W}) := \frac{\theta t_n}{4} \delta_n^{(t)} &+ \left(1 + \frac{t_n}{2}\right) D_n^{(t)*} + \theta |t_n - 1|  \nonumber \\
			&+ \frac{1}{2} \sum_{k=1}^n \lambda_k(\mathbf{W}) W_{n,k} \Big( |W_{n,k} - t_k| + W_{n,k} \, d_1(X_k^*, X_k+1 \mid \mathbf{W}) \Big).
		\end{align}

\begin{theorem}\label{thm:d11_bound}
	Let $S_n = \sum_{k=1}^n W_{n,k} X_k$ satisfy assumption \ref{a1}, and let $(t_k)_{k=0}^n$ be any deterministic grid partitioning $[0,1]$ as defined above. Then,
	$$
	d_{1,1}(S_n, D_\theta) \le \E[\eps_n^{1,1}(\mathbf{W})].
	$$
\end{theorem}

\begin{rem}[Comparison of $d_K$ and $d_{1,1}$ Bounds]
	\label{rem:metric_comparison}
	The $d_{1,1}$ bound in Theorem \ref{thm:d11_bound} differs fundamentally from the Kolmogorov bound in Theorem \ref{thm:randrates_cond} in two main aspects. First,
		due to a jump discontinuities in the derivative of the solutions to the Kolmogorov Stein equation, the proof of Theorem \ref{thm:randrates_cond} relies on a non-increasing functional decomposition of the derivative. In contrast, for $d_{1,1}$, the test functions $h \in \mathcal{H}_{1,1}$ yield a solution with bounded second derivatives. This permits a direct Taylor expansion approach over any conditionally deterministic grid $t_k$, allowing the framework to handle unordered weights without invoking order statistics. The presence of the arbitrary grid $t_k$ in the $d_{1,1}$ bound provides a natural flexibility: if the weights are ordered only by their expectations, we can take $t_k = \E[W_{n,k}]$, while if they are almost surely ordered, we can simply take $t_k = W_{n,k}$.
		
		Secondly, since for $\theta \in (0,1)$, the Dickman density $p_\theta(x)$ has a pole at the origin, bounding the Kolmogorov distance near the origin forces the convergence rate to slow to $(\eps_n(\mathbf{W}))^\theta$. On the other hand, because $d_{1,1}$ test functions are smooth, the $d_{1,1}$ error bounds remain linear in $\eps_n^{1,1}(\mathbf{W})$ across all $\theta > 0$.
	\end{rem}

We note that asymptotic convergence to generalized Dickman limits for weighted Bernoulli sums with independent, unordered random weights was established in \cite{Pi16b} relying fundamentally on Laplace transforms, yielding only weak convergence under very rigid condition on the expectations of the weights and the Bernoulli random variables. By successfully deploying Theorems \ref{thm:randrates_unordered} and \ref{thm:d11_bound}, we advance this domain much further by providing explicit, non-asymptotic error bounds in the Kolmogorov and Wasserstein-2 distance under a general setting with dependence.
Our results also generalize the bounds in \cite{BS25} in several directions. First, our bounds accommodate random weights. Moreover, we allow the weights to be dependent on each other. Finally, we can also effectively handle certain kinds of dependencies between the weights and the counts, as long as the counts are conditionally independent given the weights.
\medskip

The remainder of the paper is organized as follows. In Section \ref{sec:app}, we demonstrate the versatility of our main results through four distinct applications, each highlighting a different structural component of our framework. We begin in Section \ref{sec:harmonic} with a deterministic weight setting, analyzing sums of log-primes under a harmonic distribution over certain subsets of the natural numbers. Section \ref{sec:orderstat} then explores a model featuring dependent weights generated by independent increments. In Section \ref{sec:CUE}, we turn to a random matrix setting, examining sums weighted by the spectra of the Circular Unitary Ensemble (CUE), where the counts $X_k$ are only conditionally independent given the weights. Our final application, presented in Section \ref{sec:app_poisson}, considers randomly weighted Poisson sums in which all the random variables are mutually independent. In Section \ref{sec:StDickman}, we recall some preliminary results on Stein's method for Dickman distribution. The proofs of our main abstract results are established in Section \ref{sec:1proofs}, followed by the proofs of the four applications in Section \ref{sec:appproofs}. Finally, the proofs of several ancillary results related to the harmonic measure and CUE models are deferred to Appendices \ref{app:lemlbpf} and \ref{app:cuelb}.

    %
%
	\section{Applications}\label{sec:app}
We analyze four main applications below, demonstrating the flexibility of our main results. 

\subsection{Harmonic measures}\label{sec:harmonic}
We first apply our framework to a classical model in probabilistic number theory, namely, the distribution of prime factors under certain weighted harmonic measures. Let $(p_k)_{k\ge1}$ be an enumeration of the prime numbers in increasing order. For a fixed $n \in \mathbb{N}$, we define  $\Omega_n$ as the set of all integers whose prime factors are drawn exclusively from the first $n$ primes, that is,
$$
\Omega_n := \{ N \in \mathbb{N} : \text{for $p$ prime, }\, p \mid N \implies p \le p_n \}.
$$
Every integer $N \in \Omega_n$ possesses a unique prime factorization $N = \prod_{k=1}^n p_k^{X_k}$, where $X_k \in \mathbb{N}_0$. We also consider a natural variation of this model, namely, when we restrict our sample space to only square-free integers. Let $\Omega_n^{\text{sq}}$ denote the subset of $\Omega_n$ given by
$$
\Omega_n^{sq} = \{N \in \mathbb{N} : \text{for $p$ prime, }\, p \mid N \implies p \le p_n \text{ and } p^2 \nmid N\}.
$$
Any $N \in \Omega_n^{\text{sq}}$ can be expressed as $N = \prod_{k=1}^n p_k^{X_k^{\text{sq}}}$, with $X_k^{\text{sq}} \in \{0, 1\}$. 

It is well known that if one samples a number according to the \textit{harmonic measure} on the sets $\Omega_n$ or $\Omega_n^{\text{sq}}$, its logarithm, suitably scaled converges to the standard Dickman distribution (see e.g.\ \cite{BG19}). We consider a generalization of this measure, namely, a harmonic measure weighted by the so-called generalized divisor function, which naturally leads to generalized Dickman limits. 
For $\theta > 0$, the generalized divisor function $\tau_\theta : \mathbb{N} \to \mathbb{R}$ is defined as (see \cite[Eq.\ II.5.3]{tet}) the multiplicative arithmetic function (i.e., $\tau_\theta(ab) = \tau_\theta(a)\tau_\theta(b)$ when $\gcd(a,b) = 1$) satisfying
$$
\tau_\theta(p^j) = \binom{j+\theta-1}{j} = \frac{\theta(\theta+1)\cdots(\theta+j-1)}{j!}
$$
for any prime $p$ and integer $j \ge 0$. For integer values of $\theta = k \in \N$, the function $\tau_k(N)$ counts the number of ways $N$ can be expressed as an ordered product of $k$ integers \cite[Eq.\ II.5.4]{tet}. We endow $\Omega_n$ with a probability measure $\mathbb{P}_n$ given by
\begin{equation}\label{eq:harmonic_measure}
	\mathbb{P}_n(\{N\}) = \frac{1}{Z_n} \frac{\tau_\theta(N)}{N}, \quad N \in \Omega_n,
\end{equation}
where $Z_n = \sum_{N \in \Omega_n} \tau_\theta(N)/N \in (0,\infty)$ is the normalizing partition function. Similarly, we define the restricted harmonic measure $\mathbb{P}_n^{\text{sq}}$ on $\Omega_n^{\text{sq}}$ identically to \eqref{eq:harmonic_measure} as
\begin{equation}\label{eq:rest.harmonic_measure}
	\mathbb{P}_n^{\text{sq}}(\{N\}) = \frac{1}{Z_n^{\text{sq}}} \frac{\tau_\theta(N)}{N} = \frac{1}{Z_n^{\text{sq}}} \frac{\theta^{\omega(N)}}{N} , \quad N \in \Omega_n^{\text{sq}},
\end{equation}
where we utilize the partition function $Z_n^{\text{sq}} = \sum_{N \in \Omega_n^{\text{sq}}} \tau_\theta(N)/N \in (0,\infty)$. Here $\omega_n$ denotes the number of distinct prime divisiors of $N$ and the equality $\tau_\theta(N) = \theta^{\omega(N)}$ follows by noting that $\tau_\theta$ is multiplicative with $\tau_\theta(p) = \theta$. Thus, in both models, the likelihood of choosing an integer $N$ is proportional to its generalized divisor function and inversely proportional to its size. In particular, when $\theta=1$, since $\tau_\theta \equiv 1$, both the measures $\mathbb{P}_n(\{N\}), \mathbb{P}_n^{\text{sq}}(N) \propto 1/N$ are purely harmonic.

The following result, proved in Section \ref{sec:proofharmonic}, utilizes Corollary \ref{cor:indep_weights} and Theorem \ref{thm:d11_bound} to provide an optimal rate of convergence to the Dickman limit in both the Kolmogorov and $d_{1,1}$ distances for a transformation of a random integer sampled under $\mathbb{P}_n$ or $\mathbb{P}_n^{\text{sq}}$.

\begin{theorem}\label{thm:harmonic}
	Let $M_n, M_n^{\text{sq}}$ be random integers sampled according to the probability measures $\mathbb{P}_n$ and $\mathbb{P}_n^{\text{sq}}$, respectively. Then for any $\theta > 0$, there exists constants $0<K_1 \le K_2 <\infty$ depending only on $\theta$ such that 
	$$
	\frac{K_1}{(\log n)^{\min\{1,\theta\}}} \le d_K\left(\frac{\log M_n}{\log p_n}, D_\theta \right), \; d_K\left(\frac{\log M_n^{\text{sq}}}{\log p_n}, D_\theta\right) \le  \frac{K_2}{(\log n)^{\min\{1,\theta\}}}.
	$$
	Moreover, there exists constants $0<K_3 \le K_4 <\infty$ depending only on $\theta$ such that
		$$
		\frac{K_3}{\log n} \le d_{1,1}\left(\frac{\log M_n}{\log p_n}, D_\theta \right), \; d_{1,1}\left(\frac{\log M_n^{\text{sq}}}{\log p_n}, D_\theta\right) \le  \frac{K_4}{\log n}.
		$$
\end{theorem}
We remark here that the $d_{1,1}$ bounds in the specific case $\theta=1$ were derived in \cite[Theorems 1.3 and 1.4]{BG19}. Our results extend this to any $\theta>0$. The Kolmogorov bounds are new.

\subsection{Independent increment weights}\label{sec:orderstat} 

We now consider an application of our results where the weights $\mathbf{W}$ are not only random but also dependent. Let $(Y_i)_{i \in \N}$ be a sequence of independent, non-negative random variables. We assume $\E[Y_i] = 1$ and $\Var(Y_i) = \sigma_i^2 < \infty$ for all $i \in \N$. Denote $S_n^Y = \sum_{i=1}^n Y_i$, $n \in \N$, and define the random weights as
\begin{equation} \label{eq:unified_weights}
	W_{n,0} = 0, \quad \text{and} \quad W_{n,k} = \frac{1}{Z_n} \sum_{i=1}^k Y_i, \quad k \in [n],
\end{equation}
where $Z_n$ is a normalizing constant. We consider two canonical choices for $Z_n$:
\begin{itemize}
	\item \textbf{Model A:} $Z_n = n$. Thus $W_{\max}^{(n)} = W_{n,n} = S_n^Y/n$, and $ \E \, W_{\max}^{(n)} = 1$.
	\item \textbf{Model B:} $Z_n = S_n^Y$. This enforces $W_{\max}^{(n)} = W_{n,n} = 1$.
\end{itemize}

We study the generalized sum $S_n = \sum_{k=1}^n W_{n,k} X_k$, where $X_k \sim {\rm Poi}(\theta/k)$ are independent of each other and of the sequence $(Y_i)_{i \in \N}$. The specific choice of Poisson counts is made primarily for simplicity and exposition; our framework readily accommodates any distribution on the non-negative integers which provides sufficiently good control on $d_{TV}(X_k^*, X_k+1)$ or $d_{1}(X_k^*, X_k+1)$. 

A natural example for Model A, which was one of the main subjects of study in \cite{BS25}, is when $Y_k \equiv 1$ for all $k \in [n]$ so that $W_{n,k} = k/n$, $k \in [n]$. Such sums appear in many combinatorial examples, for instance, sums of positions of records in a random permutation \cite{ABT2000}. On the other hand, Model B is well-suited for simulating randomly deployed networks. Consider a base station (receiver) at the origin with $n-1$ sensors (transmitters) deployed uniformly at random along the interval $[0,1]$. If the increments in Model B are $Y_i \sim \text{Exp}(1)$ i.i.d., Rényi's representation implies the weights $W_{n,k} = \frac{1}{S_n^Y} \sum_{i=1}^k Y_i$ are distributed as the order statistics $U_{(k)}$ of these uniform sensor locations. 
To minimize total transmission energy---which scales proportionally with distance---an optimal routing protocol greedily assigns the heaviest data-transmission tasks to the nodes physically closest to the base station. Real-world network traffic is typically skewed; under standard Zipfian traffic profiles, the transmission volume decays harmonically.  Consequently, one can assume that the $k$-th closest node, located at distance $W_{n,k}$, is the $k$-th most active node and generates $X_k \sim \text{Poi}(\theta/k)$ data packets over a unit time window. The sum $S_n = \sum_{k=1}^n W_{n,k} X_k$ thus precisely captures the minimal total transmission cost.

The following result provides upper bounds on the Kolmogorov and $d_{1,1}$ distances between $S_n$ and $D_\theta$.

\begin{theorem}\label{thm:unified_weights}
	Let $S_n$ be as above with the normalizing constant $Z_n=n$ (Model A) or $Z_n = S_n^Y$ (Model B). Then there exists a constant $C \in (0,\infty)$ depending only on $\theta$ such that
	$$
	d_K(S_n, D_\theta) \le C \Bigg( \frac{\log n}{n^2}\sum_{i=1}^n \sigma_i^2 + \frac{1}{n}\Big(1 + \sum_{i=1}^n \sigma_i^2\Big)^{1/2} + \frac{1}{n} \sum_{j=1}^n \frac{1}{j} \Big(\sum_{i=1}^j \sigma_i^2\Big)^{1/2} \Bigg)^{\min\{1,\theta\}}.
	$$
	Moreover, there exists a constant $C' \in (0,\infty)$ depending only on $\theta$ such that
		$$
		d_{1,1}(S_n, D_\theta) \le C' \left(1 + \frac{1}{n}\Big(\sum_{i=1}^n \sigma_i^2\Big)^{1/2}\right) \Bigg( \frac{\log n}{n^2}\sum_{i=1}^n \sigma_i^2 + \frac{1}{n}\Big(1 + \sum_{i=1}^n \sigma_i^2\Big)^{1/2} + \frac{1}{n} \sum_{j=1}^n \frac{1}{j} \Big(\sum_{i=1}^j \sigma_i^2\Big)^{1/2} \Bigg).
		$$
\end{theorem}

The proof of Theorem \ref{thm:unified_weights} can be found in Section \ref{sec:proofweights}. In the classical framework for logarithmic combinatorial structures governed by the Ewens Sampling Formula with scale parameter $\theta > 0$, the cycle counts $(C_1, \dots, C_n)$ of a random permutation of $n$ elements are distributed exactly as independent Poisson random variables $X_k \sim \text{Poi}(\theta/k)$ for $k \in [n]$, conditioned on $\sum_{k=1}^n k X_k = n$ (see \cite{ABT}). This motivates the study of the unconditioned sum $S_n = \sum_{k=1}^n \frac{k}{n} X_k$. In the specific case of Model A where $Y_i \equiv 1$ for all $i \in [n]$, the weights reduce deterministically to $W_{n,k} = k/n$. Consequently, our generalized framework exactly recovers this classical sum by choosing $X_k \sim \text{Poi}(\theta/k)$.

It was shown in \cite[Theorem 1.6]{BS25} and \cite[Theorem 1.2]{BG19} that in this case, the rate of Dickman convergence in the Kolmogorov and $d_{1,1}$ distances are $n^{-\min\{1,\theta\}}$ (this rate is also optimal, see \cite[Theorem 1.6]{BS25}) and $n^{-1}$, respectively. The following corollary, which is immediate from Theorem \ref{thm:unified_weights} upon taking $\sigma_i^2=0$ for all $i \in [n]$, reproduces such a bound. The proof of the additional optimality result for the $d_{1,1}$ bound can be found in Appendix \ref{app:C}. One can also obtain a similar result from Theorem \ref{thm:unified_weights} in the case when $X_k$'s are chosen to be Bernoulli distributed as in \cite[Theorem 1.1]{BS25}.

\begin{cor}\label{lem:beta_shift}
	For $\theta>0$, let $S_n = \sum_{k=1}^n \frac{k}{n} X_k$, where $X_k \sim {\rm Poi}\left(\frac{\theta}{k}\right)$, $k \in [n]$ are independent. Then, there exists constants $C, C' \in (0,\infty)$ depending only on $\theta$ such that
	$$
	C n^{-\min\{1, \theta\}}  \le d_K(S_n, D_\theta) \le C'  n^{-\min\{1, \theta\}} \quad \text{and} \quad C  n^{-1} \le d_{1,1} (S_n, D_\theta) \le C'  n^{-1}.
	$$
\end{cor}

Another special case of Theorem \ref{thm:unified_weights} is when $\max_{i \in [n]} \sigma_{i}^2 \le \sigma_{\max}^2 <\infty$ with $\sigma_{\max}^2$ independent of $n$. In this case, the bounds in Theorem \ref{thm:unified_weights} simplify yielding the following corollary. 

\begin{cor}\label{cor:unified_weights}
	Let $S_n$ be as in Theorem \ref{thm:unified_weights} with the normalizing constant $Z_n=n$ (Model A) or $Z_n = S_n^Y$ (Model B). Additionally assume that $0 \le \Var(Y_i) \le \sigma_{\max}^2 < \infty$ for $i \in [n]$. Then there exists a constant $C \in (0,\infty)$ depending only on $\theta$ such that
	$$
	d_K(S_n, D_\theta) \le C \left(\frac{1}{\sqrt{n}}\right)^{\min\{1,\theta\}}.
	$$
	Furthermore, there exists a constant $C'  \in (0, \infty)$ depending only on $\theta$ such that
	$$
		d_{1,1}(S_n, D_\theta) \le \begin{cases}
		\frac{C'}{n}, &\text{if $Z_n = n$\quad (Model A),}\\
		\frac{C' \log n}{n} &\text{if $Z_n = S_n^Y$ (Model B)}.
	\end{cases}
	$$
	Moreover, if additionally $\max_{i \in [n]} \E [Y_i^{2+\epsilon}]<\infty$ for some $\epsilon>0$, then $d_{1,1}(S_n, D_\theta) = \mathcal{O}(1/n)$ in Model B.
\end{cor}

Note that taking $Y_i \equiv 1$ for all $i \in [n]$, the $d_{1,1}$ lower bound in Corollary \ref{lem:beta_shift} implies that the $d_{1,1}$ upper bounds obtained in Corollary \ref{cor:unified_weights} are in general not improvable. The proof of Corollary \ref{cor:unified_weights} is carried out on Appendix \ref{app:C}, the argument for the $d_{1,1}$ bound is similar to those in the proof of \cite[Theorem 1.7]{BS25}. Note that a direct application of Theorem \ref{thm:unified_weights} yields only a suboptimal $\mathcal{O}(1/\sqrt{n})$ bound in the $d_{1,1}$ distance. To obtain the sharper $\mathcal{O}(1/n)$ (or, $\mathcal{O}(\log^2 n/n)$) bound in Corollary \ref{cor:unified_weights}, we use the triangle inequality to bound $d_{1,1}(\sum_{k=1}^n (k/n)\,X_k, D_\theta)$ via Theorem \ref{thm:unified_weights} and bound $d_{1,1}(\sum_{k=1}^n (k/n)\,X_k, S_n)$ separately using Taylor's approximation.

\subsection{Weighted geometric sums on CUE spectra}\label{sec:CUE}
A key strength of the conditional framework in Theorems \ref{thm:randrates_cond} and \ref{thm:d11_bound} is its capacity to handle dependent random weights as well as cross-dependence between the weights and the counts, requiring only the conditional independence of the counts $X_k$. To illustrate this, we evaluate a conditionally geometric sum weighted by the highly rigid eigenphases of the Circular Unitary Ensemble (CUE).

A natural choice for dependent random weights are determinantal point processes (DPP). For $n \ge 1$, we consider an $n \times n$ random unitary matrix drawn from the Haar measure, which defines the Circular Unitary Ensemble (CUE). The complex eigenvalues $e^{i \theta_j}, j \in [n]$ of this matrix lie on the unit circle and are parameterized by their ordered eigenphases $0 \le \theta_1 < \theta_2 < \dots < \theta_n < 2\pi$ almost surely. We define our ordered random weights as the normalized gaps from the smallest eigenphase $\theta_1$, given by
\begin{equation}\label{eq:CUEW}
	W_{n,k} = \frac{1}{2\pi}(\theta_{k+1} - \theta_1), \quad k \in [n-1],
\end{equation}
such that $0 =:W_{n,0} < W_{n,1} < W_{n,2} < \dots < W_{n,n-1} < 1$. In contrast to independent random variables, the CUE eigenphases constitute a DPP \cite{Dyson} (also see \cite[Section 2.3]{BA}). Specifically, the resulting mutual repulsion manifests as \textit{eigenvalue rigidity}, ensuring that these gaps are tightly tethered to the uniform grid $k/n$, $k \in [n-1]$. The following result is an immediate consequence of \cite[Theorem 1]{MM19}. For completeness, its proof can be found in Section \ref{sec:CUEproofs}.

\begin{lemma}\label{lem:cue_rigidity}
	For $0 \le \theta_1 < \dots < \theta_n < 2\pi$ the ordered eigenphases of a CUE matrix, let $(W_{n,k})_{k \in [n-1]}$ be as in \eqref{eq:CUEW}. Then there exists a universal constant $C \in (0,\infty)$ such that
	$$
	\E\left[\max_{k \in [n-1]} |W_{n,k} - k/n|\right] \le C \; \frac{\log n}{n}, \quad n \ge 2.
	$$
\end{lemma}

Conditionally on $\mathbf{W} = (W_{n,k})_{k \in [n-1]}$, we define the sequence $(X_k)_{k \in [n-1]}$ as mutually conditionally independent geometric random variables with means
$$
\lambda_k(\mathbf{W}) := \E[X_k \mid \mathbf{W}] = \frac{\theta}{n W_{n,k}},
$$
where we recall that the probability mass function of a geometric random variable $X$ with success probability $p \in (0,1)$ and mean $(1-p)/p$ is given by $\P(X = k) = p(1-p)^k, k \in \N_0$. We study the sum $S_n = \sum_{k=1}^{n-1} W_{n,k} X_k$. Our motivation to choose the random variables this way comes from its connection to the free Bose gas; see Remark \ref{rem:BoseGas} below. An application of Theorem \ref{thm:randrates_cond} yields the following optimal Dickman approximation rates.

\begin{theorem}\label{thm:cue_geom}
	For $n \ge 2$, let $S_n = \sum_{k=1}^{n-1} W_{n,k} X_k$ be as above. Then there exists constants $K_1, K_2 \in (0,\infty)$ depending only on $\theta$ such that
	$$ 
	K_1 \left( \frac{1}{n} \right)^{\min\{1,\theta\}} \le d_K(S_n, D_\theta) \le K_2 \left( \frac{\log n}{n} \right)^{\min\{1,\theta\}}, \quad n \ge 2. 
	$$
	Moreover, there exists constants $K_3, K_4 \in (0,\infty)$ depending only on $\theta$ such that
		$$
		\frac{K_3}{n}  \le d_{1,1}(S_n, D_\theta) \le K_4 \frac{\log n}{n}, \quad n \ge 2.
		$$
\end{theorem}

The convergence rates are therefore optimal, up to the extraneous logarithmic factor. The proof of the result is contained in Section \ref{sec:CUEproofs}.

\begin{rem}[Connection to the Free Bose Gas]\label{rem:BoseGas} In a system of non-interacting indistinguishable particles (bosons), each of which can occupy one of a discrete set of $n$ quantum states, the total energy of the system is given by $E = \sum_{k=1}^n \epsilon_k X_k$, where $\epsilon_k$ is the energy level of state $k$ and $X_k$ is the corresponding number of occupying particles. For such systems, there exists a transition temperature below which a macroscopic fraction of the particles condenses into the ground state, a phenomenon known as Bose-Einstein condensation. In the grand canonical ensemble formulation of this model, the occupation numbers $(X_1, \dots, X_n)$ are modeled as mutually independent, geometrically distributed random variables (see e.g., \cite[Page 2]{ChaGCE}). In particular, one obtains the Bose-Einstein distribution, which dictates that the probability of observing a specific configuration $(x_1, \dots, x_n)$ is proportional to
	$$\exp \left(-\frac{1}{k_B T} \sum_{k=1}^{n} x_k \left(\epsilon_k-\mu\right)\right),
	$$
	where $T$ is the temperature, $k_B$ is Boltzmann's constant, and $\mu$ is the chemical potential, with $\mu < \epsilon_k$ for all $k \in [n]$. This yields the exact geometric means $(e^{ (\epsilon_k - \mu)/(k_B T)} -1)^{-1}$, motivating our choice of $X_k$. Furthermore, guided by the Bohigas-Giannoni-Schmit (BGS) conjecture for chaotic quantum systems lacking time-reversal symmetry, we model the energy of state $k$ relative to the chemical potential, $\epsilon_k - \mu$, using the normalized CUE spectral gaps $W_{n,k}$, a surrogate that crucially ensures the required inverse moments remain finite (see Lemma \ref{lem:cue_inverse}). Finally, rather than using the exact exponential means, we apply a linear approximation to the exponential denominator to define the means $\lambda_k(\mathbf{W}) = \theta/(n W_{n,k})$, in order to obtain a Dickman limit.
\end{rem}


\subsection{Independently weighted Poisson sums}\label{sec:app_poisson}

In our final application, we consider a model with unordered random weights. Let $(Y_k)_{k \in \N}$ be a sequence of independent, non-negative random variables with $\E[Y_k] = k$. Assume that for some $p \ge 1$, there exist constants $M_{k,p} \in (0,\infty)$ such that
\begin{equation}\label{eq:pmom}
	\E[|Y_k - k|^p] \le M_{k,p} \quad \text{for all $k \in \N$}.
\end{equation}
Define the weights $W_{n,k} = Y_k/n$, and for $\theta > 0$, let the counts $X_k \sim \text{Poi}(\theta/k)$ be mutually independent and independent of the sequence $(Y_k)_{k \in \N}$ to form the weighted sum $S_n = \sum_{k=1}^n W_{n,k} X_k$. In \cite{BG19, BS25} it was shown that for square-integrable variables $Y_k$, the smooth Wasserstein distance $d_{1,1}$ between $S_n$ and $\mathcal{D}_\theta$ converges to zero, provided the total variance $\sum_{k=1}^n \operatorname{Var}(Y_k)$ does not grow too rapidly. However, due to the jump discontinuity in the derivative of the solutions to the Kolmogorov Stein equation, these prior works were unable to handle this unordered example in the Kolmogorov distance. While the almost-sure ordering requirement of assumption \textbf{(A1)} precludes the direct application of Theorem \ref{thm:randrates_cond}, we deploy Theorem \ref{thm:randrates_unordered} to establish a non-asymptotic bound in the Kolmogorov distance.

Recall, we say $Y_k$, $k \in \N$ has a sub-Gaussian tail if there exists a universal constant $C \in (0,\infty)$ such that for all $k \in \N$ and $t \ge 0$,
\begin{equation}\label{eq:subgaussian_tail}
	\P(|Y_k - k| \ge t) \le 2 \exp\left( - \frac{t^2}{C k} \right).
\end{equation}

\begin{theorem}\label{thm:poisson_app}
	For the weighted sum $S_n = \sum_{k=1}^n W_{n,k} X_k$ as above, and the constants $C_1, C_2 \in (0,\infty)$ as in \eqref{eq:C_constants},
	\begin{equation}\label{eq:PoiFirst}
		d_K(S_n, D_\theta) \le \begin{cases} 
			C_1 \bar{\eps}_n, & \text{if } \theta \ge 1, \\ 
			2\theta^2 \bar{\eps}_n + C_2 \bar{\eps}_n^\theta, & \text{if } \theta \in (0,1),
		\end{cases}
	\end{equation}
	where
	$$
	\bar{\eps}_n = \frac{1}{n} + \frac{5 + 2\sum_{j=1}^n \frac{1}{j}}{n} \Bigg( \sum_{k=1}^n M_{k,p} \Bigg)^{1/p}.
	$$

If we further assume that $(Y_k)_{k \in \N}$ satisfies \eqref{eq:subgaussian_tail}, then there exists a constant $K \in (0,\infty)$, depending only on $C$ and $\theta$, such that
		$$
		d_K(S_n, D_\theta) \le K \left(\frac{(\log n)^{3/2}}{\sqrt{n}} \right)^{\min\{1,\theta\}}.
		$$
	\end{theorem}

For a proof of Theorem \ref{thm:poisson_app}, see Section \ref{sec:proofpoi}.
\begin{rem}[Comparison with Wasserstein-2 Bounds in \cite{BS25}]\label{rem:comparison}
	It is instructive to compare the bounds in Theorem \ref{thm:poisson_app} with analogous results derived in the Wasserstein-2 distance in \cite[Theorem 1.7]{BS25}. Assuming finite second moments of the weights $Y_k$ (i.e., $p=2$), that work established an upper bound of the order$$\frac{1}{n} + \frac{1}{n^2} \sum_{k=1}^n \frac{\sigma_k^2}{k}.$$
	Our abstract $d_{1,1}$ bound in Theorem \ref{thm:d11_bound} applied to this case matches this order by arguing similarly as in the proof of \cite[Theorem 1.7]{BS25} via the triangle inequality.
	For typical applications where the variance scales linearly, i.e., $\sigma_k^2 = \mathcal{O}(k)$, this bound yields an $\mathcal{O}(1/n)$ convergence rate. In contrast, our Kolmogorov bound in \eqref{eq:PoiFirst} with $p=2$ diverges if we merely assume $\sigma_k^2 = \mathcal{O}(k)$. Even under a stronger sub-Gaussian assumption, the second bound in Theorem \ref{thm:poisson_app} achieves a slower rate of $\mathcal{O}((n^{-1/2} (\log n)^{3/2})^{\min\{1,\theta\}})$.
	
	This discrepancy highlights a fundamental difference between the two distances, especially within the framework of Stein's method. The Wasserstein-2 distance employs twice-differentiable test functions, naturally accommodating Taylor series expansion arguments. However, because the Kolmogorov distance is governed by discontinuous indicator functions, we are forced to conditionally sort the independent weights and bound their absolute maximal deviation $\max_{k \in [n]} |Y_{(k)} - k|$ (see \eqref{eq:sorting_nonexpansive}). While this necessary workaround yields possibly suboptimal rates, the question of whether an optimal $\mathcal{O}(1/n)$ convergence rate is achievable in the Kolmogorov distance remains open.
\end{rem}

	%
%
\section{Proof of main results} The proof of our general abstract bounds in Section \ref{sec:mainres} depend on the application of Stein's method for Dickman approximation, developed in the works \cite{Gold,BG19,BS25}. We first recall in Section \ref{sec:StDickman} some of the results from \cite{BS25} that will be useful for us, and then in Section \ref{sec:1proofs} present the proofs of Theorems \ref{thm:randrates_cond} and \ref{thm:randrates_unordered}.

\subsection{Stein's method for Dickman distribution}\label{sec:StDickman}
Let $p_\theta$ denote the density of $D_\theta$. We will use the fact (see e.g.\ \cite[Section~1]{Pi16b}) that for $\theta>0$, there exists a constant $k_\theta>0$ such that 
$$
p_\theta(x) \le k_\theta \,\text{ for } \; x \ge 1,
$$ 
and in particular, we can take $k_1=e^{-\gamma}$, where $\gamma$ is the Euler-Mascheroni constant. For $\theta>0$ and $s \in \R_+$, denote by $R_\theta(s)$ the function
\begin{equation}\label{eq:R}
	R_\theta(s) = \begin{cases} \max \left\{\frac{e^{-\theta \gamma}}{\Gamma(\theta)},k_\theta\right\}, & \quad  \theta \ge 1,\\  s^{\theta-1}, & \quad \theta \in (0,1). \end{cases}
\end{equation}
The following result allows one to only bound the difference of the distribution functions for $a$ bounded away from $0$ to obtain a bound on the Kolmogorov distance.
\begin{lemma}[\cite{BS25}, Lemma 3.2]\label{lem:lb}
	For $\theta>0$ and a non-negative random variable $S$, assume that there exist $a_0,b\geq 0$ such that
	$$
	\sup_{a\geq a_0} |\Prob{S \leq a} - \Prob{D_\theta \leq a}|\leq b.
	$$
	Then
	$$
	d_K(S,D_\theta) \leq R_\theta(a_0)a_0 + b,
	$$
	where we let $R_\theta(0) \cdot 0 =0$ by convention.
\end{lemma}
To bound the discrepancy $\Prob{S \leq a} - \Prob{D_\theta \leq a}$, we utilize Stein's method. By \cite[Lemma 3.9]{Gold}, a non-negative random variable $S$ has the distribution $\mathcal{D}_\theta$ if and only if 
$$
\E \left[ \frac{S}{\theta}f'(S) + f(S) - f(S+1) \right] = 0
$$ 
for all Lipschitz functions $f$. For a fixed $a \ge 0$, in order to bound $|\Prob{S \leq a} - \Prob{D_\theta \leq a}|$, we thus consider the test function $h_a(x) = \ind(x \le a)$, $x \ge 0$ and the corresponding Stein equation
\begin{equation} \label{eq:stein}
	(x/\theta) f'(x) + f(x) - f(x+1) = \ind(x \le a) - \P(D_\theta \le a), \quad x \ge 0.
\end{equation}

By \cite[Theorem 1.10]{BS25}, there exists a solution $f_a$ to the delay-differential equation \eqref{eq:stein} that can be decomposed into well-behaved functions. 

\begin{lemma}[\cite{BS25}, Theorem 1.10]\label{lem:cont}
	Let $\theta>0$. For $a \ge 0$, there exists a solution $f_a$ to \eqref{eq:stein} that can be expressed as a sum $f_a=f_1+f_2$, where 
	$$
	f_1 = \min\{1,a^\theta/x^\theta\} - \Prob{D_\theta \le a}, \quad \text{with} \quad f_1'(x) = -\theta a^\theta \ind(x>a)/x^{\theta+1}, \quad x \ge 0
	$$
	and $f_2$ is a differentiable function on $\R_+$ satisfying $f_2' = u_+ - u_-$ with $u_+, u_-$ being non-negative, non-increasing functions bounded by $\theta^2$. In particular,
	$$
	f_a'= (g_1 + u_+) - (g_2 + u_-),
	$$
	where $g_1(x) = \frac{\theta}{a} \ind(x \le a)$ and $g_2(x) = \frac{\theta}{a} \ind(x \le a) + \frac{\theta a^\theta}{x^{\theta+1}} \ind(x > a)$ are also non-increasing and non-negative functions on $\R_+$.
\end{lemma}
We remark that, by convention, we set $f_1'(a) = 0$ to ensure that $f_a$ satisfies the Stein equation \eqref{eq:stein} pointwise for every $x \ge 0$. 
We will also need an estimate of the expectation $\E [g_1(S)]$ and $\E [g_2(S)]$ for our proofs. The following lemma yields the necessary bounds.

\begin{lemma}[\cite{BS25}, Lemma 3.3]\label{lem:mom}
	Let $\theta>0$ and let $W$ be a non-negative random variable. \begin{enumerate}
		\item For $\alpha>0$ and $u \ge 0$,
		$$
		\Prob{S + u \leq \alpha} \leq \left[R_\theta(\alpha)\alpha+d_K(S,D_\theta) \right] \mathds{1}(\alpha \ge u).
		$$
		\item For $t>0$, $\alpha>0$ and $u \ge 0$,
		$$
		\E \left[\frac{\alpha^t}{(S+u)^{1+t}}\mathds{1}(S +u>\alpha)\right] \le \zeta(1+t)	\left( R_\theta (\max\{\alpha, u\}) + \frac{2d_K(S,D_\theta)}{\max\{\alpha, u\}}\right),
		$$
		where $\zeta(s)=\sum_{m=1}^\infty m^{-s}$ for $s >1$ is the Riemann zeta function. \end{enumerate}
\end{lemma}

Let $g_1$ and $g_2$ be as in Lemma \ref{lem:cont} for some $a>0$. Note that we always have that $g_1 \le g_2$. Also, for any non-negative $S$, taking $u=0$ and $\alpha=a > 0$ in Lemma \ref{lem:mom}, we obtain
\begin{equation}\label{eq:Eg1bd}
	\E [g_1(S)] \le  \theta R_\theta(a) + \frac{\theta d_K(S,D_\theta)}{a},
\end{equation}
and
\begin{equation}\label{eq:Eg2bd}
	\E [g_2(S)] \le \theta R_\theta(a) + \frac{\theta d_K(S,D_\theta)}{a} + \theta \zeta(1+\theta)	\left( R_\theta(a) + \frac{2d_K(S,D_\theta)}{a}\right).
\end{equation}

For $\theta>0$, to bound the $d_{1,1}(S,D_\theta)$ distance for a non-negative random variable $S$, one utilizes the generalized Dickman Stein equation
	\begin{equation}\label{eq:stein_eq_d11}
		\frac{x}{\theta}f_h'(x) + f_h(x) - f_h(x+1) = h(x) - \E[h(D_\theta)].
	\end{equation}
The following result from \cite{BG19} guarantees the existence of a smooth solution to \eqref{eq:stein_eq_d11} for $h \in \mathcal{H}_{1,1}$. Below, we denote by $\|\cdot\|_\infty$ the supremum norm.
\begin{theorem}[\cite{BG19}, Theorem 4.9]\label{thm:d11solbd}
	For any $h \in \mathcal{H}_{1,1}$, there exists a solution $f_h\in \mathcal{H}_{\theta,\theta/2}$ to \eqref{eq:stein_eq_d11} with $\|f_h'\|_\infty \le \theta$ and $\|f_h''\|_\infty \le \theta/2$.
\end{theorem}

\subsection{Pointwise approximation of the distribution function}\label{sec:1proofs} Our main result Theorem \ref{thm:randrates_cond} is a consequence of the following pointwise bound on the difference of the distribution functions of $S_n$ and $D_\theta$.
\begin{proposition}\label{prop:master}
	Let $S_n$ be as in \eqref{eq:S_n} satisfying  assumptions \ref{a2}--\ref{a1}. Then for any $a, \theta > 0$, $\eps_n (\mathbf{W})$ as in Theorem \ref{thm:randrates_cond} and $g_2(x) = \frac{\theta}{a} \ind(x \le a) + \frac{\theta a^\theta}{x^{\theta+1}} \ind(x > a)$ as in Lemma \ref{lem:cont},
	$$
	|\P(S_n \le a \mid \mathbf{W}) - \P(D_\theta \le a)| \le \eps_n (\mathbf{W}) \E \Big[ \Big( \theta^2 + g_2(S_n) \Big) \mid \mathbf{W}\Big].
	$$
\end{proposition}

To prove Proposition \ref{prop:master}, we first decompose the difference of distribution functions in four different parts in Theorem \ref{thm:main_bound}, and then estimate each of them individually in Section \ref{sec:pathwisebd}. Recall the definition of $(X_k^* | \mathbf{W})$ from \eqref{eq:sb}. By the maximal coupling theorem, given $\mathbf{W}$, there exists a coupling $(X_k, X_k^*)$ for which 
\begin{equation}\label{eq:optcoup}
	\P(X_k^* \neq X_k + 1 \mid \mathbf{W}) = d_{TV}(X_k^*, X_k + 1 \mid \mathbf{W}).
\end{equation}
Below, we always consider $(X_k, X_k^*)$ to have this optimal coupling conditionally. Let $S_n^{(k)} = S_n - W_{n,k}X_k$ and define
\begin{align}\label{eq:t14_rand}
	T_1(g \mid \mathbf{W}) &=\sum_{k=1}^n \lambda_k(\mathbf{W}) W_{n,k}  \E \left[ g(S_n^{(k)} + W_{n,k} X_k^*) - g(S_n + W_{n,k}) \;\middle|\; \mathbf{W} \right],\nonumber\\
	T_2(g \mid \mathbf{W}) &= \sum_{k=1}^n (\lambda_k(\mathbf{W}) W_{n,k} - \Delta_{n,k}) \E \left[ g(S_n + W_{n,k}) \;\middle|\; \mathbf{W} \right],\nonumber\\
	T_3(g \mid \mathbf{W}) &=\sum_{k=1}^n \Delta_{n,k}  \E \left[ \int_0^1 \Big[ g(S_n + W_{n,k-1} + (W_{n,k} - W_{n,k-1}) u) - g(S_n + W_{n,k}) \Big] \dd u \;\middle|\; \mathbf{W} \right],\nonumber\\
	&= \sum_{k=1}^n \Delta_{n,k}  \E \left[ g(S_n + {W}_{n,k-1} + ({W}_{n,k} - {W}_{n,k-1}) U) - g(S_n + {W}_{n,k})\;\middle|\; \mathbf{W} \right],\nonumber\\
	T_4(g \mid \mathbf{W}) &= \theta \E \left[ \int_{W_{\max}^{(n)}}^1 g(S_n + t) \dd t \;\middle|\; \mathbf{W} \right],
\end{align}
for any real-valued function $g$ for which these functions are well-defined.

\begin{theorem}\label{thm:main_bound}
	Let $S_n = \sum_{k=1}^n W_{n,k} X_k$ be defined as above satisfying Assumptions \ref{a1}--\ref{a2}. Then for any $a, \theta > 0$, letting $f_a$ be as in Lemma \ref{lem:cont}, we have
	$$
	\P(S_n \le a \mid \mathbf{W}) - \P(\mathcal{D}_\theta \le a) =  \frac{1}{\theta} \Big( T_1(f_a' \mid \mathbf{W}) + T_2(f_a' \mid \mathbf{W}) - T_3(f_a' \mid \mathbf{W})  - T_4(f_a' \mid \mathbf{W}) \Big).
	$$
\end{theorem}

\begin{proof}
	By \eqref{eq:stein}, we have
	\begin{equation} \label{eq:stein_eval}
		\P(S_n \le a \mid \mathbf{W}) - \P(D_\theta \le a) = \frac{1}{\theta} \left( \E[S_n f_a'(S_n) \mid \mathbf{W}] - \theta \int_0^1 \E[f_a'(S_n+t) \mid \mathbf{W}]\; \dd t \right).
	\end{equation}
	
	Noting the size-bias relation \eqref{eq:sb}, and conditioning on $(S_n^{(k)}, W_{n,k})$ independent of $X_k$, we have
	\begin{align}\label{eq:decomp_step3}
		\E[S_n f_a'(S_n) \mid \mathbf{W}] &=\sum_{k=1}^n \lambda_k(\mathbf{W}) \E \left[ {W}_{n,k} f_a'(S_n^{(k)} + {W}_{n,k} X_k^*) \;\middle|\; \mathbf{W} \right] \nonumber\\
		&= T_1(f_a' \mid \mathbf{W}) + \sum_{k=1}^n \lambda_k(\mathbf{W}) \E \left[ {W}_{n,k} f_a'(S_n + {W}_{n,k}) \;\middle|\; \mathbf{W} \right] \nonumber\\
		&= T_1(f_a' \mid \mathbf{W})  + T_2(f_a' \mid \mathbf{W}) + \sum_{k=1}^n  \Delta_{n,k} \E \left[  f_a'(S_n +{W}_{n,k}) \;\middle|\; \mathbf{W} \right].
	\end{align}
	
	For the integral term in \eqref{eq:stein_eval}, We can decompose it as
	\begin{align}\label{eq:decomp_integral}
		\theta \int_0^1 \E[f_a'(S_n+t) \mid \mathbf{W}]\; \dd t 
		&= T_4(f_a' \mid \mathbf{W}) + \theta \E \left[ \int_0^{W_{\max}^{(n)}} f_a'(S_n + t) \, \dd t \;\middle|\; \mathbf{W} \right].
	\end{align}
	Because the coordinates ${W}_{n,k}$ partition the interval $[0,W_{\max}^{(n)}]$, we can rewrite the second integral using a uniform random variable $U \sim \mathbb{U}[0,1]$ independent of all else as
	\begin{align*}
		\theta \E \left[ \int_0^{W_{\max}^{(n)}} f_a'(S_n + t) \, \dd t \;\middle|\; \mathbf{W} \right] & = \sum_{k=1}^n \Delta_{n,k} \E\left[ f_a'(S_n + {W}_{n,k-1} + ({W}_{n,k} - {W}_{n,k-1}) U) \;\middle|\; \mathbf{W} \right].
	\end{align*}
	Subtracting this from the discrete sum in \eqref{eq:decomp_step3} yields
	\begin{equation}\label{eq:decomp_T3}
		\E[S_n f_a'(S_n) \mid \mathbf{W}] - \theta \E \left[ \int_0^{W_{\max}^{(n)}} f_a'(S_n + t) \, \dd t \;\middle|\; \mathbf{W} \right]= T_1(f_a' \mid \mathbf{W})  + T_2(f_a' \mid \mathbf{W}) -T_3(f_a' \mid \mathbf{W}).
	\end{equation}
	Finally, Substituting \eqref{eq:decomp_integral} and \eqref{eq:decomp_T3} back into \eqref{eq:stein_eval} yields the result.
\end{proof}

\subsubsection{Proof of Proposition \ref{prop:master}}\label{sec:pathwisebd}
Because the derivative of solution to our Stein equation decompose into non-increasing, non-negative functions (Lemma \ref{lem:cont}), we first bound the operators $T_1$ -- $T_3$ directly utilizing pathwise monotonicity. Recall that $X_k$ is a non-negative integer-valued random variable and
$$x_{\min}^{(k)}(\mathbf{W}) := \min \{ x \in \mathbb{N}_0 : \P(X_k = x \mid \mathbf{W}) > 0 \}.$$
By definition, $\P(X_k = x_{\min}^{(k)}(\mathbf{W}) \mid \mathbf{W}) > 0$. For any $g: \R_+ \to \R_+$, we can bound
$$
g(S_n) = g(S_n^{(k)} + W_{n,k} X_k) \ge g(S_n^{(k)} + W_{n,k} x_{\min}^{(k)}(\mathbf{W})) \ind(X_k = x_{\min}^{(k)}(\mathbf{W})).
$$
Thus, taking a conditional expectation using the conditional independence of $S_n^{(k)}$ and $X_k$, we obtain 
\begin{equation}\label{eq:Snk_bound}
	\E[g(S_n) \mid \mathbf{W}] \ge \P(X_k =x_{\min}^{(k)}(\mathbf{W}) \mid \mathbf{W}) \E[g(S_n^{(k)} + W_{n,k} x_{\min}^{(k)}(\mathbf{W})) \mid \mathbf{W}].
\end{equation}
For economy of notation, we denote 
$$
\eps_{n,2}(\mathbf{W}):= \sum_{k=1}^n \lambda_k(\mathbf{W}) W_{n,k} \frac{d_{TV}(X_k^*, X_k + 1 \mid \mathbf{W})}{\P(X_k = x_{\min}^{(k)}(\mathbf{W}) \mid \mathbf{W})}.
$$

\begin{lemma}[Bounds for $T_1$ -- $T_3$]\label{lem:bound_TE}
	Let $g : \R \to \R_+$ be a non-negative, non-increasing function. Then,
	\begin{align*}
		 |T_1(g \mid \mathbf{W})| \le \eps_{n,2}(\mathbf{W}) \E[g(S_n) &\mid \mathbf{W}], \quad \qquad |T_2(g \mid \mathbf{W})| \le D_n^* \E[g(S_n) \mid \mathbf{W}], \\
		\text{and}	\qquad 0 \le T_3(g \mid &\mathbf{W}) \le \theta \delta_n \E[g(S_n) \mid \mathbf{W}].   
	\end{align*}
\end{lemma}

\begin{proof}We proceed by bounding each of the three operators separately.
\medskip

	\noindent \underline{\em Bound for $T_1$:}
	Using the triangle inequality, we can bound
	\begin{align}\label{eq:splitsum}
		|T_1(g \mid \mathbf{W})| & \le \sum_{k=1}^n \lambda_k(\mathbf{W}) {W}_{n,k}\, \E \left[ \left| g(S_n^{(k)} + {W}_{n,k} X_k^*) - g(S_n + {W}_{n,k}) \right| \;\Big|\; \mathbf{W} \right]\nonumber\\
		& = \sum_{k=1}^n \lambda_k(\mathbf{W}) {W}_{n,k} \,\E \left[ \Big| g(S_n^{(k)} + {W}_{n,k} X_k^*) - g(S_n^{(k)} + {W}_{n,k} X_k + {W}_{n,k}) \Big| \; \Big|\; \mathbf{W} \right].
	\end{align}
	Because $g$ is non-negative and non-increasing, for any $x, y \ge 0$, we have
	 \begin{equation}\label{eq:gbd}
	 		 |g(x) - g(y)| \le g(\min(x,y)).
	 \end{equation}
	 Since conditionally $X_k, X_k^* \ge x_{\min}^{(k)}(\mathbf{W})$ almost surely, and $({W}_{n,k})_{k \in [n]}$ are non-negative, every argument of $g$ in \eqref{eq:splitsum} is bounded below by $S_n^{(k)} + W_{n,k} x_{\min}^{(k)}(\mathbf{W})$. Moreover, the difference is non-zero only when $X_k^* \neq X_k + 1$ conditionally given $\mathbf{W}$. By our optimal coupling of $(X_k, X_k^*)$, we have from \eqref{eq:optcoup} that $\P(X_k^* \neq X_k + 1 \mid \mathbf{W}) = d_{TV}(X_k^*, X_k + 1 \mid \mathbf{W})$. Noting that $S_n^{(k)}$ is conditionally independent of $(X_k, X_k^*)$, we have from \eqref{eq:gbd} that
	\begin{multline*}
		\lambda_k(\mathbf{W}) {W}_{n,k}\,\E \left[ \Big| g(S_n^{(k)} + {W}_{n,k} X_k^*) - g(S_n^{(k)} + {W}_{n,k} X_k + {W}_{n,k}) \Big| \;\middle|\; \mathbf{W} \right] \\ \le \lambda_k(\mathbf{W}) {W}_{n,k} \, \E[ g(S_n^{(k)} + W_{n,k} x_{\min}^{(k)}(\mathbf{W})) \mid \mathbf{W}] \; d_{TV}(X_k^*, X_k + 1 \mid \mathbf{W}).
	\end{multline*}
	Substituting this bound into \eqref{eq:splitsum} along with using \eqref{eq:Snk_bound} completes the proof. \medskip
	
	\noindent \underline{\em Bound for $T_2$:}
	By \eqref{eq:t14_rand}, and noting that $\lambda_k(\mathbf{W}) W_{n,k} - \Delta_{n,k} = D_{n,k} - D_{n,k-1}$, with the convention $D_{n,0} = 0$, we apply summation by parts to obtain
	\begin{align*}
		T_2(g \mid \mathbf{W}) &= \sum_{k=1}^n (\lambda_k(\mathbf{W}) W_{n,k} - \Delta_{n,k}) \, \E \left[ g(S_n + {W}_{n,k}) \;\middle|\; \mathbf{W} \right]\\
		&= \sum_{k=1}^n (D_{n,k}-D_{n,k-1}) \, \E \left[ g(S_n + {W}_{n,k}) \;\middle|\; \mathbf{W} \right]\\
		& =  D_{n,n} \, \E\left[ g(S_n + {W}_{n,n}) \;\middle|\; \mathbf{W} \right] +  \sum_{k=1}^{n-1} D_{n,k} \, \E \left[ g(S_n + {W}_{n,k}) - g(S_n + {W}_{n,k+1}) \;\middle|\; \mathbf{W} \right].
	\end{align*}
	
	Because ${W}_{n,k}$ is a monotonically increasing sequence and $g$ is a non-increasing function, we have $g(S_n + {W}_{n,k}) - g(S_n + {W}_{n,k+1}) \ge 0$. Thus, bounding $D_{n,k}$ by its maximum absolute value $D_n^*$, we have
	\begin{align*}
		|T_2(g \mid \mathbf{W})| &\le D_n^* \left( \E\left[g(S_n + {W}_{n,n}) \;\middle|\; \mathbf{W}\right] + \sum_{k=1}^{n-1} \E\left[ \left( g(S_n + {W}_{n,k}) - g(S_n + {W}_{n,k+1}) \right) \;\middle|\; \mathbf{W}\right] \right)\\
		&= D_n^* \, \E\left[ g(S_n + {W}_{n,1}) \;\middle|\; \mathbf{W} \right] \le D_n^* \, \E[g(S_n) \mid \mathbf{W}],
	\end{align*}
	where the telescoping sum collapses to $g(S_n + {W}_{n,1})$, and the final step holds since ${W}_{n,1} \ge 0$ and $g$ is non-increasing. This completes the proof.
	\medskip
	
	\noindent \underline{\em Bound for $T_3$:}
	By \eqref{eq:t14_rand},
	$$
	T_3(g \mid \mathbf{W}) = \sum_{k=1}^n \Delta_{n,k} \, \E \left[ g(S_n + {W}_{n,k-1} + ({W}_{n,k} - {W}_{n,k-1}) U) - g(S_n + {W}_{n,k}) \;\middle|\; \mathbf{W} \right].
	$$
	Since $g$ is non-increasing, we have
	$$
	g(S_n + {W}_{n,k-1}) \ge g(S_n + {W}_{n,k-1} + ({W}_{n,k} - {W}_{n,k-1}) U) \ge g(S_n + {W}_{n,k}).
	$$
	Thus, recalling from \eqref{def:delD} that $\theta \delta_n = \max_{k \in [n]}  \Delta_{n,k}$, we obtain
	\begin{align*}
		0 &\le T_3(g \mid \mathbf{W}) \le \sum_{k=1}^n \Delta_{n,k} \, \E \left[ g(S_n + {W}_{n,k-1}) - g(S_n + {W}_{n,k}) \;\middle|\; \mathbf{W} \right]\\
		& \le \theta \delta_n \sum_{k=1}^n \, \E \left[ g(S_n + {W}_{n,k-1}) - g(S_n + {W}_{n,k}) \;\middle|\; \mathbf{W} \right]\\
		& = \theta \delta_n \, \E\left[ g(S_n) - g(S_n + W_{n,n}) \;\middle|\; \mathbf{W} \right].
	\end{align*}
	Finally, because $g$ is non-negative, dropping the term $-g(S_n + W_{n,n})$ yields the conclusion.
\end{proof}

We are now ready to prove Proposition \ref{prop:master} combining the individual bounds obtained in Lemma \ref{lem:bound_TE}.

\begin{proof}[Proof of Proposition \ref{prop:master}]
	For $a>0$, Lemma \ref{lem:cont} yields the decomposition $f_a' = (g_1 + u_+) - (g_2 + u_-)$. By the linearity of $T_1, \dots, T_4$, we apply the triangle inequality to bound the total error term by term, and apply Lemma \ref{lem:bound_TE} for each of these non-negative and non-increasing functions.  
	
	Recall that $\|u_{\pm}\|_\infty \le \theta^2$ and $g_1 \le g_2$. For $T_1$, we apply the triangle inequality and Lemma \ref{lem:bound_TE} to obtain
	\begin{align}\label{eq:T2_eval_rand}
		|T_1(f_a' \mid \mathbf{W})| &\le |T_1(u_+ \mid \mathbf{W})| + |T_1(g_1 \mid \mathbf{W})| + |T_1(u_- \mid \mathbf{W})| + |T_1(g_2 \mid \mathbf{W})| \nonumber \\
		&\le \eps_{n,2}(\mathbf{W}) \, \E \Big[ u_+(S_n) + g_1(S_n) + u_-(S_n) + g_2(S_n) \;\Big|\; \mathbf{W} \Big] \nonumber\\
		&\le 2 \eps_{n,2}(\mathbf{W}) \, \E \Big[ \theta^2 + g_2(S_n) \;\Big|\; \mathbf{W} \Big].
	\end{align}
	
	Estimating $T_2$ in exactly the same way yields
	\begin{align}\label{eq:T3_eval_rand}
		|T_2(f_a' \mid \mathbf{W})| &\le 2 D_n^* \E \Big[ \theta^2 + g_2(S_n) \;\Big|\; \mathbf{W} \Big].
	\end{align}
	
	For $T_3$, because $T_3 \ge 0$, using Lemma \ref{lem:bound_TE} we can bound
	\begin{align}\label{eq:T1_eval_rand}
		|T_3(f_a' \mid \mathbf{W})| &= |T_3(u_+ + g_1 \mid \mathbf{W}) - T_3(u_- + g_2 \mid \mathbf{W})| \nonumber\\
		&\le \max\{T_3(u_+ + g_1 \mid \mathbf{W}), T_3(u_- + g_2 \mid \mathbf{W})\} \nonumber \\
		&\le \theta \delta_n \, \E \Big[ \max\{u_+(S_n) + g_1(S_n), u_-(S_n) + g_2(S_n)\} \;\Big|\; \mathbf{W} \Big] \nonumber\\
		&\le \theta \delta_n \, \E \Big[ \theta^2 + g_2(S_n) \;\Big|\; \mathbf{W} \Big].
	\end{align}
	
	Finally for $T_4$, using the fact that $|f_a'(x)| \le \max\{g_1+ u_+, g_2+ u_-\} \le \theta^2 + g_2(x)$ and that $g_2$ is non-increasing, we obtain
	\begin{align}\label{eq:E_eval_rand}
		|T_4(f_a' \mid \mathbf{W})| &\le \theta \E \left[ \left| \int_{W_{\max}^{(n)}}^1 f_a'(S_n + t) \dd t \right| \;\middle|\; \mathbf{W} \right] \nonumber\\
		&\le \theta |W_{\max}^{(n)} - 1| \E \Big[ \sup_{t \ge 0} |f_a'(S_n + t)| \;\Big|\; \mathbf{W} \Big] \nonumber\\
		&\le \theta |W_{\max}^{(n)} - 1| \E \Big[ \theta^2 + g_2(S_n) \;\Big|\; \mathbf{W} \Big].
	\end{align}
	
	Combining \eqref{eq:T2_eval_rand}, \eqref{eq:T3_eval_rand}, \eqref{eq:T1_eval_rand} and \eqref{eq:E_eval_rand}  into Theorem \ref{thm:main_bound} yields the result.
\end{proof}

\subsection{Proofs of Theorems \ref{thm:randrates_cond} and \ref{thm:randrates_unordered}}
We use the pointwise bound in Proposition \ref{prop:master} to obtain Theorems \ref{thm:randrates_cond} and \ref{thm:randrates_unordered}. Since the function $g_2$ in Proposition \ref{prop:master} may explode for very small values of $a$, to obtain a bound on the conditional Kolmogorov distance $d_K(S_n, D_\theta \mid \mathbf{W})$, we will utilize Proposition \ref{prop:master} only for $a \ge \tau_n(\mathbf{W})$ using an appropriate $\mathbf{W}$-measurable cut-off $\tau_n(\mathbf{W})>0$, and then apply Lemma \ref{lem:lb} conditionally. A bound on the unconditional $d_K(S_n, D_\theta)$ is then obtained simply by taking an expectation.

\begin{proof}[Proof of Theorem \ref{thm:randrates_cond}]
	Using the bound for $\E[g_2(S_n) \mid \mathbf{W}]$ from \eqref{eq:Eg2bd} applied to the conditional measure, we obtain from Proposition \ref{prop:master} that for any $a>0$,
	\begin{equation}\label{eq:stein_eval_det}
		|\P(S_n \le a \mid \mathbf{W}) - \P(D_\theta \le a)| \le \eps_n(\mathbf{W}) \left(c_1(a) + c_\theta \frac{d_K(S_n, D_\theta \mid \mathbf{W})}{a}\right),
	\end{equation}
	where $c_\theta := \theta(1 + 2\zeta(1+\theta))$ and $c_1(a) := \theta^2 + \theta R_\theta(a)\big(1 + \zeta(1+\theta)\big)$.
	
	For $a \in [\tau_n(\mathbf{W}), \infty)$ with $\tau_n(\mathbf{W}) = 2c_\theta \eps_n(\mathbf{W})$, the coefficient of $d_K(S_n, D_\theta \mid \mathbf{W})$ in \eqref{eq:stein_eval_det} is smaller than or equal to $\frac{c_\theta \eps_n(\mathbf{W})}{2 c_\theta \eps_n(\mathbf{W})} = \frac{1}{2}$. Moreover, note from \eqref{eq:R} that for any $\theta>0$, the function $R_\theta(a)$, and therefore $c_1(a)$ is non-increasing in $a>0$. Thus, from \eqref{eq:stein_eval_det} it follows that
	\begin{align*}
		\sup_{a\geq \tau_n(\mathbf{W})} |\mathbb{P}(S_n\leq a \mid \mathbf{W}) - \mathbb{P}(D_\theta\leq a)| & \leq \eps_n(\mathbf{W}) c_1(\tau_n(\mathbf{W}))+ \frac{1}{2} d_K(S_n,D_\theta \mid \mathbf{W}).
	\end{align*}
	Now applying Lemma \ref{lem:lb} conditionally, we obtain
	\begin{align*}
		d_K(S_n, D_\theta \mid \mathbf{W}) &\le \tau_n(\mathbf{W}) R_\theta(\tau_n(\mathbf{W})) +  \eps_n(\mathbf{W}) c_1(\tau_n(\mathbf{W}))+ \frac{1}{2} d_K(S_n,D_\theta \mid \mathbf{W}) ,
	\end{align*}
	whence
	$$
	d_K(S_n, D_\theta \mid \mathbf{W}) \le 2 \tau_n(\mathbf{W}) R_\theta(\tau_n(\mathbf{W})) + 2\eps_n(\mathbf{W}) c_1(\tau_n(\mathbf{W})).
	$$
	Substituting $\tau_n(\mathbf{W}) = 2c_\theta \eps_n(\mathbf{W})$, we have
	\begin{align*}
		&2 \tau_n(\mathbf{W}) R_\theta(\tau_n(\mathbf{W})) + 2 \eps_n(\mathbf{W}) c_1(\tau_n(\mathbf{W})) \\
		&= 4c_\theta \eps_n(\mathbf{W}) R_\theta(2c_\theta \eps_n(\mathbf{W})) + 2\eps_n(\mathbf{W}) \Big(\theta^2 + \theta R_\theta(2c_\theta \eps_n(\mathbf{W}))(1+\zeta(1+\theta))\Big) \\
		&= 2\theta^2 \eps_n(\mathbf{W}) + 2\eps_n(\mathbf{W}) R_\theta(2c_\theta \eps_n(\mathbf{W})) \Big(2c_\theta + \theta(1+\zeta(1+\theta))\Big).
	\end{align*}
	Since $2c_\theta + \theta(1+\zeta(1+\theta)) = 2\theta(1+2\zeta(1+\theta)) + \theta + \theta\zeta(1+\theta) = \theta(3 + 5\zeta(1+\theta))$, this yields
	$$
	d_K(S_n, D_\theta \mid \mathbf{W}) \le 2\theta^2 \eps_n(\mathbf{W}) + 2\theta (3 + 5\zeta(1+\theta))  \eps_n(\mathbf{W}) R_\theta(2c_\theta \eps_n(\mathbf{W})).
	$$
	For $\theta \ge 1$, substituting $R_\theta(2c_\theta \eps_n(\mathbf{W})) =  \max \left\{e^{-\theta \gamma}/\Gamma(\theta),k_\theta\right\}$, the bound evaluates to exactly $C_1 \eps_n(\mathbf{W})$. For $\theta \in (0,1)$, substituting $R_\theta(2c_\theta \eps_n(\mathbf{W})) = (2c_\theta \eps_n(\mathbf{W}))^{\theta-1}$ yields the bound 
	\begin{multline*}
			2\theta^2 \eps_n(\mathbf{W}) + 2\theta (3 + 5\zeta(1+\theta)) (2c_\theta)^{\theta-1} \eps_n(\mathbf{W})^\theta\\
			= 2\theta^2 \eps_n(\mathbf{W}) + 3(2c_\theta)^\theta \frac{6\theta  + 10\theta \,\zeta(1+\theta)}{6\theta + 12 \theta \,\zeta(1+\theta)} \eps_n(\mathbf{W})^\theta \le 2\theta^2 \eps_n(\mathbf{W}) + C_2 \eps_n(\mathbf{W})^\theta,
	\end{multline*}
	where $C_2 = 3(2c_\theta)^\theta$. Combining the bounds, we obtain
	\begin{equation}\label{eq:condbd}
		d_K(S_n, D_\theta \mid \mathbf{W})  \le \begin{cases} 
			C_1 \eps_n(\mathbf{W}), & \text{if } \theta \ge 1, \\[6pt] 
			2\theta^2 \eps_n(\mathbf{W}) + C_2 \eps_n(\mathbf{W})^\theta, & \text{if } \theta \in (0,1). 
		\end{cases}
	\end{equation}
	
	Finally, applying the triangle inequality, we can bound
		\begin{multline*}\label{eq:dkcond}
		d_K(S_n, D_\theta) = \sup_{x \ge 0} \big| \E \big[ \P(S_n \le x \mid \mathbf{W}) - \P(D_\theta \le x) \big] \big| \\
		\le \E \left[ \sup_{x \ge 0} \big| \P(S_n \le x \mid \mathbf{W}) - \P(D_\theta \le x) \big| \right] = \E \big[ d_K(S_n, D_\theta \mid \mathbf{W}) \big].
	\end{multline*}
	Taking an expectation throughout the bound in \eqref{eq:condbd} thus yields for $\theta \ge 1$ that
	$$
	d_K(S_n, D_\theta) \le C_1\E \left[\eps_n(\mathbf{W}) \right].
	$$
	Similarly, for $\theta \in (0,1)$, taking expectation in \eqref{eq:condbd} and applying Jensen's inequality yields
	$$
	d_K(S_n, D_\theta) \le \E \left[ 2\theta^2 \eps_n(\mathbf{W}) + C_2\eps_n(\mathbf{W})^\theta \right] \le 2\theta^2 \E \left[ \eps_n (\mathbf{W})\right] + C_2 (\E[\eps_n (\mathbf{W})])^\theta.
	$$
	This concludes the proof.
\end{proof}

\begin{proof}[Proof of Theorem \ref{thm:randrates_unordered}]
	Fix $n \in \N$. Conditionally on the weights $\mathbf{W}$, let $\sigma : [n] \to [n]$ be a deterministic permutation of $[n]$ that sorts the weights $W_{n,1}, \hdots, W_{n,n}$ into non-decreasing order, i.e., 
	$$	W_{n,\sigma(1)} \le W_{n,\sigma(2)} \le \dots \le W_{n,\sigma(n)},$$where any ties are broken arbitrarily. We define the correspondingly permuted sequence of random counts as $V_k := X_{\sigma(k)}$ for $k \in [n]$.
	
Because $X_1, \dots, X_n$ are mutually independent conditionally on $\mathbf{W}$ by assumption \ref{a1}, and because $\sigma$ is fully determined by $\mathbf{W}$, the permuted sequence $(V_1, \dots, V_n)$ remains a sequence of mutually independent random variables under the conditional measure $\P(\, \cdot \, \, | \, \mathbf{W})$. Moreover, we have $\E[V_k \mid \mathbf{W}] = \lambda_{\sigma(k)}(\mathbf{W})$. Since the order statistics $W_{n,(k)}$ trivially satisfy assumption \ref{a2}, and conditionally $(V_k)_{k \in [n]}$ remain independent, the relabelled sum
$$
S_n = \sum_{k=1}^n W_{n,k} X_k = \sum_{k=1}^n W_{n,(k)} V_k
$$
satisfies all the conditions of Theorem \ref{thm:randrates_cond}. It only remains to verify that the aggregate error \eqref{eq:eps_n_random}, evaluated for the permuted sequence yields exactly $\eps_n^{\text{os}}(\mathbf{W})$. The terms $\delta_n$ and $D_n^*$ take the form of $\delta_n^{\text{os}}$ and $D_n^{*,\text{os}}$ with the ordered weights. Also, the maximum weight $W_{\max}^{(n)} = W_{n,(n)}$ is unaffected by the ordering. Finally, noting $W_{n,(k)} = W_{n,\sigma(k)}$ and $V_k = X_{\sigma(k)}$, the third summand in \eqref{eq:eps_n_random} simplifies to
$$
\frac{2}{\theta}\sum_{k=1}^n \frac{\lambda_{\sigma(k)}(\mathbf{W}) W_{n,(k)} \, d_{TV}(V_k^*, V_k+1 \mid \mathbf{W})}{\P(V_k = x_{\min}^{(\sigma(k))}(\mathbf{W}) \mid \mathbf{W})} = \frac{2}{\theta} \sum_{k=1}^n \frac{\lambda_k(\mathbf{W}) W_{n,k} \, d_{TV}(X_k^*, X_k+1 \mid \mathbf{W})}{\P(X_k = x_{\min}^{(k)}(\mathbf{W}) \mid \mathbf{W})}.
$$
Thus, applying Theorem \ref{thm:randrates_cond} conditionally with the sorted sequence yields
$$
d_K(S_n, D_\theta \mid \mathbf{W}) \le \begin{cases} 
	C_1 \eps_n^{\text{os}} (\mathbf{W}), & \text{if } \theta \ge 1, \\[6pt] 
	2\theta^2 \eps_n^{\text{os}} (\mathbf{W}) + C_2 \left(\eps_n^{\text{os}} (\mathbf{W})\right)^\theta, & \text{if } \theta \in (0,1). 
\end{cases}
$$
Since $d_K(S_n, D_\theta) \le \E[ d_K(S_n, D_\theta \mid \mathbf{W})]$, taking an expectation and applying Jensen's inequality for $\theta \in (0,1)$ yields the result.
\end{proof}

\subsection{Proof of Theorem \ref{thm:d11_bound}}

For any test function $h \in \mathcal{H}_{1,1}$, let $f_h$ denote the solution to the Dickman Stein equation \eqref{eq:stein_eq_d11} guaranteed by Theorem \ref{thm:d11solbd} with $\|f_h'\|_\infty \le \theta$ and $\|f_h''\|_\infty \le \theta/2$. Evaluating for $x=S_n$ and taking conditional expectations given $\mathbf{W}$, we obtain from \eqref{eq:stein_eq_d11} that
\begin{equation}\label{eq:stein_eval_d11}
	\E[h(S_n) - h(D_\theta) \mid \mathbf{W}] = \E \left[ \frac{S_n}{\theta}f_h'(S_n) - \int_0^1 f_h'(S_n+t)dt \;\middle|\; \mathbf{W} \right].
\end{equation}
For the term $\E[S_n f_h'(S_n) \mid \mathbf{W}]$, we deploy the same size-bias coupling argument used in the proof of Theorem \ref{thm:main_bound}. Specifically, following the derivation in \eqref{eq:decomp_step3}, we can write
\begin{align*}
	\frac{1}{\theta}\E[S_n f_h'(S_n) \mid \mathbf{W}] &= \frac{1}{\theta}\sum_{k=1}^n \lambda_k(\mathbf{W}) W_{n,k} \E[f_h'(S_n^{(k)} + W_{n,k}X_k^*) \mid \mathbf{W}]\\
	&= I_1 + I_2 + I_3 + I_4 +  \E \left[ \int_0^{t_n}f_h'(S_n + t) \, \dd t\mid \mathbf{W} \right],
\end{align*}
where analogously to \eqref{eq:t14_rand}, we denote 
\begin{align*}
	I_1 &= \frac{1}{\theta}\sum_{k=1}^n \lambda_k(\mathbf{W}) W_{n,k} \E \big[ f_h'(S_n^{(k)} + W_{n,k}X_k^*) - f_h'(S_n + W_{n,k}) \mid \mathbf{W} \big], \\
	I_2 &= \frac{1}{\theta}\sum_{k=1}^n \lambda_k(\mathbf{W}) W_{n,k} \E \big[ f_h'(S_n + W_{n,k}) - f_h'(S_n + t_k) \mid \mathbf{W} \big], \\
	I_3 &= \sum_{k=1}^n \left( \frac{\lambda_k(\mathbf{W}) W_{n,k}}{\theta} - \Delta t_k \right) \E \big[ f_h'(S_n + t_k) \mid \mathbf{W} \big], \\
	I_4 &= \sum_{k=1}^n \E \left[ \int_{t_{k-1}}^{t_k} \big( f_h'(S_n + t_k) - f_h'(S_n + t) \big) dt \;\middle|\; \mathbf{W} \right].
\end{align*}
Since $\|f_h'\|_ \infty \le \theta$, this yields via \eqref{eq:stein_eval_d11} that 
\begin{equation}\label{eq:I14bd}
	\E[h(S_n) - h(D_\theta) \mid \mathbf{W}] - (I_1 + I_2 + I_3 + I_4) \le \theta |t_n -1|.
\end{equation}
Using $\|f_h''\|_\infty \le \theta/2$, noting that $S_n = S_n^{(k)} + W_{n,k}X_k$, and taking an optimal coupling of $(X_k^*, X_k)$ given $\mathbf{W}$ so that $\E\big[|X_k^* - (X_k+1)| \big| \mathbf{W} \big] = d_1(X_k^* ,X_k+1 \mid \mathbf{W})$, we obtain applying the Mean Value Theorem that
\begin{align*}
	|I_1| &\le \frac{1}{\theta} \sum_{k=1}^n \lambda_k(\mathbf{W}) W_{n,k} \E \left[ (\theta/2) \big| S_n^{(k)} + W_{n,k}X_k^* - (S_n^{(k)} + W_{n,k}X_k + W_{n,k}) \big| \,\Big|\, \mathbf{W} \right] \\
	&= \frac{1}{2} \sum_{k=1}^n \lambda_k(\mathbf{W}) W_{n,k}^2 \E\left[|X_k^* - (X_k+1)| \, \Big| \, \mathbf{W}\right] = \frac{1}{2} \sum_{k=1}^n \lambda_k(\mathbf{W}) W_{n,k}^2 \, d_1(X_k^* ,X_k+1 \mid \mathbf{W}).
\end{align*}
Similarly, for $I_2$ we can bound
$$
|I_2| \le \frac{1}{2} \sum_{k=1}^n \lambda_k(\mathbf{W}) W_{n,k} |W_{n,k} - t_k|.
$$
For $I_3$, since $D_{n,k}^{(t)} = \sum_{j=1}^k (\lambda_j(\mathbf{W}) W_{n,j} - \theta \Delta t_j)$ for $k \in [n]$, with the convention $D_{n,0}^{(t)} = 0$, we can rewrite $I_3$ employing summation by parts as
\begin{align*}
	I_3 &= \frac{1}{\theta} \sum_{k=1}^n \big( D_{n,k}^{(t)} - D_{n,k-1}^{(t)} \big) \E \big[ f_h'(S_n + t_k) \mid \mathbf{W} \big] \\
	&= \frac{1}{\theta} D_{n,n}^{(t)} \E \big[ f_h'(S_n + t_n) \mid \mathbf{W} \big] - \frac{1}{\theta} \sum_{k=1}^{n-1} D_{n,k}^{(t)} \E \big[ f_h'(S_n + t_{k+1}) - f_h'(S_n + t_k) \mid \mathbf{W} \big].
\end{align*}
Applying the uniform bounds $\|f_h'\|_\infty \le \theta$ and $\|f_h''\|_\infty \le \theta/2$ to the respective terms, we obtain
$$
|I_3| \le |D_{n,n}^{(t)}| + \frac{1}{2} \sum_{k=1}^{n-1} |D_{n,k}^{(t)}| \Delta t_{k+1} \le D_n^{(t)*} + \frac{1}{2} D_n^{(t)*} \sum_{k=1}^{n-1} \Delta t_{k+1} \le \left(1 + \frac{t_n}{2}\right) D_n^{(t)*}.
$$
Finally, for $I_4$, applying the Lipschitz property of $f_h'$ to the integral over each sub-interval $[t_{k-1}, t_k], k \in [n]$ yields
$$
|I_4| \le \frac{\theta}{2}\sum_{k=1}^n \int_{t_{k-1}}^{t_k}  (t_k - t) dt = \frac{\theta}{4} \sum_{k=1}^n (\Delta t_k)^2 \le \frac{\theta}{4} \delta_n^{(t)} \sum_{k=1}^n \Delta t_k = \frac{\theta t_n}{4} \delta_n^{(t)}.
$$
Summing the bounds for $|I_1|$ through $|I_4|$ and combining with \eqref{eq:I14bd}, we obtain
$$
d_{1,1}(S_n, D_\theta \mid \mathbf{W}) = \sup_{h \in \mathcal{H}_{1,1}}|\E[h(S_n) - h(D_\theta) \mid \mathbf{W}]| \le \eps_n^{1,1}(\mathbf{W}),
$$
with the aggregate error $\eps_n^{1,1}(\mathbf{W})$ defined in \eqref{eq:eps_n_d11}. Taking the unconditional expectation and passing the supremum over $h \in \mathcal{H}_{1,1}$ yields
$$
d_{1,1}(S_n, D_\theta) = \sup_{h \in \mathcal{H}_{1,1}}|\E[h(S_n) - h(D_\theta)| \le 
\E \left[ d_{1,1}(S_n, D_\theta \mid \mathbf{W}) \right] \le \E [\eps_n^{1,1}(\mathbf{W})],
$$
concluding the proof.

\section{Proofs of results in Section \ref{sec:app}}\label{sec:appproofs}

\subsection{Proof of Theorem \ref{thm:harmonic}}\label{sec:proofharmonic}
It is a classical result that under the true harmonic measure (i.e., $\theta=1$), the primes exponents for a random number sampled according to the measures $\mathbb{P}_n$ and $\mathbb{P}_n^{\text{sq}}$ are independent and distributed as Geometric (see e.g.\ \cite{Pinsky16a}) and Bernoulli (see e.g.\ \cite{CS}), respectively. The following lemma establishes the exact distribution of the prime exponents $(X_k)_{k \in [n]}$ under general values of $\theta>0$.

Recall that a random variable $X$ following a negative binomial distribution with parameters $\theta>0$ and $q \in [0,1]$, written as $X \sim \operatorname{NB}(\theta,q)$, has the probability mass function
$$
\mathbb{P}(X = j) = \binom{j+\theta-1}{j} (1-q)^j q^\theta = \frac{\theta(\theta+1)\cdots(\theta+j-1)}{j!}(1-q)^j q^\theta, \quad j \in \N_0.
$$

\begin{lemma}\label{lem:prime_exponents} Let $M_n=\prod_{k=1}^n p_k^{X_k}$ and $M_n^{\text{sq}}=\prod_{k=1}^n p_k^{X_k^{\text{sq}}}$ be sampled according to the measures $\mathbb{P}_n$ and $\mathbb{P}_n^{\text{sq}}$ in \eqref{eq:harmonic_measure} and \eqref{eq:rest.harmonic_measure}, respectively. Then the following holds.
	\begin{enumerate}
		\item 	Under $\mathbb{P}_n$, the exponents $X_1, \dots, X_n$ are independent with $X_k \sim \operatorname{NB}(\theta,1-1/p_k), k \in [n]$.
		\item Under $\mathbb{P}_n^{\text{sq}}$, the exponents $X_1^{\text{sq}}, \dots, X_n^{\text{sq}}$ are independent with $X_k^{\text{sq}} \sim \operatorname{Ber}\big(\theta/(p_k + \theta)\big)$, $k \in [n]$.
	\end{enumerate}

\end{lemma}

\begin{proof}
	For the first assertion, we start by computing the partition function $Z_n = \sum_{N \in \Omega_n} \frac{\tau_\theta(N)}{N}$ explicitly. Because every $N \in \Omega_n$ can be uniquely decomposed into product of powers of the first $n$ primes, and $\tau_\theta$ is multiplicative, we can write
	\begin{equation}\label{eq:Zn}
		Z_n = \sum_{j_1, \hdots, j_n=0}^\infty \prod_{k=1}^n \frac{\tau_\theta(p_k^{j_k})}{p_k^{j_k}} = \prod_{k=1}^n \left( \sum_{j=0}^\infty \binom{j+\theta-1}{j} \left(\frac{1}{p_k}\right)^j \right) =  \prod_{k=1}^n (1 - 1/p_k)^{-\theta},
	\end{equation}
	where the final step is by the generalized binomial theorem noting $p_k > 1$ for $k \in \N$. Thus, by \eqref{eq:harmonic_measure}, we obtain
	\begin{align*}
		\mathbb{P}(X_1 = j_1, \dots, X_n = j_n) = \mathbb{P}_n\left(\big\{\prod_{k=1}^n p_k^{j_k}\big \}\right) &= \left( \prod_{k=1}^n \left(1 - \frac{1}{p_k}\right)^\theta \right) \frac{\prod_{k=1}^n \tau_\theta(p_k^{j_k})}{\prod_{k=1}^n p_k^{j_k}} \\
		&= \prod_{k=1}^n \left[ \binom{j_k+\theta-1}{j_k} \left(\frac{1}{p_k}\right)^{j_k} \left(1 - \frac{1}{p_k}\right)^\theta \right],
	\end{align*}
	which is the product of the probability mass functions of $\operatorname{NB}(\theta, 1-1/p_k), k \in [n]$ distributions. The conclusion follows.
	\medskip
	
	Next we prove the second assertion. Note that $\tau_\theta(p^j) = \theta^j$ for $j \in \{0,1\}$ and a prime $p$. Therefore, arguing similarly as above, we can write
	\begin{equation}\label{eq:ZnSq}
		Z_n^{\text{sq}} = \sum_{j_1, \hdots, j_n=0}^1 \prod_{k=1}^n \frac{\tau_\theta(p_k^{j_k})}{p_k^{j_k}} = \prod_{k=1}^n \left( \frac{\tau_\theta(p_k^0)}{p_k^0} + \frac{\tau_\theta(p_k^1)}{p_k^1} \right) = \prod_{k=1}^n \left( 1 + \frac{\theta}{p_k} \right).
	\end{equation}
	Hence by \eqref{eq:rest.harmonic_measure} and substituting $Z_n^{\text{sq}}$, we obtain
	\begin{multline*}
		\mathbb{P}(X_1^{\text{sq}} = j_1, \dots, X_n^{\text{sq}} = j_n) = \mathbb{P}_n^{\text{sq}}\left(\big\{\prod_{k=1}^n p_k^{j_k}\big \}\right) = \left( \prod_{k=1}^n \left(1 + \frac{\theta}{p_k}\right)^{-1} \right) \prod_{k=1}^n \left( \frac{\theta}{p_k} \right)^{j_k} \\
		= \prod_{k=1}^n \left[ \left(\frac{\theta}{p_k}\right)^{j_k} \left( \frac{p_k}{p_k + \theta} \right) \right] = \prod_{k=1}^n \left[ \left(\frac{\theta}{p_k+ \theta}\right)^{j_k} \left( \frac{p_k}{p_k + \theta} \right)^{1-j_k} \right],
	\end{multline*}
	yielding the result.
\end{proof}

\begin{rem}
	Note that when $\theta=1$, then $\tau_1 \equiv 1$. In this case, the measures $\mathbb{P}_n$ and $\mathbb{P}_n^{\text{sq}}$ become truly harmonic, and the exponents $(X_k)_{k \in [n]}$ remain independent and are distributed as $\operatorname{Geom}(1-1/p_k)$ and $\operatorname{Ber}(1/(1+p_k))$, respectively.
\end{rem}

For $M_n = \prod_{k=1}^n p_k^{X_k}$ and $M_n^{\text{sq}} = \prod_{k=1}^n p_k^{X_k^{\text{sq}}}$, $n \in \N$ as in Lemma \ref{lem:prime_exponents}, consider the sums
$$
S_n := \frac{\log M_n}{\log p_n}= \frac{1}{\log p_n} \sum_{k=1}^n X_k \log p_k, \quad \text{and} \quad S_n^{\text{sq}} : = \frac{\log M_n^{\text{sq}}}{\log p_n}= \frac{1}{\log p_n} \sum_{k=1}^n X_k^{\text{sq}} \log p_k,
$$
where $X_k \sim \text{NB}(\theta, 1 - 1/p_k)$ and $X_k^{\text{sq}} \sim \text{Ber}\big(\theta/(p_k + \theta)\big)$ are independent for $k \in [n]$. This fits perfectly into our framework of Corollary \ref{cor:indep_weights} by defining the deterministic weights $W_{n,k} = \frac{\log p_k}{\log p_n}$, which trivially satisfy the monotonicity assumption \ref{a2} with $W_{\max}^{(n)}=1$. In order to apply Corollary \ref{cor:indep_weights} to prove the upper bound in Theorem \ref{thm:harmonic}, we first make the following observation on the size-biased distribution of $NB(\theta,1-1/p_k)$.
\begin{lemma}\label{lem:nb_size_bias}
	Let $X_k \sim {\rm NB}(\theta, 1 - 1/p_k), k \in \N$. Then its size-biased distribution $X_k^*$ satisfies
	$$
	X_k^* \stackrel{d}{=} X_k + Y_k,
	$$
	where $Y_k$ is an independent random variable supported on $\N$ satisfying $\P(Y_k \neq 1) = 1/p_k$ and $\E Y_k = (1-1/p_k)^{-1}$.
\end{lemma}

\begin{proof}
	Denoting $p = 1 - 1/p_k$ and $q = 1/p_k$ for simplicity, note that $\E[X_k] = \frac{\theta q}{p}$ and the probability generating function (PGF) of $X_k$ is given by
	$$
	G_{X_k}(s) = \E[s^{X_k}] = \left( \frac{p}{1 - qs} \right)^\theta, \quad |s| <1/q.
	$$
	By the definition of the size-biased distribution, for $|s|<1/q$, we have
	\begin{multline*}
		G_{X_k^*}(s) = \E[s^{X_k^*}] = \sum_{j=0}^\infty s^j \, \frac{j \, \P(X_k = j)}{\E[X_k]} = \frac{s}{\E[X_k]} G_{X_k}'(s) \\
		= \frac{s}{\theta q/p} \left[ \theta \left( \frac{p}{1 - qs} \right)^{\theta - 1} \frac{pq}{(1 - qs)^2} \right] = \left( \frac{ps}{1 - qs} \right) \left( \frac{p}{1 - qs} \right)^\theta = \left( \frac{ps}{1 - qs} \right) 	G_X(s).
	\end{multline*}
	Thus, we have $X_k^* \stackrel{d}{=} X_k + Y_k$, where $Y_k$ is independent, with its PGF given by
	$$
	G_{Y_k}(s) =\left( \frac{ps}{1 - qs} \right)  = ps \sum_{i=0}^\infty (qs)^i = \sum_{y=1}^\infty p q^{y-1} s^y,  \quad |s| <1/q,
	$$
	whence it has the probability mass function $\P(Y_k = y) = p q^{y-1}$ for $y \in \N$. Consequently, $Y_k \ge 1$ a.s., and $\P(Y_k = 1) = p = 1 - 1/p_k$. Moreover, for $k \in \N$, we have
	$$
	\E[Y_k] = G_{Y_k}'(1) = \frac{1}{1-1/p_k},
	$$
	yielding the result.
\end{proof}

We are now ready to prove Theorem \ref{thm:harmonic}. We prove the upper and lower bounds separately.

\begin{proof}[Proof of upper bounds in Theorem \ref{thm:harmonic}]
	We first prove the upper bound on the Kolmogorov distance. Because the weights $W_{n,k} = \log p_k/ \log p_n, k \in [n]$ are ordered, deterministic and $W_{\max}^{(n)} = 1$, by Corollary \ref{cor:indep_weights}, the bound to the Kolmogorov distances are determined by the aggregate errors
	\begin{equation}\label{eq:er1}
		\eps_n = \delta_n + (2/\theta)D_n^* + (2/\theta)\sum_{k=1}^n \lambda_k W_{n,k} \, \frac{d_{TV}(X_k^*, X_k+1)}{\P(X_k = 0)},
	\end{equation}
	and
	\begin{equation}\label{eq:er2}
		\eps_n^{\text{sq}} = \delta_n + (2/\theta) D_n^* + (2/\theta) \sum_{k=1}^n \lambda_k W_{n,k} \, \frac{d_{TV}((X_k^{\text{sq}})^*, X_k^{\text{sq}}+1)}{\P(X_k^{\text{sq}} = 0)},
	\end{equation}
	respectively. Note here that $\delta_n$ and $D_n^*$ do not depend on the prime powers $(X_k)_{k \in [n]}$ and $(X_k^{\text{sq}})_{k \in [n]}$ and therefore are the same in both cases. 
	
	We first bound $\delta_n$. We have $\Delta_{n,k} = \theta(W_{n,k} - W_{n,k-1})= \frac{\theta(\log p_k - \log p_{k-1})}{\log p_n}$, $k \in [n]$, with the convention $p_0=1$. Hence by Bertrand's postulate which states that $p_k/p_{k-1} \le 2$ for $k \ge 1$, we obtain that
	\begin{equation}\label{eq:delnrec}
			\delta_n  = \frac{1}{\theta}\max_{k \in [n]} \Delta_{n,k} \le \frac{\log 2}{\log p_n}.
	\end{equation}
	
	To bound $D_n^*$, note that we can write the expected value for the exponents (i.e., $\E X_k$ or $\E X_k^{\text{sq}}$, $k \in [n]$) as $\lambda_k = \frac{\theta}{p_k} + a_k$, where for $\mathbb{P}_n$, we have
	\begin{equation*}
		a_k = \E[X_k] - \frac{\theta}{p_k}= \frac{\theta}{p_k - 1} - \frac{\theta}{p_k} = \frac{\theta}{p_k(p_k-1)}, \quad k \in [n],
	\end{equation*}
	while for the square-free measure $\mathbb{P}_n^{\text{sq}}$,
	\begin{equation*}
			a_k = \E[X_k^{\text{sq}}] -  \frac{\theta}{p_k}= \frac{\theta}{p_k+\theta} - \frac{\theta}{p_k} = - \frac{\theta^2}{p_k(p_k+\theta)}, \quad k \in [n].
	\end{equation*}
	In both models, this residual satisfies $|a_k| \le \mathcal{O}(1/p_k^2)$. Recall that $D_n^*= \max_{m \in [n]} |D_{n,m}|$, where
	$$
	D_{n,m} = \sum_{k=1}^m (\lambda_k W_{n,k} - \Delta_{n,k}) = \frac{\theta}{\log p_n} \Bigg( \sum_{k=1}^m \frac{\log p_k}{p_k} - \log p_m \Bigg) + \frac{1}{\log p_n} \sum_{k=1}^m a_k \log p_k.
	$$
	By Mertens' First Theorem, 
	$$
	\Big|\sum_{k=1}^m \frac{\log p_k}{p_k} - \log p_m\Big| \le 2.
	$$
	Also,
	\begin{equation}\label{eq:intub}
		0 \le \sum_{k=1}^\infty \frac{1}{(p_k-1)^2} \log p_k \le 2\sum_{m \ge 2} \frac{\log m}{m^2} = -2\zeta'(2)\approx 1.87<\infty.
	\end{equation}
	Because $|a_k| \le \mathcal{O}(1/p_k^2)$, we thus have that $\sum_{k=1}^\infty |a_k|  \log p_k <\infty$. This yields
	\begin{equation}\label{eq:Dnrec}
			D_n^* \le \frac{\theta}{\log p_n} \, \max_{m \in [n]} \Big| \sum_{k=1}^m \frac{\log p_k}{p_k} - \log p_m \Big| + \frac{1}{\log p_n} \sum_{k=1}^n |a_k| \log p_k = \mathcal{O}\left(\frac{1}{\log p_n}\right).
	\end{equation}
	
	Finally, we estimate the third components in the aggregate error terms in \eqref{eq:er1} and \eqref{eq:er2}. By Lemma \ref{lem:nb_size_bias}, we have $X_k^* \stackrel{d}{=} X_k + Y_k$, where $Y_k \in \N$, and
	$$
	d_{TV}(X_k^*, X_k+1) \le \P(X_k+Y_k \neq X_k+1) = \P(Y_k \neq 1) = \frac{1}{p_k}.
	$$
	For $X_k^{\text{sq}}$, we have $(X_k^{\text{sq}})^* \equiv 1$, so that 
	$$
	d_{TV}((X_k^{\text{sq}})^*, X_k^{\text{sq}} + 1)= \P(1 \neq X_k^{\text{sq}} + 1) = \P(X_k^{\text{sq}} = 1) = \frac{\theta}{p_k+\theta} \le \frac{\theta}{p_k}.
	$$
	Because $\P(X_k = 0) =  \left( 1 - 1/p_k \right)^\theta \ge 2^{-\theta}$ and $\P(X_k^{\text{sq}}=0) = p_k/(p_k+\theta) \ge 2/(2+\theta)$ for $k \in [n]$, for $x_{\min}^{(k)}$ defined in \eqref{eq:xmin}, we have $x_{\min}^{(k)} = 0$ in both cases. Thus, using \eqref{eq:intub} we can bound
	$$
	\sum_{k=1}^n \lambda_k W_{n,k} \, \frac{d_{TV}(X_k^*, X_k+1)}{\P(X_k = 0)} \le \frac{\theta 2^\theta}{\log p_n} \sum_{k=1}^n \frac{\log p_k}{p_k (p_k-1)} = \mathcal{O}\left(\frac{1}{\log p_n}\right),
	$$
	and
	$$
	\sum_{k=1}^n \lambda_k W_{n,k} \, \frac{d_{TV}((X_k^{\text{sq}})^*, X_k^{\text{sq}} + 1)}{\P(X_k^{\text{sq}} = 0)} \le \frac{\theta^2 (2+\theta)}{2\log p_n} \sum_{k=1}^n \frac{\log p_k}{p_k(p_k + \theta)} = \mathcal{O}\left(\frac{1}{\log p_n}\right).
	$$
	Combining these bounds with those obtained for $\delta_n, D_n^*$ in \eqref{eq:delnrec} and \eqref{eq:Dnrec} above, we obtain $\eps_n, \eps_n^{\text{sq}}  = \mathcal{O}\left(\frac{1}{\log p_n}\right)$. The asserted Kolmogorov upper bound now follows from Corollary \ref{cor:indep_weights} upon noting that $\log p_n \ge \log n$.
	
	We conclude by establishing the corresponding upper bounds for the $d_{1,1}$ distance via Theorem \ref{thm:d11_bound}. As highlighted in Remark \ref{rem:metric_comparison}, since the deterministic weights are naturally ordered, we can make the natural grid choice $t_k = W_{n,k}$ for $k \in [n]$, with $t_0 = 0$.

		Under this choice, the components of the aggregate error $\eps_n^{1,1}$ in \eqref{eq:eps_n_d11} perfectly collapse into the existing error terms above. Specifically, the grid gap and the grid discrepancies are given by $\delta_n^{(t)} = \max_{k \in [n]} (W_{n,k} - W_{n,k-1}) = \delta_n$ and
		$
		D_{n,m}^{(t)} = \sum_{j=1}^m \big(\lambda_j W_{n,j} - \theta(W_{n,j} - W_{n,j-1})\big) = D_{n,m},
		$
		for $m \in [n]$ respectively. Furthermore, since $t_n = W_{n,n} = 1$, we have $\theta|t_n-1|=0$ as well as $|W_{n,k} - t_k| = 0, k \in [n]$. Thus, the aggregate error in \eqref{eq:eps_n_d11} simplifies to
		$$
		\eps_n^{1,1} = \frac{\theta}{4} \delta_n + \frac{3}{2} D_n^* + \frac{1}{2} \sum_{k=1}^n \lambda_k W_{n,k}^2 \, d_1(X_k^*, X_k+1).
		$$
		We have already established in \eqref{eq:delnrec} and \eqref{eq:Dnrec} above that $\delta_n$ and $D_n^*$ are of order $\mathcal{O}(1/\log p_n)$. It remains only to bound the size-bias coupling term. For the square-free case $\mathbb{P}_n^{\text{sq}}$, we have $(X_k^{\text{sq}})^* \equiv 1$, so that $d_1((X_k^{\text{sq}})^*, X_k^{\text{sq}} + 1) = \E[X_k^{\text{sq}}] \le \theta/p_k$. For the measure $\mathbb{P}_n$, we have from Lemma \ref{lem:nb_size_bias} that $X_k \sim \text{NB}(\theta, 1-1/p_k), k \in [n]$ and $X_k^* \stackrel{d}{=} X_k + Y_k$ where $Y_k \in \N$ almost surely. Thus, Lemma \ref{lem:nb_size_bias} yields 
		$$
		d_1(X_k^*, X_k+1) \le \E|Y_k - 1| = \E[Y_k] - 1 = \frac{1}{p_k-1} = \mathcal{O}(1/p_k).
		$$
		In both models, since $\lambda_k \le \mathcal{O}(1/p_k)$ by Lemma \ref{lem:prime_exponents}, and $W_{n,k} = \frac{\log p_k}{\log p_n}$ for $k \in [n]$, by Theorem \ref{thm:d11_bound} and \eqref{eq:intub}, we obtain
		$$
		d_{1,1}(S_n, D_\theta) \le \eps_n^{1,1} = \mathcal{O}\left(\frac{1}{\log p_n}\right) + \mathcal{O}\left(\frac{1}{\log^2 p_n}\right) \sum_{k=1}^n \frac{\log^2 p_k}{p_k^2} = \mathcal{O}\left(\frac{1}{\log p_n}\right),
		$$
		yielding the desired upper bound.
\end{proof}

Next, we prove the matching lower bounds in Theorem \ref{thm:harmonic}. To obtain the lower bound on the Kolmogorov distance, we compare the probability that $S_n$ or $S_n^{\text{sq}}$ is less than or equal to 1 with the same probability for $D_\theta$. The lemma below shows that for $\theta > 1$, their absolute difference is at least of the order of $1/\log p_n$. Its proof involves a careful estimation of the probabilities $\mathbb P(S_n\le 1)$ and $\mathbb P(S_n^{\text{sq}}\le 1) $ using the Selberg-Delange method from analytic number theory (see e.g.\ \cite{tet}), and can be found in Appendix \ref{app:lemlbpf}.
\begin{lemma}\label{lem:snlb}
	For $\theta > 1$,
	\begin{equation}\label{eq:snlbint}
		\mathbb P(S_n\le 1) = \mathbb P(D_\theta\le 1) + \frac{C_\theta}{\log p_n} + \mathcal{O}\left(\frac{\log \log p_n}{(\log p_n)^{\min \{2, \theta\}}}\right),
	\end{equation}
	and
	\begin{equation}\label{eq:snsqlbint}
		\mathbb P(S_n^{\text{sq}}\le 1) = \mathbb P(D_\theta\le 1) + \frac{C_\theta^{\text{sq}}}{\log p_n} + \mathcal{O}\left(\frac{\log \log p_n}{(\log p_n)^{\min \{2, \theta\}}}\right),
	\end{equation}
	where $C_\theta = \frac{\theta^2 \gamma e^{-\gamma\theta}}{\Gamma(\theta+1)}$ and $C_\theta^{\text{sq}} = \frac{c_\theta^{\text{sq}} e^{-\gamma\theta}}{\Gamma(\theta+1)} \neq 0$, with $\theta^2 \gamma \le c_\theta^{\text{sq}} \in (0,\infty)$ a constant depending only on $\theta$.
\end{lemma}

\begin{proof}[Proof of lower bounds in Theorem \ref{thm:harmonic}]
	
	We start by proving the lower bound on $d_K(S_n,D_\theta)$. For $\theta > 1$, the lower bounds follows directly from Lemma \ref{lem:snlb}, upon noting that that $\log p_n \sim \log n$ due to the Prime number theorem, and
	$$
	d_K(S_n,D_\theta)\ge \bigl|\mathbb P(S_n\le 1)-\mathbb P(D_\theta\le 1)\bigr|
	\quad \text{
	and} \quad 
	d_K(S_n^{\text{sq}},D_\theta)\ge \bigl|\mathbb P(S_n^{\text{sq}}\le 1)-\mathbb P(D_\theta\le 1)\bigr|.
	$$
	
	Next we consider the case when $\theta \le 1$. Because $D_\theta$ is absolutely continuous, $\mathbb{P}(D_\theta \le 0) = \P(D_\theta=0)= 0$. On the other hand, $S_n = 0$ if and only $X_k = 0$ for all $k \in [n]$. By independence, and applying Mertens' third theorem, we thus have
	$$
	\mathbb{P}(S_n = 0) = \prod_{k=1}^n \mathbb{P}(X_k = 0) = \prod_{k=1}^n \left(1 - \frac{1}{p_k}\right)^\theta \asymp \left( \frac{e^{-\gamma}}{\log p_n} \right)^\theta.
	$$
	Arguing similarly, for the square-free measure, applying the inequality $1/(1 + x) \ge e^{-x}$ for $x > 0$, we obtain
	$$
	\mathbb{P}(S_n^{\text{sq}} = 0) =\prod_{k=1}^n \mathbb{P}(X_k^{\text{sq}} = 0) = \prod_{k=1}^n \frac{p_k}{p_k+\theta} \ge  \prod_{k=1}^n \exp\left(-\frac{\theta}{p_k}\right) = \exp\left( -\theta \sum_{k=1}^n \frac{1}{p_k} \right).
	$$
	By Mertens' Second Theorem, the sum of the reciprocals of primes satisfies $\sum_{k=1}^n \frac{1}{p_k} = \log \log p_n + M + o(1)$, where $M$ is the Meissel-Mertens constant. Substituting this yields
	$$
	\mathbb{P}(S_n^{\text{sq}} = 0) \ge \exp\big(-\theta(\log \log p_n + M + o(1))\big) = \frac{e^{-\theta M}}{(\log p_n)^\theta} (1+o(1)).
	$$
	Therefore, using $\mathbb{P}(D_\theta \le 0) =0$ and that $\log p_n \sim \log n$, we obtain
	\begin{align*}
		d_K(S_n, D_\theta), \; d_K(S_n^{\text{sq}}, D_\theta) &\ge \min \Big\{ \mathbb{P}(S_n \le 0), \mathbb{P}(S_n^{\text{sq}} \le 0) \Big\}= \Omega\left( \frac{1}{(\log p_n)^\theta} \right) = \Omega\left( \frac{1}{(\log n)^\theta} \right), 
	\end{align*}
	yielding the asserted lower bound for the Kolmogorov distance.

	Finally, we establish the lower bounds for the $d_{1,1}$ distance. Since the identity test function $h(x) = x$ is in $\mathcal{H}_{1,1}$, we obtain the lower bound
	\begin{equation}\label{eq:d11_mean_lb}
		d_{1,1}(S_n, D_\theta) = \sup_{f \in \mathcal{H}_{1,1}} \big|\E[f(S_n)] - \E[f(D_\theta)]\big| \ge \big|\E[S_n] - \E[D_\theta]\big| = \big|\E[S_n] - \theta\big|.
	\end{equation}
	From Lemma \ref{lem:prime_exponents}, under the measure $\P_n$ we have
	$$
	\E[S_n] = \frac{1}{\log p_n} \sum_{k=1}^n \E[X_k] \log p_k = \frac{\theta}{\log p_n} \sum_{k=1}^n \frac{\log p_k}{p_k - 1} =\frac{\theta}{\log p_n} \left(\sum_{k=1}^n \frac{\log p_k}{p_k} + \sum_{k=1}^n \frac{\log p_k}{p_k(p_k - 1)}\right).
	$$
	Here we use the following stronger version of Merten’s theorem, see e.g.\ \cite{Finch}: For $n \ge 1$,
	\begin{equation}\label{eq:Mertenst}
		\sum_{k=1}^n \log \left(p_k\right) / p_k=\log p_n +R_n \quad \text{with } \quad \lim _{n \rightarrow \infty} R_n=-\gamma-\sum_{k=1}^{\infty} \frac{\log p_k}{p_k \left(p_k-1\right)}=: R =-1.33 \ldots,
	\end{equation}
	where $\gamma$ is the Euler-Mascheroni constant. Thus, 
	$$
	\E [S_n] = \frac{\theta}{\log p_n} \left(\log p_n +R_n + \sum_{k=1}^n \frac{\log p_k}{p_k(p_k - 1)}\right) = \frac{\theta}{\log p_n} \left(\log p_n -\gamma -R +o(1) \right).
	$$
	Because $\gamma \approx 0.577$, by \eqref{eq:d11_mean_lb} we obtain
	$$
	d_{1,1}(S_n, D_\theta) \ge \big|\E[S_n] - \theta\big| = \frac{-R-\gamma}{\log p_n} + o\left(\frac{1}{\log p_n}\right) = \Omega \left(\frac{1}{\log n}\right),
	$$
	where the final step is due to the fact that $\log p_n \sim \log n$, which is a consequence of the Prime number theorem. This yields the assertion for $\P_n$.

	An analogous argument holds for the square-free measure $\mathbb{P}_n^{\text{sq}}$. By Lemma \ref{lem:prime_exponents}, we have $\E[X_k^{\text{sq}}] = \frac{\theta}{p_k + \theta}$, so that
	$$
	\E[S_n] = \frac{\theta}{\log p_n} \sum_{k=1}^n \frac{\log p_k}{p_k +\theta} =\frac{\theta}{\log p_n} \left(\sum_{k=1}^n \frac{\log p_k}{p_k} - \theta \sum_{k=1}^n \frac{\log p_k}{p_k(p_k+\theta)} \right).
	$$
	Because $p_k(p_k+\theta) \ge p_k^2$, by \eqref{eq:intub}, the second series converges absolutely to a positive finite constant $C_\theta := \sum_{k=1}^\infty \frac{\log p_k}{p_k(p_k+\theta)} \in (0,\infty)$. Thus using \eqref{eq:Mertenst} to note that $\sum_{k=1}^n \frac{\log p_k}{p_k} = \log p_n + R + o(1)$, we obtain
	$$
	\E[S_n] = \frac{\theta}{\log p_n} \big( \log p_n + R - \theta C_\theta + o(1) \big) = \theta + \frac{\theta(R - \theta C_\theta)}{\log p_n} + o\left(\frac{1}{\log p_n}\right).
	$$
	Because $R \approx -1.33 < 0$ and $\theta C_\theta > 0$, the sum of the constants $R - \theta C_\theta$ is strictly negative. This yields
	$$
	d_{1,1}(S_n, D_\theta) \ge \big|\E[S_n] - \theta\big| = \frac{\theta|R - \theta C_\theta|}{\log p_n} + o\left(\frac{1}{\log p_n}\right) = \Omega\left(\frac{1}{\log n}\right).
	$$
	As before, the final step uses the Prime Number Theorem relation $\log p_n \sim \log n$. This concludes the proof for $\mathbb{P}_n^{\text{sq}}$.
\end{proof}

\subsection{Proof of Theorem \ref{thm:unified_weights}}\label{sec:proofweights}

	
	Since the weights $\mathbf{W}$ satisfy assumption \ref{a2} and are independent of the mutually independent sequence $(X_k)_{k \in [n]}$, we apply Corollary \ref{cor:indep_weights} to obtain the upper bound on the Kolmogorov distance. Because $X_k \sim \text{Poi}(\theta/k)$, we have $X_k^* \stackrel{d}{=} X_k + 1$, yielding $d_{TV}(X_k^*, X_k+1) = 0, k \in [n]$. Thus, the aggregate error $\eps_n$ in  Corollary \ref{cor:indep_weights} simplifies to
	\begin{equation}\label{eq:eps_unified}
		\eps_n = \delta_n + (2/\theta) D_n^* + |W_{\max}^{(n)} - 1|.
	\end{equation}
	
	By \eqref{def:delD}, the maximum increment is given by $\delta_n = (1/Z_n) \max_{k \in [n]} Y_k$. Next, consider the centered process $M_k := \sum_{i=1}^k (Y_i - 1), k \in \N$, and note
	$$
	|W_{\max}^{(n)}- 1| = \left| \frac{S_n^Y}{Z_n} - 1 \right| = 
	\begin{cases}
		\frac{|M_n|}{n}, & \text{if $Z_n = n$\quad (Model A),} \\
		0, & \text{if $Z_n = S_n^Y$ (Model B)}.
	\end{cases}
	$$
	Third, we evaluate the discrepancy term $D_n^* = \max_{k \in [n]} |D_{n,k}|$. Since $\lambda_j = \E X_j = \theta/j, W_{n,j} = (1/Z_n) \sum_{i=1}^j Y_i$, and $\Delta_{n,j} = \theta Y_j / Z_n$, $j \in [n]$, by \eqref{def:delD} we have for $k \in [n]$ that
	\begin{equation*}
		D_{n,k} = \frac{\theta}{Z_n} \sum_{j=1}^k \Big( \frac{1}{j} \sum_{i=1}^j Y_i - Y_j \Big) = \frac{\theta}{Z_n} \sum_{j=1}^k \Big( \frac{M_j + j}{j} - Y_j \Big) = \frac{\theta}{Z_n} \Big( \sum_{j=1}^k \frac{M_j}{j} - M_k \Big) := \frac{\theta}{Z_n} A_k.
	\end{equation*}
	Notice that $A_k$ is identical for both models A and B. Thus, \eqref{eq:eps_unified} yields the representation
	\begin{equation}\label{eq:eps_n_representation}
		\eps_n = \delta_n + (2/\theta) D_n^* + |W_{\max}^{(n)} - 1| = \frac{1}{Z_n} \max_{k \in [n]} Y_k + \frac{2}{Z_n} \max_{k \in [n]} |A_k| + \frac{|M_n|}{n} \ind_{\{\text{Model A}\}}.
	\end{equation}
	Denoting
	\begin{equation}\label{eq:B_n}
		B_n:=\frac{\log n}{n^2}\sum_{i=1}^n \sigma_i^2 + \frac{1}{n}\Big(1 + \sum_{i=1}^n \sigma_i^2\Big)^{1/2} + \frac{1}{n} \sum_{j=1}^n \frac{1}{j} \Big(\sum_{i=1}^j \sigma_i^2\Big)^{1/2},
	\end{equation}
	the asserted Kolmogorov bound in Theorem \ref{thm:unified_weights} now follows by an application of Corollary \ref{cor:indep_weights}, and Lemma \ref{lem:expected_error_unified} below which shows that $\E [\eps_n] = \mathcal{O}(B_n)$. 
	
	Finally, we establish the upper bound for the $d_{1,1}$ distance. Since the random weights are ordered, we make the canonical grid choice $t_k = W_{n,k}$ for $k \in [n]$, with $t_0 = 0$. This ensures that $|W_{n,k} - t_k| = 0$ for all $k \in [n]$ as well as $\delta_n^{(t)} = \delta_n$ and $D_n^{(t)*} = D_n^*$.
	
	Furthermore, because $X_k \sim \operatorname{Poi}(\theta/k)$, we have $X_k^* \stackrel{d}{=} X_k + 1$, so that $d_1(X_k^*, X_k+1) = 0$. Substituting $t_n = W_{\max}^{(n)}$, the aggregate error in \eqref{eq:eps_n_d11} thus collapses to
	\begin{align}\label{eq:eps_d11_decoupled}
		\varepsilon_n^{1,1}(\mathbf{W}) &= \frac{\theta}{4} W_{\max}^{(n)} \delta_n + \left(1 + \frac{1}{2}W_{\max}^{(n)}\right) D_n^* + \theta |W_{\max}^{(n)} - 1| \nonumber\\
		&\le \frac{\theta}{4}\delta_n + \frac{3}{2}D_n^* + \theta|W_{\max}^{(n)} - 1| + |W_{\max}^{(n)} - 1| \left( \frac{\theta}{4} \delta_n + \frac{1}{2} D_n^* \right),
	\end{align}
	where the final step follows by noting $W_{\max}^{(n)} \le 1 + |W_{\max}^{(n)} - 1|$.
	
	To bound the expectation of \eqref{eq:eps_d11_decoupled}, observe that the sum $\frac{\theta}{4}\delta_n + \frac{3}{2}D_n^* + \theta|W_{\max}^{(n)} - 1|$ is bounded by $C_\theta \varepsilon_n$ for some constant $C_\theta \in (0,\infty)$ depending only on $\theta$, where $\varepsilon_n$ is the aggregate error defined in \eqref{eq:eps_n_representation}. Hence, its expectation is of the order of $B_n$ by Lemma \ref{lem:expected_error_unified} below. Since $|W_{\max}^{(n)} - 1| = 0$ in Model B, the final term in \eqref{eq:eps_d11_decoupled} vanishes completely for Model B. In Model A, we have $|W_{\max}^{(n)} - 1| = |M_n|/n$. By the Cauchy-Schwarz inequality, we thus obtain for Model A that
	\begin{equation}\label{eq:CS_cross}
		\mathbb{E}\big[|W_{\max}^{(n)} - 1| \, \delta_n\big] \le \left\|M_n/n\right\|_2 \|\delta_n\|_2, \quad \text{and} \quad \mathbb{E}\big[|W_{\max}^{(n)} - 1| D_n^*\big] \le \left\|M_n/n\right\|_2 \, \|D_n^*\|_2,
	\end{equation}
	where $\|\cdot\|_2$ denotes the $L^2$-norm. Let us now make a few observations to bound these $L^2$-norms. First, using $\max_{k \in [n]} Y_k \le 1+ \max_{k \in [n]} |Y_k-1|$, we obtain
	\begin{equation}\label{eq:deln2}
		\E\Big[(\max_{k \in [n]} Y_k)^2\Big] \le 2+2 \E\Big[\max_{k \in [n]} \,(Y_k -1)^2\Big] \le 2 + 2\sum_{k=1}^n \E (Y_k-1)^2 = 2 + 2\sum_{k=1}^n \sigma_k^2.
	\end{equation}
	Also, by the triangle inequality, we can bound
	\begin{equation}\label{eq:Akbd}
		\big\|\max_{k \in [n]} |A_k| \big\|_2 \le \Big\|\max_{k \in [n]} \Big|\sum_{j=1}^k \frac{M_j}{j}\Big| \Big\|_2 + \big\|\max_{k \in [n]} |M_k| \big\|_2  \le \sum_{j=1}^n \frac{\|M_j\|_2}{j}  + \big\|\max_{k \in [n]} |M_k| \big\|_2.
	\end{equation}
	Since $M_k = \sum_{i=1}^k (Y_i-1)$ for $k \in \N$ is a martingale, Doob's $L^2$-Maximal Inequality guarantees 
	\begin{equation}\label{eq:Doob}
		\big\|\max_{k \in [n]} |M_k| \big\|_2 \le 2\sqrt{\E[M_n^2]} = 2\Big(\sum_{i=1}^n \sigma_i^2\Big)^{1/2}.
	\end{equation}
	For Model A, since $Z_n = n$ deterministically, we therefore obtain
	$$
	\|\delta_n\|_2 \le \frac{1}{n} \big\|\max_{k \in [n]} Y_k\big\|_2 \le \frac{1}{n}\Big(2 + 2\sum_{i=1}^n \sigma_i^2\Big)^{1/2} = \mathcal{O}(B_n),
	$$
	and
	\begin{align*}
		\|D_n^*\|_2 \le \frac{\theta}{n} \big\|\max_{k \in [n]} |A_k|\big\|_2 &\le \frac{\theta}{n} \left(\sum_{j=1}^n \frac{1}{j} \Big(\sum_{i=1}^j \sigma_i^2\Big)^{1/2} + 2\Big(\sum_{i=1}^n \sigma_i^2\Big)^{1/2}\right) = \mathcal{O}(B_n).
	\end{align*}
	Thus, from \eqref{eq:eps_d11_decoupled} and \eqref{eq:CS_cross}, noting $\|M_n/n\|_2 = \frac{1}{n}(\sum_{i=1}^n \sigma_i^2)^{1/2}$ yields
	\begin{equation*}
		\mathbb{E}[\varepsilon_n^{1,1}(\mathbf{W})] \le C_{\theta}' \left(1 + \frac{1}{n}\Big(\sum_{i=1}^n \sigma_i^2\Big)^{1/2} \right) B_n,
	\end{equation*}
	for some constant $C_\theta' \in (0,\infty)$ depending on $\theta$. The desired upper bound now follows from Theorem \ref{thm:d11_bound}, concluding the proof.
	
	\begin{lemma} \label{lem:expected_error_unified}
		For $\eps_n$ defined in \eqref{eq:eps_n_representation} and $B_n$ as in \eqref{eq:B_n}, we have
		$$
		0 \le \E[\eps_n] = \mathcal{O}(B_n).
		$$
	\end{lemma}
	
	\begin{proof}
		Define the event $\mathcal{E} = \{Z_n > n/2\}$. For Model A, $Z_n = n$ whence $\mathcal{E}^c$ is empty. For Model B, since $Z_n = S_n^Y$, note by Chebyshev's inequality that 
		\begin{equation}\label{eq:tailbd}
			\P(\mathcal{E}^c) \le \P(|S_n^Y - n| \ge n/2) \le 4\sum_{i=1}^n \sigma_i^2 / n^2.
		\end{equation}
		In Model B, we trivially have $|W_{\max}^{(n)} - 1| = 0$. Moreover, for all $j \in [n]$,
		$$
		\delta_n = \max_{j \in [n]}(W_{n,j} - W_{n,j-1})= \max_{j \in [n]} Y_j/S_n^Y\le 1.
		$$ 
		Note that $\sum_{j=1}^n \Delta_{n,j} = \theta$ and $W_{n,j} \le 1$ for $j \in [n]$ in Model B. We can thus bound
		$$
		\eps_n \le 1+ \frac{2}{\theta}D_n^* \le 1+ \frac{2}{\theta} \left(\sum_{j=1}^n \lambda_j W_{n,j} + \sum_{j=1}^n \Delta_{n,j}\right) \le 1+ \frac{2}{\theta} \left(\theta \sum_{j=1}^n \frac{1}{j} + \theta \right)= \mathcal{O}(\log n),
		$$
		where we have also used $\lambda_j = \theta/j$, $j \in [n]$. Consequently, \eqref{eq:tailbd} yields that for both models,
		\begin{equation}\label{eq:erEc}
			\E[\eps_n \ind_{\mathcal{E}^c}] = \mathcal{O}\left(\frac{\log n}{n^2} \sum_{i=1}^n \sigma_i^2 \right).
		\end{equation}
		
		We thus restrict our remaining analysis to $\mathcal{E}$. We bound the expectation of the three components in \eqref{eq:eps_n_representation} in this case separately. Applying the Cauchy-Schwarz inequality and noting that $Z_n \ge n/2$ on $\mathcal{E}$, we obtain
		\begin{multline*}
			\E\left[\frac{1}{Z_n} \max_{k \in [n]} Y_k \ind_{\mathcal{E}}\right] \le \left( \E\Big[(\max_{k \in [n]} Y_k)^2\Big] \right)^{1/2} \left( \E\big[Z_n^{-2} \ind_{\mathcal{E}}\big] \right)^{1/2} \\
			\le \frac{2}{n} \left( \E\Big[(\max_{k \in [n]} Y_k)^2\Big] \right)^{1/2} \le \frac{2}{n} \Big(2 + 2\sum_{k=1}^n \sigma_k^2\Big)^{1/2} = \mathcal{O}\left( \frac{1}{n}\Big(1 + \sum_{k=1}^n \sigma_k^2\Big)^{1/2} \right),
		\end{multline*}
		where the penultimate step is due to \eqref{eq:deln2}. Similarly  using the Cauchy-Schwarz inequality along with \eqref{eq:Akbd} and \eqref{eq:Doob}, we can bound
		\begin{multline*}
			\E\left[\frac{1}{Z_n} \max_{k \in [n]} |A_k| \ind_{\mathcal{E}}\right] \le \left( \E\left[ \max_{k \in [n]} |A_k|^2 \right] \right)^{1/2} \left( \E\big[Z_n^{-2} \ind_{\mathcal{E}}\big] \right)^{1/2} \le \frac{2}{n}\, \big\|\max_{k \in [n]} |A_k| \big\|_2\\
			\le \frac{2}{n} \Bigg( \sum_{j=1}^n \frac{\|M_j\|_2}{j}  + \big\|\max_{k \in [n]} |M_k| \big\|_2\Bigg) \le \frac{2}{n} \Bigg(  \sum_{j=1}^n \frac{1}{j} \Big(\sum_{i=1}^j \sigma_i^2\Big)^{1/2} + 2\Big(\sum_{i=1}^n \sigma_i^2\Big)^{1/2} \Bigg),
		\end{multline*}
		where we have also used the fact that 
		$
		\sum_{j=1}^n \|M_j\|_2 / j= \sum_{j=1}^n \frac{1}{j} \big(\sum_{i=1}^j \sigma_i^2\big)^{1/2}.
		$
		
		Finally, for Model A, we can bound
		$
		\E[|M_n|]/n \le \|M_n\|_2 / n = \frac{1}{n} \big(\sum_{i=1}^n \sigma_i^2\big)^{1/2}.
		$
		Summing the expected errors over $\mathcal{E}^c$ from \eqref{eq:erEc} and over $\mathcal{E}$ in the three different components in \eqref{eq:eps_n_representation} from the bounds above, we obtain the result.
	\end{proof}

	\subsection{Proof of Theorem \ref{thm:cue_geom}}\label{sec:CUEproofs}
	We start by proving Lemma \ref{lem:cue_rigidity}. For economy of notation, we denote $\Delta_n = \max_{k \in [n-1]} |W_{n,k} - k/n|$.
	 \begin{proof}[Proof of Lemma \ref{lem:cue_rigidity}] By the triangle inequality, we can bound
		\begin{equation}\label{eq:CUEDel}
			\Delta_n = \max_{k \in [n-1]} \left| \frac{\theta_{k+1} - \theta_1}{2\pi} - \frac{k}{n} \right| \le \max_{k \in [n-1]} \left| \frac{\theta_{k+1}}{2\pi} - \frac{k+1}{n} \right| + \left| \frac{\theta_1}{2\pi} - \frac{1}{n} \right| \le 2 \max_{j \in [n]} \left| \frac{\theta_j}{2\pi} - \frac{j}{n} \right|.
		\end{equation}
		Let $\mu_n$ denote the empirical spectral measure of the CUE matrix, and $\nu$ denote the uniform probability measure on the unit circle in the complex plane. By \cite[Theorem 1]{MM19}, we have 
		\begin{equation}\label{eq:Meckes}
			\E[d_K(\mu_n, \nu)] \le c \frac{\log n}{n}
		\end{equation}
		for a universal constant $c \in (0,\infty)$; here $d_K(\mu_n, \nu)$ denotes the Kolmogorov distance. Let $\mathcal{N}_s$ count the number of eigenphases in the interval $(0, 2\pi s]$. Since under the Haar measure, $\theta_1>0$ and $\theta_i \neq \theta_j$ for $1 \le i \neq j \le n$ almost surely, we have $\mathcal{N}_{\theta_j/(2\pi)} = j$, and therefore we can bound
		$$
		\max_{j \in [n]} \left| \frac{\theta_j}{2\pi} - \frac{j}{n} \right| = \max_{j \in [n]} \left|\frac{\theta_j}{2\pi} - \frac{\mathcal{N}_{\theta_j/(2\pi)}}{n}\right| \le \sup_{0 \le \theta < 2\pi} \left| \frac{\mathcal{N}_{\theta/(2\pi)}}{n} - \frac{\theta}{2\pi} \right| = d_K(\mu_n, \nu),
		$$
		Taking an expectation in \eqref{eq:CUEDel} now yields via \eqref{eq:Meckes} that $\E[\Delta_n] \le 2 \E[d_K(\mu_n, \nu)] = 2c \frac{\log n}{n}$.
	\end{proof}

By the translation invariance of the Haar measure, evaluating the normalized gap $2\pi W_{n,k} = \theta_{k+1} - \theta_1$, $k \in [n-1]$ is equivalent to conditioning the point process of eigenphases on $[0,2\pi)$ to have an eigenphase exactly at the origin and measuring the distance to the $k$-th subsequent point. Therefore, the distribution of $W_{n,k}$ is governed by the \textit{reduced} Palm measure of the CUE spectra, which we denote as $\P_0$. In particular, the probability that $W_{n,k} \le s$ is the same as the probability that there are at least $k$ points in the interval $(0,2\pi s]$ under $\P_0$. Denoting by $\mathcal{N}_s$ the number of eigenphases in the interval $(0, 2\pi s]$, this translates to
\begin{equation}\label{eq:palm}
	\P(W_{n,k} \le s) = \P_0(\mathcal{N}_s \ge k), \quad \text{for $k \in [n-1]$ and $s \in (0,\infty)$.}
\end{equation}

The following lemma, proved in Appendix \ref{app:cuelb}, estimates the mean of $\mathcal{N}_s$ under $\P_0$.
	\begin{lemma}\label{lem:palm_mean_bounds}
		Let $\mathcal{N}_s$ be as above and let $\mu(s) := \E_0[\mathcal{N}_s]$ be its expectation under the reduced Palm measure. Then, for $s \in (0,1)$, there exists a universal constant $C' \in (0,\infty)$ such that
		\begin{equation}\label{eq:meso-micro}
			\mu(s) \le \min \{ns, C' n^3 s^3\}.
		\end{equation}
	\end{lemma}

	Equipped with these analytical bounds, we can bound the inverse moments of the weights, necessary for the proof of Theorem \ref{thm:cue_geom}.

	\begin{lemma}\label{lem:cue_inverse}
		For the normalized gaps $W_{n,k} = \frac{1}{2\pi}(\theta_{k+1} - \theta_1)$, there exists a universal constant $C \in (0,\infty)$ such that for all $k \in [n-1]$,
		$$
		\E[W_{n,k}^{-1}] \le C\; \frac{n}{k}.
		$$
	\end{lemma}
	
	\begin{proof}
		Fix $k \in [n-1]$. Denoting $s_0 = k/(en) \in (0,1)$, by \eqref{eq:palm} we obtain
		\begin{equation}\label{eq:CUEint1}
			\E[W_{n,k}^{-1}] = \int_0^\infty \P(W_{n,k} \le s) \frac{1}{s^2} \dd s = \int_0^\infty \frac{\P_0(\mathcal{N}_s \ge k)}{s^2} \dd s \le \int_0^{s_0} \frac{\P_0(\mathcal{N}_s \ge k)}{s^2} \dd s  + \frac{en}{k}.
		\end{equation}
		The unconditioned CUE eigenphases form a determinantal point process (DPP) (see \cite[Section 2.3]{BA}). By \cite[Theorem 1.7]{ST2003}, we have that the reduced Palm measure of any DPP is itself a DPP. 
		Consequently, by \cite[Theorem 7]{Hough}, the number of eigenphases $\mathcal{N}_s$ in $(0,2\pi s]$ under $\P_0$ is distributed as a sum of independent Bernoulli random variables. Therefore, by the multiplicative Chernoff bound (the bound being trivial for $k\le \mu_s$), we obtain
		\begin{equation}\label{eq:Chernoff}
			\P_0(\mathcal{N}_s \ge k) \le (e \mu(s)/k)^k, 
		\end{equation}
		where $\mu(s) = \E_0[\mathcal{N}_s]$. Writing $s_1 = \min\{k/(en), 1/n\} \in (0,1)$, applying the first and the second bounds in \eqref{eq:meso-micro} from Lemma \ref{lem:palm_mean_bounds} on the intervals $(s_1, s_0]$ and $(0,s_1]$, respectively, \eqref{eq:Chernoff} yields
		\begin{align*}
			\int_0^{s_0} \frac{\P_0(\mathcal{N}_s \ge k)}{s^2} \dd s &\le \int_0^{s_1} \left( \frac{e C' n^3 s^3}{k} \right)^k \frac{1}{s^2} \dd s + \mathds{1}(k \ge 2)\int_{s_1}^{s_0} \left( \frac{e n s}{k} \right)^k \frac{1}{s^2} \dd s \\
			&\le \left( \frac{e C' n^3}{k} \right)^k \frac{n^{-(3k-1)}}{3k-1} + \mathds{1}(k \ge 2) \left( \frac{e n}{k} \right)^k \frac{s_0^{k-1}}{k-1} \le C''\frac{n}{k}
		\end{align*}
		for some universal constant $C''\in (0,\infty)$. The result now follows from \eqref{eq:CUEint1}.
	\end{proof}

\begin{proof}[Proof of upper bounds in Theorem \ref{thm:cue_geom}] We start by bounding the Kolmogorov distance. Since the weights are ordered and $X_k$'s are conditionally independent, the sum $S_n$ satisfies assumptions \ref{a2} and \ref{a1} of Theorem \ref{thm:randrates_cond}. Although $S_n$ contains $n-1$ summands, we slightly abuse notation by writing $n$ in place of $n-1$ throughout the subsequent proof. We estimate each term of the aggregate error $\eps_n(\mathbf{W})$ defined in \eqref{eq:eps_n_random}. We can bound
	\begin{equation}\label{eq:CUEbd1}
		\delta_n = \max_{k \in [n-1]} (W_{n,k} - W_{n,k-1}) \le 2\Delta_n + \frac{1}{n}, \quad \text{and} \quad |W_{n,n-1} - 1| \le \Delta_n + \frac{1}{n},
	\end{equation}
	where $\Delta_n = \max_{k \in [n-1]} |W_{n,k} - k/n|$. Second, noting $\lambda_j(\mathbf{W}) = \frac{\theta}{n W_{n,j}}$, we have
	\begin{equation}\label{eq:CUEbd2}
		D_n^*  = \max_{k \in [n-1]} \bigg|\sum_{j=1}^k (\lambda_j(\mathbf{W}) W_{n,j} - \theta(W_{n,j} - W_{n,j-1}))\bigg|= \theta \max_{k \in [n-1]} \left| \frac{k}{n} - W_{n,k} \right| = \theta \Delta_n.
	\end{equation}
	Finally, we estimate the third component in $\eps_n(\mathbf{W})$ in \eqref{eq:eps_n_random}. Conditionally on $\mathbf{W}$, $X_k$ has success probability $p_k(\mathbf{W}) = (1 + \lambda_k(\mathbf{W}))^{-1}$, so that $x_{\min}^{(k)}(\mathbf{W}) = 0$, and $\P(X_k = 0 \mid \mathbf{W}) = p_k(\mathbf{W})$. Moreover, since a geometric random variable with success probability $p$ is equidistributed as $NB(1,p)$, by Lemma \ref{lem:nb_size_bias}, we have
	$$
	d_{TV}(X_k^*, X_k+1 \mid \mathbf{W}) \le 1 - p_k(\mathbf{W}) = \frac{\lambda_k(\mathbf{W})}{1 + \lambda_k(\mathbf{W})}.
	$$
	Substituting these, we obtain
	\begin{equation}\label{eq:CUEbd3}
		\sum_{k=1}^{n-1} \lambda_k(\mathbf{W}) W_{n,k} \, \frac{d_{TV}(X_k^*, X_k+1 \mid \mathbf{W})}{\P(X_k = x_{\min}^{(k)}(\mathbf{W}) \mid \mathbf{W})}  \le \sum_{k=1}^{n-1} \left( \frac{\theta}{n W_{n,k}} \right)^2 W_{n,k} = \frac{\theta^2}{n^2} \sum_{k=1}^{n-1} \frac{1}{W_{n,k}}.
	\end{equation}
	Combining the bounds in \eqref{eq:CUEbd1}, \eqref{eq:CUEbd2} and \eqref{eq:CUEbd3}, and taking expectation, we obtain
	$$
	\E [\eps_n(\mathbf{W})] \le \frac{2}{n} + 5 \E [\Delta_n] + \frac{2 \theta}{n^2} \sum_{k=1}^{n-1} \E\left[\frac{1}{W_{n,k}}\right].
	$$
	By Lemma \ref{lem:cue_rigidity}, we have $ \E [\Delta_n] = \mathcal{O}(\log n/n)$. On the other hand, by Lemma \ref{lem:cue_inverse}, we can bound
	$$
	\sum_{k=1}^{n-1} \E\left[\frac{1}{W_{n,k}}\right] \le  Cn \sum_{k=1}^{n-1} \frac{1}{k} = \mathcal{O}(n \log n).
	$$ 
	This yields $\E[\eps_n(\mathbf{W})] = \mathcal{O}(\frac{\log n}{n})$. An application of Theorem \ref{thm:randrates_cond} now establishes the asserted upper bound on $d_K(S_n,D_\theta)$.

	Finally, we establish the upper bound for the $d_{1,1}$ distance. Since the weights are ordered, we apply Theorem \ref{thm:d11_bound} using the canonical grid choice $t_k = W_{n,k}$ for $k \in [n-1]$, with $t_0 = 0$. This ensures that $|W_{n,k} - t_k| = 0$, and $\delta_n^{(t)} = \delta_n$ and $D_n^{(t)*} = D_n^*$.
	
	Also, because $0 \le \theta_1 < \dots < \theta_n < 2\pi$, we have $t_{n-1} = W_{n,n-1} = \frac{\theta_n - \theta_1}{2\pi} < 1$. Thus, $|t_{n-1} - 1| = 1 - W_{n,n-1}$. Bounding the aggregate error $\eps_n^{1,1}(\mathbf{W})$ from \eqref{eq:eps_n_d11} therefore yields
	\begin{equation}\label{eq:CUE_d11_aggregate}
		\eps_n^{1,1}(\mathbf{W}) \le \frac{\theta}{4} \delta_n + \frac{3}{2} D_n^* + \theta |W_{n,n-1} - 1| + \frac{1}{2} \sum_{k=1}^{n-1} \lambda_k(\mathbf{W}) W_{n,k}^2 \, d_1(X_k^*, X_k+1 \mid \mathbf{W}).
	\end{equation}
	
	To bound the size-bias coupling term, recall that conditionally given $\mathbf{W}$, $X_k$ is geometric. By Lemma \ref{lem:nb_size_bias}, we have $X_k^* \stackrel{d}{=} X_k + Y_k$, where $Y_k \in \N$ is conditionally independent of $X_k$, and $\E[Y_k \mid \mathbf{W}] = 1 + \lambda_k(\mathbf{W})$. Thus, we can bound
	\begin{equation}\label{eq:CUE_d11_sizebias}
		d_1(X_k^*, X_k+1 \mid \mathbf{W}) \le \E[|Y_k - 1| \mid \mathbf{W}] = \E[Y_k \mid \mathbf{W}] - 1 = \lambda_k(\mathbf{W}).
	\end{equation}
	Substituting \eqref{eq:CUE_d11_sizebias} into the sum in \eqref{eq:CUE_d11_aggregate}, and recalling $\lambda_k(\mathbf{W}) = \frac{\theta}{n W_{n,k}}$, we have
	\begin{equation*}\label{eq:CUE_d11_sum_bound}
		\frac{1}{2} \sum_{k=1}^{n-1} \lambda_k(\mathbf{W}) W_{n,k}^2 \, d_1(X_k^*, X_k+1 \mid \mathbf{W}) \le \frac{1}{2} \sum_{k=1}^{n-1} \lambda_k(\mathbf{W})^2 W_{n,k}^2 = \frac{1}{2} \sum_{k=1}^{n-1} \left( \frac{\theta}{n} \right)^2 \le \frac{\theta^2}{2n}.
	\end{equation*}
	By \eqref{eq:CUE_d11_aggregate}, and the bounds in \eqref{eq:CUEbd1} and \eqref{eq:CUEbd2}, we thus obtain
	\begin{align*}
		\eps_n^{1,1}(\mathbf{W}) &\le \frac{\theta}{4} \left( 2\Delta_n + \frac{1}{n} \right) + \frac{3\theta}{2} \Delta_n + \theta \left( \Delta_n + \frac{1}{n} \right) + \frac{\theta^2}{2n} = 3\theta \Delta_n + \left( \frac{5\theta}{4} + \frac{\theta^2}{2} \right) \frac{1}{n}.
	\end{align*}
	Theorem \ref{thm:d11_bound} now yields the upper bound noting by Lemma \ref{lem:cue_rigidity} that $ \E [\Delta_n] = \mathcal{O}(\log n/n)$, concluding the proof.
\end{proof}

Note that unlike the Kolmogorov bound, the smooth $d_{1,1}$ distance bound doesn't require the finite inverse moments in Lemma \ref{lem:cue_inverse}.
Finally, we prove the lower bounds in Theorem \ref{thm:cue_geom}. For this we need the following lemma, whose proof can be found in Appendix \ref{app:cuelb}.

\begin{lemma}\label{lem:laplace_Sn}
	Let $S_n = \sum_{k=1}^{n-1} W_{n,k} X_k$ be as in Theorem \ref{thm:cue_geom}. Then for any $\theta > 0$, there exists a constant $C \in (0,\infty)$ depending only on $\theta$ such that
	$$
	\E[e^{-S_n}] \ge \E[e^{-D_\theta}] \left( 1 + \frac{C}{n}\right).
	$$
\end{lemma}

\begin{proof}[Proof of lower bounds in Theorem \ref{thm:cue_geom}]
	We first prove the lower bound for the Kolmogorov distance. Let $\theta \ge 1$. Denote $F_n(x) = \P(S_n \le x)$, $F(x) = \P(D_\theta \le x)$, and $L_D = \E[e^{-D_\theta}]$. Integrating by parts, we can bound
	$$
	\E[e^{-S_n}] - \E[e^{-D_\theta}] = \int_0^\infty \big( F_n(x) - F(x) \big) e^{-x} \dd x \le \sup_{x \ge 0} |F_n(x) - F(x)| \int_0^\infty e^{-x} \dd x = d_K(S_n, D_\theta).
	$$
	Therefore, by Lemma \ref{lem:laplace_Sn}, there exists a constant $C \in (0,\infty)$ depending only on $\theta$ such that
	$$
	d_K(S_n, D_\theta) \ge L_D \left( 1 + \frac{C}{n} \right) - L_D =  \frac{L_D C}{n},
	$$
	yielding the desired lower bound for $\theta \ge 1$.
	
	Next, fix $\theta \in (0,1)$. Note that $\P(D_\theta \le 0) = 0$ so that $d_K(S_n,D_\theta) \ge \P(S_n \le 0)$. Since the weights $W_{n,k}$ are almost surely positive for $k \in [n-1]$, we have that $S_n \le 0$ if and only if $X_k = 0$ for all $k \in [n-1]$. The conditional independence of $(X_k)_{k \in [n-1]}$ thus implies
	$$
	\P(S_n \le 0 \mid \mathbf{W}) = \prod_{k=1}^{n-1} \P(X_k = 0 \mid \mathbf{W}) = \prod_{k=1}^{n-1} \left( 1 + \frac{\theta}{n W_{n,k}} \right)^{-1} = e^{-Z},
	$$
	with $Z = \sum_{k=1}^{n-1} \log\left( 1 + \frac{\theta}{n W_{n,k}} \right)$. By Jensen's inequality, we have
	$$
	\P(S_n \le 0) = \E[e^{-Z}] \ge e^{-\E[Z]}.
	$$
	To upper bound $\E[Z]$, we use the simple identity 
    \begin{equation}\label{eq:logid}
            \log(1 + a/x) = \int_x^\infty \frac{a}{s(s+a)} \, \dd s \quad  a,x > 0.
    \end{equation}
    Applying this to each summand in $Z$ with $a = \theta/n$ and $x=W_{n,k}$, by Tonelli's theorem, we obtain
	\begin{align*}
		\E[Z] &= \E \left[ \sum_{k=1}^{n-1} \int_{W_{n,k}}^\infty \frac{\theta/n}{s(s + \theta/n)} \,\dd s \right] = \int_0^\infty \left(\sum_{k=1}^{n-1} \P(W_{n,k} \le s) \right) \frac{\theta/n}{s(s + \theta/n)} \,\dd s.
	\end{align*}
	Since there are exactly $n-1$ eigenphases in $(0,2\pi]$ under $\P_0$, by \eqref{eq:palm} we have
	$$
    \sum_{k=1}^{n-1} \P(W_{n,k} \le s) = \sum_{k=1}^{n-1} \P_0(\mathcal{N}_s \ge k)  = \sum_{k=1}^{\infty} \P_0(\mathcal{N}_s \ge k) = \E_0[\mathcal{N}_s] = \mu(s).
    $$
	For $s \in (0,1)$, applying the bound $\mu(s) \le ns$ established in \eqref{eq:meso-micro} while or $s \ge 1$, using the trivial bound $\mu(s) \le n$, we obtain
    \begin{align*}
        \E[Z] & \le \int_0^1 (ns) \frac{\theta/n}{s(s + \theta/n)} \dd s + \int_1^\infty n \frac{\theta/n}{s(s + \theta/n)} \dd s\\
        & \qquad = \theta \int_0^1 \frac{1}{s + \theta/n} \dd s + n \log\left( 1 + \frac{\theta}{n} \right) \le \theta \log \frac{1 + \theta/n}{\theta/n} + \theta  = \theta \log\left( \frac{n}{\theta} + 1\right) + \theta,
    \end{align*}
	where the second step is due to \eqref{eq:logid} and the penultimate step used the simple inequality $\log (1+x) \le x$ for $x \ge 0$. An application of Jensen's inequality now yields
	$$
	d_K(S_n, D_\theta) \ge \P(S_n \le 0) \ge e^{-\E [Z]}\ge \exp\Big( - \theta \log\left( \frac{n}{\theta} + 1 \right) - \theta \Big) = e^{-\theta} \left( \frac{n}{\theta} + 1 \right)^{-\theta} = \Omega(n^{-\theta}),
	$$
	concluding the proof for $\theta \in (0,1)$, and therefore proving the asserted lower bound on $d_K(S_n, D_\theta)$.

	To establish the lower bound for the $d_{1,1}$ distance, we again consider the identity function $h(x) = x$. Since $h \in \mathcal{H}_{1,1}$ and $\lambda_k(\mathbf{W}) = \frac{\theta}{n W_{n,k}}$, we can evaluate
	\begin{equation*}
		\E[S_n] = \E \left[ \sum_{k=1}^{n-1} W_{n,k} \E[X_k \mid \mathbf{W}] \right] = \E \left[ \sum_{k=1}^{n-1} W_{n,k} \frac{\theta}{n W_{n,k}} \right] = \sum_{k=1}^{n-1} \frac{\theta}{n} = \theta - \frac{\theta}{n}.
	\end{equation*}
	Noting $\E[D_\theta] = \theta$, this immediately yields
	\begin{equation*}
		d_{1,1}(S_n, D_\theta) = \sup_{f \in \mathcal{H}_{1,1}} \big|\E[f(S_n)] - \E[f(D_\theta)]\big| \ge \big|\E[h(S_n)] - \E[h(D_\theta)]\big| = \left| \theta - \frac{\theta}{n} - \theta \right| = \frac{\theta}{n}
	\end{equation*}
	proving the desired $d_{1,1}$ lower bound.
\end{proof}

\subsection{Proof of Theorem \ref{thm:poisson_app}}\label{sec:proofpoi}
Before we proceed to the proof of Theorem \ref{thm:poisson_app}, we first note the following simple fact: For $y_1, \dots, y_n \in \R$, let $y_{(1)} \le \dots \le y_{(n)}$ denote their order statistics. Then,
	\begin{equation}\label{eq:sorting_nonexpansive}
			\max_{k \in [n]} |y_{(k)} - k| \le \max_{k \in [n]} |y_k - k|.
	\end{equation}
	Indeed, denoting $\Delta = \max_{k \in [n]} |y_k - k|$, 
	note that for any fixed $k \in [n]$, and $i \le k$, we have $y_i \le i + \Delta \le k + \Delta$, 
	yielding $y_{(k)} \le k + \Delta$. By symmetric reasoning, note that $y_i \ge i - \Delta \ge k - \Delta$ for any $i \ge k$, 
	so that $y_{(k)} \ge k - \Delta$, proving \eqref{eq:sorting_nonexpansive}.

\begin{proof}[Proof of Theorem \ref{thm:poisson_app}]
	The sequence $(X_k)_{k \in \N}$ consists of independent Poisson random variables with means $\lambda_k(\mathbf{W}) \equiv \lambda_k= \E[X_k] = \theta/k$, which trivially satisfies assumption \ref{a1}. We can therefore apply Theorem \ref{thm:randrates_unordered}. Let $W_{n,(1)} \le \dots \le W_{n,(n)}$ be the ordered weights, and let $\sigma$ be the permutation such that $W_{n,(k)} = W_{n,\sigma(k)}, k \in [n]$. Since $X_k$ is Poisson distributed, we have $X_k^* \stackrel{d}{=} X_k + 1$, so that $d_{TV}(X_k^*, X_k + 1 \mid \mathbf{W}) = d_{TV}(X_k^*, X_k + 1)= 0$ for $k \in [n]$, and hence the third term in the aggregate error \eqref{eq:eps_n_unordered} vanishes. Theorem \ref{thm:randrates_unordered} therefore yields
	\begin{equation}\label{eq:Pdcond}
		d_K(S_n, D_\theta) \le \begin{cases} 
			C_1 \E[\eps_n^{\text{os}}], & \text{if } \theta \ge 1, \\ 
			2\theta^2 \E[\eps_n^{\text{os}}] + C_2 \left(\E[\eps_n^{\text{os}}]\right)^\theta, & \text{if } \theta \in (0,1),
		\end{cases}
	\end{equation}
	with 
	\begin{equation}\label{eq:Peps}
		\eps_n^{\text{os}} := \delta_n^{\text{os}} + \frac{2}{\theta} \max_{k \in [n]} |D_{n,k}^{\text{os}}| + |W_{\max}^{(n)} - 1|,
	\end{equation}
	where
	$$
	\delta_n^{\text{os}} = \max_{k \in [n]} ({W}_{n,(k)} - {W}_{n,(k-1)}), \quad \text{and} \quad D_{n,k}^{\text{os}} := \sum_{j=1}^k (\lambda_{\sigma(j)} W_{n,(j)} - \theta (W_{n,(j)} - W_{n,(j-1)})).
	$$
	
	Let $(Y_{(k)})_{k \in [n]}$ denote the order statistics of $(Y_k)_{k \in [n]}$, so that $W_{n,(k)} = Y_{(k)}/ n$. By \eqref{eq:sorting_nonexpansive}, we have $\max_{k \in [n]} |Y_{(k)} - k| \le \Delta_n$, where we define $\Delta_n := \max_{k \in [n]} |Y_k - k|$.
	Bounding $Y_{(k)} - Y_{(k-1)} \le |Y_{(k)} - k| + |Y_{(k-1)} - (k-1)| + 1 \le 2\Delta_n + 1$, we thus obtain
	\begin{equation}\label{eq:Pdel}
		\delta_n^{\text{os}} \le \frac{1}{n}(2\Delta_n + 1).
	\end{equation}
	On the other hand, since $W_{\max}^{(n)} = \frac{1}{n} Y_{(n)}$, one obtains directly from \eqref{eq:sorting_nonexpansive} that
	\begin{equation}\label{eq:PW}
		|W_{\max}^{(n)}- 1| = \frac{1}{n} |Y_{(n)} - n| \le \frac{1}{n} \Delta_n.
	\end{equation}
	Finally, we bound the discrepancy term $\max_{k \in [n]} |D_{n,k}^{\text{os}}|$. Observe that $\sum_{j=1}^k \theta(W_{n,(j)} - W_{n,(j-1)}) = \theta W_{n,(k)} = \frac{\theta}{n} Y_{(k)}$. Noting that $\lambda_{\sigma(j)} = \frac{\theta}{\sigma(j)}$, this yields
	$$
	\frac{2}{\theta} \max_{k \in [n]} |D_{n,k}^{\text{os}}|  = \frac{2}{n} \max_{k \in [n]} \Bigg| \sum_{j=1}^k \frac{Y_{\sigma(j)}}{\sigma(j)} - Y_{(k)} \Bigg| \le \frac{2}{n} \max_{k \in [n]} \Bigg( |k - Y_{(k)}| + \sum_{j=1}^k \frac{|Y_{\sigma(j)} - \sigma(j)|}{\sigma(j)} \Bigg),
	$$
	where the final step follows upon writing $Y_{\sigma(j)} = \sigma(j) + (Y_{\sigma(j)} - \sigma(j))$ and applying the triangle inequality. Noting that $\{\sigma(1), \dots, \sigma(k)\}$ are $k$ distinct positive integers in $[n]$, we have $\sum_{j=1}^k \sigma(j)^{-1} \le H_n$, where for $j \in [n]$, we write $H_j := \sum_{i=1}^j (1/j)$. Thus, \eqref{eq:sorting_nonexpansive} yields
	\begin{equation}\label{eq:PD}
		\frac{2}{\theta} \max_{k \in [n]} |D_{n,k}^{\text{os}}|  \le \frac{2}{n} \left( \Delta_n + \Delta_n H_n \right).
	\end{equation}
	Combining the bounds \eqref{eq:Pdel}, \eqref{eq:PW} and \eqref{eq:PD} via \eqref{eq:Peps}, the error $\eps_n^{\text{os}}$ is bounded by
	\begin{equation}\label{eq:poisson_cond_bound}
		\eps_n^{\text{os}} \le \frac{1}{n} \Big( (2\Delta_n + 1) + \Delta_n + 2\Delta_n(1 + H_n) \Big) = \frac{1}{n} \Big( 1 + \Delta_n(5 + 2H_n) \Big).
	\end{equation}
	
	Finally, we bound $\E[\Delta_n]$, thereby bounding $\E[\eps_n^{\text{os}}]$. If we only assume \eqref{eq:pmom} for $p \ge 1$, since the $L^\infty$-norm is always upper bounded by the $L^p$-norm, Jensen's inequality yields
	$$
	\E[\Delta_n] = \E\left[ \max_{k \in [n]} |Y_k - k| \right] \le \E\Bigg[ \Big( \sum_{k=1}^n |Y_k - k|^p \Big)^{1/p} \Bigg] \le \Bigg( \sum_{k=1}^n \E[|Y_k - k|^p] \Bigg)^{1/p} \le \Bigg( \sum_{k=1}^n M_{k,p} \Bigg)^{1/p},
	$$
	yielding the first assertion by \eqref{eq:Pdcond} and \eqref{eq:poisson_cond_bound}. 
	
	Next, assume that the $(Y_k)_{k \in [n]}$ are sub-Gaussian satisfying \eqref{eq:subgaussian_tail}. For $t \ge 0$, the union bound yields
	$$
	\P(\Delta_n \ge t) = \P\left( \max_{k \in [n]} |Y_k - k| \ge t \right) \le \sum_{k=1}^n \P(|Y_k - k| \ge t) \le 2n \exp\left( - \frac{t^2}{C n} \right).
	$$
	Choosing $t_0 = \sqrt{Cn \log n}$, we thus obtain
	\begin{align*}
		\E[\Delta_n] &= \int_0^{t_0} \P(\Delta_n \ge t) \dd t + \int_{t_0}^\infty \P(\Delta_n \ge t) \dd t \\
		&\le t_0 + \int_{t_0}^\infty 2n \exp\left( - \frac{t^2}{Cn} \right) \dd t \le t_0 + 2n \int_{t_0}^\infty \frac{t}{t_0} \exp\left( - \frac{t^2}{Cn} \right) \dd t = t_0 + \frac{n^2 C}{t_0} \exp\left( - \frac{t_0^2}{Cn} \right).
	\end{align*}
	Substituting $t_0 = \sqrt{Cn \log n}$ now yields
	$$
	\E[\Delta_n] \le \sqrt{Cn \log n} + \frac{n^2 C}{\sqrt{Cn \log n}} \cdot \frac{1}{n} = \mathcal{O}\big( \sqrt{n \log n} \big).
	$$
	Because $H_n = \mathcal{O}(\log n)$, we substitute the bound for $\E[\Delta_n]$ into \eqref{eq:poisson_cond_bound} to obtain
	$$
	\E[\eps_n^{\text{os}}]= \mathcal{O}\left( \frac{(\log n)^{3/2}}{\sqrt{n}} \right),
	$$
	yielding the second assertion via \eqref{eq:Pdcond}.
\end{proof}

\section*{Acknowledgements}
\addcontentsline{toc}{section}{Acknowledgements}

 The research was supported in part by the German Research Foundation (DFG) Project 531540467.

	\printbibliography

	\appendix

    %
%
\section{Proof of Lemma \ref{lem:snlb}}\label{app:lemlbpf}
\subsection{The case of the measure $\mathbb{P}_n$: Proof of \eqref{eq:snlbint}}
	We fix $\theta>1$. The event $S_n \le 1$ corresponds to the condition that the randomly sampled integer is smaller than $p_n$. Thus,
	\begin{equation}\label{eq:Sle1}
		\mathbb{P}(S_n \le 1) = \frac{1}{Z_n} \sum_{m=1}^{p_n} \frac{\tau_\theta(m)}{m} =: \frac{S(p_n)}{Z_n}.
	\end{equation}
	Next, we establish the asymptotic orders of the sum $S(p_n)$ and the normalizing constant $Z_n$. Below, we set $x=p_n$ for simplicity. 
	
	\medskip
	
	\noindent\underline{\textit{Asymptotics of the sum in \eqref{eq:Sle1}:}} 
	Let $T(x) = \sum_{1 \le m\le x} \tau_\theta(m)$ and $S(x)=\sum_{1 \le m\le x}\frac{\tau_\theta(m)}{m}$. We will first estimate $T(x)$ using the Selberg-Delange method and then use the Abel summation formula to transfer to $S(x)$. By \cite[Ch.\ II.5, Eqs.\ (5.2) and (5.3)]{tet}, the Dirichlet series (see \cite[Ch.\ I.2, Definition 2.3]{tet}) of $\tau_\theta$ is given by
	$$
	F(s) = \sum_{m=1}^\infty \tau_\theta(m) m^{-s} = \zeta(s)^\theta, \quad s \in \mathbb{C} \text{ with } \Re (s)>1,
	$$
	where $\Re (s)$ denotes the real part of $s$. To extract the asymptotic behavior of its coefficients, we apply the Selberg-Delange method (see \cite[Ch.\ II.5, Theorem 5.4]{tet}). We can trivially factor the Dirichlet series as $F(s) = \zeta(s)^\theta G(s; \theta)$, by taking (see \cite[Ch.\ II.5, Eq.\ (5.32)]{tet})
	$$
	G(s; \theta) \equiv 1 =: \sum_{n=1}^\infty g_\theta(n)/n^s,
	$$
	where we trivially have $g_\theta(1) = 1$ and $g_\theta(n) = 0$ for all $n \ge 2$. Consequently, the absolute convergence condition in \cite[Ch.\ II.5, Eq.\ (5.33)]{tet} is trivially satisfied for $z=\theta$ and $N=1$, since
	$$
	\sum_{n = 1}^\infty \frac{|g_\theta(n)|}{n} (\log 3n)^{N+1+\max(1-\theta,\, 0)} = (\log 3)^{2} < \infty.
	$$
	Thus, since $F(s)$ is convergent for $\Re (s)>1$, \cite[Ch.\ II.5, Theorem 5.4]{tet} guarantees an asymptotic expansion of $T$ determined by the Taylor coefficients $\gamma_j(\theta) = Z^{(j)}(1, \theta), j \ge 0$ of the function $Z(s;\theta) = \big((s-1)\zeta(s)\big)^\theta/s$ around $s=1$ \cite[Ch.\ II.5, Eqs.\ (5.5) and (5.6)]{tet}. Specifically, applying \cite[Ch.\ II.5, Theorem 5.4 and Eq.\ (5.16)]{tet} with $N=1$ and $A=\theta$, we obtain
	\begin{equation}\label{eq:Tintermed}
		T(x) = x (\log x)^{\theta-1} \left(\lambda_0(\theta) + \frac{\lambda_1(\theta)}{\log x} \right) + \mathcal{O}\big(x (\log x)^{\theta - 3}\big).
	\end{equation}
	The expansion coefficients are given by (see \cite[Ch.\ II.5, Eq.\ (5.13)]{tet} with $G\equiv 1$)
	\begin{equation}\label{eq:gk}
		\lambda_k(\theta) = \frac{1}{\Gamma(\theta - k)} \frac{\gamma_k(\theta)}{k!}= \frac{1}{\Gamma(\theta - k)} \frac{Z^{(k)}(1, \theta)}{k!}.
	\end{equation}
	To evaluate $\gamma_j(\theta) = Z^{(j)}(1, \theta)$ for $j = 0,1$, we note two classical limit evaluations of the Riemann zeta function, namely, $\lim_{s \to 1} (s-1)\zeta(s) = 1$ and $\frac{d}{ds}\big((s-1)\zeta(s)\big)\big|_{s=1} = \gamma$. Thus, 
	\begin{equation}\label{eq:Z}
		Z(1; \theta) = \lim_{s \to 1}  ((s-1)\zeta(s))^\theta = 1 \quad \text{and} \quad Z'(1; \theta) = -1 + \theta\gamma.
	\end{equation}
	From \eqref{eq:gk}, this yields $\lambda_0(\theta) = \frac{1}{\Gamma(\theta)}$ while
	$
	\lambda_1(\theta) = \frac{\theta\gamma - 1}{\Gamma(\theta - 1)} = \frac{(\theta-1)(\theta\gamma - 1)}{\Gamma(\theta)}
	$. 
	Thus we obtain from \eqref{eq:Tintermed} that
	\begin{equation}\label{eq:T}
		T(x) = \frac{x\,(\log x)^{\theta-1}}{\Gamma(\theta)} \Bigl(1 + \frac{\alpha_1}{\log x} + R(\log x) \Bigr)\Bigr),
	\end{equation}
	where $\alpha_1 = (\theta-1)(\theta\gamma - 1)$ and the remainder term satisfies $|R(u)| \le C u^{-2}, u >0$, for some constant $C \in (0,\infty)$ depending only on $\theta$. 
	
	Next, noting that $T(t) = \tau_\theta(1)=1$ for $t \in [1,2)$, by Abel's summation formula, we obtain
	\begin{equation}\label{eq:Sintermed}
		S(x) = \frac{T(x)}{x} + \int_1^x \frac{T(t)}{t^2} \dd t = \frac{T(x)}{x} + \int_2^x \frac{T(t)}{t^2} \dd t + \frac{1}{2}.
	\end{equation}
	We substituting the expansion of $T(x)$ from \eqref{eq:T} into this identity. The first term thus yields
	$$
	\frac{T(x)}{x} = \frac{(\log x)^{\theta-1}}{\Gamma(\theta)} \Bigl(1 + \mathcal{O}\Bigl(\frac{1}{\log x}\Bigr)\Bigr).
	$$
	Focusing next on the integral term in \eqref{eq:Sintermed}, applying the change of variables $u = \log t$, we obtain
	\begin{align}\label{eq:intint}
	\int_2^x \frac{T(t)}{t^2} \dd t &= \int_{\log 2}^{\log x} \frac{u^{\theta-1}}{\Gamma(\theta)} \Bigl(1 + \alpha_1 u^{-1} + R(u)\Bigr) \dd u \nonumber\\
	&= \frac{(\log x)^\theta - (\log 2)^\theta}{\Gamma(\theta+1)} + \frac{1}{\Gamma(\theta)} \int_{\log 2}^{\log x} \Bigl(\alpha_1 u^{\theta-2} + u^{\theta-1} R(u)\Bigr) \dd u.
	\end{align}
	For the subsequent integral in \eqref{eq:intint}, we can evaluate
	\begin{align*}
		\frac{\alpha_1}{\Gamma(\theta)} \int_{\log 2}^{\log x} u^{\theta-2} \dd u &= \frac{(\theta-1)(\theta\gamma - 1)}{\Gamma(\theta)} \frac{(\log x)^{\theta-1} - (\log 2)^{\theta-1}}{\theta-1} = \frac{\theta\gamma - 1}{\Gamma(\theta)}\big[(\log x)^{\theta-1} - (\log 2)^{\theta-1}\big].
	\end{align*}
	On the other hand, noting $|R(u)| \le C u^{-2}$ for $u>0$, we can bound
	$$
	\frac{1}{\Gamma(\theta)} \int_{\log 2}^{\log x} u^{\theta-1} |R(u)| \dd u \le \frac{C}{\Gamma(\theta)} \int_{\log 2}^{\log x} u^{\theta-3} \dd u = \mathcal{O}\big((\log x)^{\theta-2}\big) + \mathcal{O}(\log \log x) \mathds{1}_{\{\theta=2\}}.
	$$
	Combining with the bounds obtained above, from \eqref{eq:Sintermed} we obtain
	\begin{align}\label{eq:S}
		S(x) &= \frac{(\log x)^\theta}{\Gamma(\theta+1)} + \frac{\theta\gamma}{\Gamma(\theta)}(\log x)^{\theta-1} + \mathcal{O}\big((\log x)^{\theta-2}\big) + \mathcal{O}(\log \log x) \mathds{1}_{\{\theta=2\}} + K \nonumber\\
		& = \frac{(\log x)^\theta}{\Gamma(\theta+1)}\Bigl(1+\frac{c_\theta}{\log x}+\mathcal{O}\Bigl(\frac{\log \log x}{(\log x)^\eta}\Bigr)\Bigr),
	\end{align}
	where $K \in \mathbb{R}$ is a constant depending only on $\theta$, $c_\theta = \theta^2 \gamma$ and $\eta = \min\{2, \theta\}$. \medskip
	
	\noindent\underline{\textit{Asymptotics of the normalising constant in \eqref{eq:Sle1}:}} 
	Recall from \eqref{eq:Zn} that $Z_n = \prod_{p \le x}\Bigl(1-\frac1p\Bigr)^{-\theta}$, here and throughout this section, $\prod_p$ and $\sum_p$ denote products and sums over primes $p$. We apply the following asymptotic approximation (see e.g.\ \cite[Eqs.\ (2.32) with $\alpha=1$ and (2.9)]{RS62}): for some positive absolute constant $a \in (0,\infty)$ and any $b \in (0,a)$, it holds that
	\begin{align*}
		\prod_{p\le x}\left(1-\frac{1}{p}\right)&=\frac{e^{-\gamma}}{\log x}+\mathcal{O}\left(e^{-a\sqrt{\log x}}\right) \\
		&= \frac{e^{-\gamma}}{\log x} \left(1+\mathcal{O}\left(e^{\log \log x -a\sqrt{\log x}}\right) \right) = \frac{e^{-\gamma}}{\log x} \left(1+\mathcal{O}\left(e^{-b\sqrt{\log x}}\right) \right).
	\end{align*}
	Noting from the generalized binomial theorem that $\left(1+\mathcal{O}\left(e^{-b\sqrt{\log x}}\right) \right)^{-\theta} = 1+ \mathcal{O}\left(e^{-b\sqrt{\log x}}\right)$,
	the product $Z_n$ can thus be expanded as
	\begin{equation}\label{eq:Znasymp}
		Z_n = \prod_{p\le x}\Bigl(1-\frac1p\Bigr)^{-\theta}
		= (e^\gamma\log x)^\theta \Bigl(1 + \mathcal{O}\big(e^{-b \sqrt{\log x}}\big)\Bigr).
	\end{equation}

	\noindent\underline{\textit{Obtaining \eqref{eq:snlbint}:}} 
	Recalling that $x = p_n$, and combining \eqref{eq:Sle1}, \eqref{eq:S} and \eqref{eq:Znasymp} yields
	$$
	\begin{aligned}
		\mathbb P(S_n\le 1) &= \frac{1}{(e^\gamma\log p_n)^\theta}\Bigl(1 + \mathcal{O}\big(e^{-b\sqrt{\log p_n}}\big)\Bigr)
		\frac{(\log p_n)^\theta}{\Gamma(\theta+1)}\Bigl(1+\frac{c_\theta}{\log p_n}+\mathcal{O}\Bigl(\frac{\log \log p_n}{(\log p_n)^\eta}\Bigr)\Bigr)\\[4pt]
		&= \frac{e^{-\gamma\theta}}{\Gamma(\theta+1)}
		\Bigl(1+\frac{c_\theta}{\log p_n}+\mathcal{O}\Bigl(\frac{\log \log p_n}{(\log p_n)^\eta}\Bigr)\Bigr).
	\end{aligned}
	$$
	Since $\mathbb P(D_\theta\le 1) = \displaystyle e^{-\gamma\theta}/\Gamma(\theta+1)$ (see e.g.\ \cite[Eq.\ (17)]{BS25}), we obtain
	$$
	\mathbb P(S_n\le 1) = \mathbb P(D_\theta\le 1) + \frac{C_\theta}{\log p_n} + \mathcal{O}\Bigl(\frac{\log \log p_n}{(\log p_n)^\eta}\Bigr),
	$$
	where $C_\theta = \frac{\theta^2 \gamma e^{-\gamma\theta}}{\Gamma(\theta+1)}$, yielding the result.

\subsection{The case of the measure $\mathbb{P}_n^{\text{sq}}$: Proof of \eqref{eq:snsqlbint}}
	Fix $\theta>1$. The event $S_n^{\text{sq}} \le 1$ corresponds to the randomly sampled integer being smaller than $p_n$. Thus,
	\begin{equation}\label{eq:Wlesq}
		\mathbb{P}(S_n^{\text{sq}} \le 1) = \frac{1}{Z_n^{\text{sq}}} \sum_{m=1}^{p_n} \frac{\mu^2(m) \theta^{\omega(m)}}{m} =: \frac{S(p_n)}{Z_n^{\text{sq}}},
	\end{equation}
	where $\mu^2(m) = \mathds{1}\{m \text{ is square-free}\}, m \in \N$ is the Möbius function. We again find the asymptotic orders of the sum $S(p_n)$ and the normalizing constant $Z_n^{\text{sq}}$ separately, and set $x=p_n$ below. \medskip
	
	\noindent\underline{\textit{Asymptotics of the sum in \eqref{eq:Wlesq}:}} 
	The argument is similar to the proof for \eqref{eq:snlbint} above. As before, we let $T(x) = \sum_{1 \le m\le x} \mu^2(m)\theta^{\omega(m)}$ and $S(x)=\sum_{1 \le m\le x}\frac{\mu^2(m)\theta^{\omega(m)}}{m}$. The Dirichlet series for the multiplicative arithmetic function $\mu^2(m)\theta^{\omega(m)}$ is given by
	$$
	F(s) = \sum_{m=1}^\infty \mu^2(m)\theta^{\omega(m)} m^{-s} = \prod_p \left( \sum_{k=0}^\infty \frac{\mu^2(p^k)\theta^{\omega(p^k)}}{p^{ks}} \right) =\prod_p \left( \sum_{k=0}^1 \frac{\theta^{\omega(p^k)}}{p^{ks}} \right) = \prod_p (1 + \theta p^{-s}),
	$$ 
	which is absolutely convergent for $\Re (s)>1$. To extract the asymptotic behavior of the coefficients, we again apply the Selberg-Delange method (see \cite[Ch.\ II.5, Theorem 5.4]{tet}). Noting by the Euler product formula that $\zeta(s) = \prod_p \frac{1}{1 - p^{-s}}$ for $\Re(s)>1$, we factor the Dirichlet series as $F(s) = \zeta(s)^\theta G(s; \theta)$ for $\Re(s)>1$, where
	\begin{equation} \label{eq:G_euler}
		G(s; \theta) = \prod_p \big(1 + \theta p^{-s}\big)\big(1 - p^{-s}\big)^\theta =  \prod_p (1 + \theta p^{-s}) \left(\sum_{k=0}^\infty \binom{\theta}{k} (-p^{-s})^k\right);
	\end{equation}
	here the last step is due to the generalized binomial theorem noting $|p^{-s}|<1$. In order to express $G$ as a Dirichlet series, define now the multiplicative arithmetic function $g_\theta$ on $\N$ as $g_\theta(1)=1$ and for prime $p$ and $k \in \N$,
	\begin{equation}\label{eq:gdef}
		g_\theta(p^k) = (-1)^k \binom{\theta}{k} + \theta (-1)^{k-1} \binom{\theta}{k-1}.
	\end{equation}
	Note, for all prime $p$, we have $g_\theta(p)=0$. Also, for $k \ge 2$, we can bound
	\begin{align}\label{eq:gdefmod}
		|g_\theta(p^k)| &\le \left| \binom{\theta}{k} \right| + \theta \left| \binom{\theta}{k-1} \right| \nonumber \\
		&= \frac{|\theta| |\theta-1| \dots |\theta-k+1|}{k!} + \theta \frac{|\theta| |\theta-1| \dots |\theta-k+2|}{(k-1)!} \nonumber \\
		&\le \frac{\lceil\theta\rceil (\lceil\theta\rceil+1) \dots (\lceil\theta\rceil+k-1)}{k!} + \theta \frac{\lceil\theta\rceil (\lceil\theta\rceil+1) \dots (\lceil\theta\rceil+k-2)}{(k-1)!} \nonumber \\
		&= \binom{k + \lceil \theta \rceil - 1}{k} + \theta \binom{k + \lceil \theta \rceil - 2}{k-1} = \mathcal{O}\big(k^{\lceil \theta \rceil - 1}\big) + \theta \cdot \mathcal{O}\big(k^{\lceil \theta \rceil - 1}\big) = \mathcal{O}\big(k^{\lceil \theta \rceil -1}\big).
	\end{align}
	Let $s=\sigma+it$. By the polynomial growth of $g_\theta(p^k)$ in \eqref{eq:gdefmod} and that $g_\theta(p) = 0$ for all prime $p$, there exists an  absolute constant $C \in (0,\infty)$ depending only on $\theta$ such that
	\begin{equation}\label{eq:g1exp}
			\sum_{k= 1}^\infty |g_\theta(p^k) p^{-ks}| =  \sum_{k = 1}^\infty |g_\theta(p^k)| p^{-k\sigma} = \sum_{k= 2}^\infty |g_\theta(p^k)| p^{-k\sigma} \le  C p^{-2\sigma}.
	\end{equation}
	Because $\sum_p p^{-2\sigma} \le \zeta(2\sigma)<\infty$ for $\sigma>1/2$, we obtain
	\begin{equation}\label{eq:holo}
		\sum_p \sum_{k = 1}^\infty |g_\theta(p^k) p^{-ks}| = \sum_p \sum_{k = 1}^\infty |g_\theta(p^k)| p^{-k\sigma} <\infty, \quad \text{for $\sigma>1/2$}.
	\end{equation}
	Therefore by \cite[Theorem 15.4]{Rudin}, the product $ \prod_p \left(1 + \sum_{k=1}^\infty g_\theta(p^k) p^{-ks} \right)$ converges uniformly for $\Re(s)>1/2$, and we can write
	\begin{align}\label{eq:Gextend}
		G(s,\theta) &= \prod_p \big(1 + \theta p^{-s}\big)\big(1 - p^{-s}\big)^\theta \nonumber\\
		& =\prod_p (1 + \theta p^{-s}) \left(\sum_{k=0}^\infty \binom{\theta}{k} (-p^{-s})^k\right) = \prod_p \left(1 + \sum_{k=1}^\infty g_\theta(p^k) p^{-ks} \right) = \sum_{n=1}^\infty \frac{g_\theta(n)}{n^s},
	\end{align}
showing that $G$ is the Dirichlet series for the function $g_\theta$. Note that even though $F$ only converges absolutely for $\Re(s)>1$, the function $G(s,\theta)$ is holomorphic in the half-plane $\sigma > 1/2$. In particular, for $s=3/4$, we have
	\begin{equation}\label{eq:sumholo}
	\sum_{n = 1}^\infty \frac{|g_\theta(n)|}{n} (\log 3n)^{2} \ll \sum_{n=1}^\infty \frac{|g_\theta(n)|}{n^{3/4}} <\infty,
\end{equation}
verifying \cite[Ch.\ II.5, Eq.\ (5.33)]{tet} with $N=1$ and $z=\theta$. Thus, \cite[Ch.\ II.5, Theorem 5.4]{tet} guarantees an asymptotic expansion for $T(x)$ governed by the Taylor coefficients $\gamma_j(\theta) = Z^{(j)}(1, \theta), j \ge 0$ of the function $Z(s;\theta) = \big((s-1)\zeta(s)\big)^\theta/s$ around $s=1$ \cite[Ch.\ II.5, Eqs.\ (5.5) and (5.6)]{tet}. Applying \cite[Ch.\ II.5, Theorem 5.4 and Eq.\ (5.16)]{tet}  with $N=1$ and $A=\theta$, we thus obtain
	\begin{equation} \label{eq:T_expansion}
		T(x) = x (\log x)^{\theta-1} \left(\lambda_0(\theta) + \frac{\lambda_1(\theta)}{\log x} \right) + \mathcal{O}\big(x (\log x)^{\theta - 3}\big).
	\end{equation}
	To evaluate the coefficients $\lambda_k(\theta), k=0,1,$ we utilize again the evaluations $Z(1; \theta) = 1$ and $Z'(1; \theta) = \theta\gamma - 1$ from \eqref{eq:Z}. Furthermore, since $G(s; \theta)$ is holomorphic in the half-plane $\sigma > 1/2$, at $s=1$, the function is analytic, ensuring that $G(1; \theta)$ and $G'(1; \theta)$ are well-defined. Thus by \cite[Ch.\ II.5, Eq.\ (5.13)]{tet}, 
		$$
		\lambda_0(\theta) = \frac{1}{\Gamma(\theta)} G(1; \theta) Z(1; \theta) = \frac{G(1; \theta)}{\Gamma(\theta)},
		$$
		while
		$$
		\lambda_1(\theta) = \frac{1}{\Gamma(\theta - 1)} \Big[ G'(1; \theta)Z(1; \theta) + G(1; \theta)Z'(1; \theta) \Big] = \frac{(\theta-1)\big[G'(1; \theta) + G(1; \theta)(\theta\gamma - 1)\big]}{\Gamma(\theta)}.
		$$
	Note, since
	\begin{equation}\label{eq:g2exp}
					1 + \sum_{k=1}^\infty g_\theta(p^k) p^{-k} = \big(1 + \theta/p\big)\big(1 - 1/p\big)^\theta >0, 
	\end{equation}
	by \cite[Theorem 15.4]{Rudin}, we have that $G(1,\theta) =\prod_p \left(1 + \sum_{k=1}^\infty g_\theta(p^k) p^{-k} \right) > 0$. Therefore, plugging the expressions for $\lambda_0$ and $\lambda_1$ in \eqref{eq:T_expansion}, we obtain 
	\begin{equation} \label{eq:T_final}
		T(x) = \frac{x\,(\log x)^{\theta-1}}{\Gamma(\theta)} G(1; \theta) \Bigl(1 + \frac{\alpha_1}{\log x} + R(\log x)\Bigr),
	\end{equation}
	where the remainder satisfies $|R(u)| \le C u^{-2}, u >0$, for some constant $C \in (0,\infty)$ depending only on $\theta$ and $\alpha_1 = \frac{\lambda_1(\theta)}{\lambda_0(\theta)} = (\theta-1)\Bigl(\frac{G'(1; \theta)}{G(1; \theta)} + \theta\gamma - 1\Bigr) \in \R$ is finite.

	To transfer this expansion to $S(x)$, we argue exactly as in the case of $\mathbb{P}_n$. Using $T(t) = \mu^2(1) \theta^{\omega(1)}=1$ for $t \in [1,2)$, we note by Abel's summation formula that
	$$
	S(x) = \frac{T(x)}{x} + \int_1^x \frac{T(t)}{t^2} \dd t = \frac{T(x)}{x} + \int_2^x \frac{T(t)}{t^2} \dd t + \frac{1}{2}.
	$$
	Now arguing exactly as for \eqref{eq:S}, one obtains 
	\begin{equation}\label{eq:Ssq}
		S(x) = G(1; \theta) \frac{(\log x)^\theta}{\Gamma(\theta+1)}\Bigl(1+\frac{c_\theta^{\text{sq}}}{\log x}+\mathcal{O}\Bigl(\frac{\log \log x}{(\log x)^\eta}\Bigr)\Bigr),
	\end{equation}
	where $c_\theta^{\text{sq}} = \theta \Bigl(\frac{G'(1; \theta)}{G(1; \theta)} + \theta\gamma\Bigr)$ and $\eta = \min\{2, \theta\}$.
	\medskip
	
	\noindent\underline{\textit{Asymptotics of the normalising constant in \eqref{eq:Wlesq}:}} 
	Recall from \eqref{eq:ZnSq} that $Z_n^{\text{sq}} = \prod_{p \le x} \bigl(1 + \frac{\theta}{p}\bigr)$.
	We can express this via \eqref{eq:Znasymp} and \eqref{eq:G_euler} as
	$$
	Z_n^{\text{sq}} = \prod_{p\le x} \Bigl(1+\frac{\theta}{p}\Bigr)\Bigl(1-\frac1p\Bigr)^\theta \prod_{p \le x} \Bigl(1-\frac1p\Bigr)^{-\theta} = \frac{G(1; \theta) \bigl(e^\gamma\log x\bigr)^\theta\Bigl(1 + \mathcal{O}\big(e^{-b\sqrt{\log x}}\big)\Bigr)}{N(x)}
	$$
	for some constant $b \in (0,\infty)$, where we write
	$$
	N(x) := \prod_{p > x} \left(1 + \frac{\theta}{p}\right)\left(1 - \frac{1}{p}\right)^\theta.
	$$
	By \eqref{eq:g2exp} and \eqref{eq:g1exp}, we have
	$$
	\log N(x) = \sum_{p > x} \log \left(	1 + \sum_{k=1}^\infty g_\theta(p^k) p^{-k} \right) = \sum_{p > x} \log\Big(1 + \mathcal{O}\big(p^{-2}\big)\Big) = \mathcal{O}\left(\sum_{p > x} \frac{1}{p^2}\right) =  \mathcal{O}\left(\frac{1}{x}\right).
	$$
	Thus,
	$
	N(x) = \exp\big(\mathcal{O}(1/x)\big) = 1 + \mathcal{O}\big(1/x\big).
	$
	Noting that $(1 + \mathcal{O}(1/x))^{-1} = 1 + \mathcal{O}(1/x) = 1+ \mathcal{O}(e^{-\log x})$, and substituting this back in the expression for $Z_n^{\text{sq}}$ yields
	\begin{equation}\label{eq:Znsqasymp}
		Z_n^{\text{sq}} = G(1; \theta) \bigl(e^\gamma\log x\bigr)^\theta\Bigl(1 + \mathcal{O}\big(e^{-b\sqrt{\log x}}\big)\Bigr).
	\end{equation}
	
	\noindent\underline{\textit{Obtaining \eqref{eq:snsqlbint}:}} 
	Combining \eqref{eq:Wlesq}, \eqref{eq:Ssq} and \eqref{eq:Znsqasymp}, we can complete the proof arguing exactly as for $S_n$.
	In particular,
	$$
	\begin{aligned}
		\mathbb P(S_n^{\text{sq}}\le 1) = \frac{S(x)}{Z_n^{\text{sq}}} &= \frac{G(1; \theta)}{G(1; \theta)(e^\gamma\log x)^\theta}\Bigl(1 + \mathcal{O}\big(e^{-b\sqrt{\log x}}\big)\Bigr)
		\frac{(\log x)^\theta}{\Gamma(\theta+1)}\Bigl(1+\frac{c_\theta^{\text{sq}}}{\log x}+\mathcal{O}\Bigl(\frac{\log \log x}{(\log x)^\eta}\Bigr)\Bigr)\\[4pt]
		&= \frac{e^{-\gamma\theta}}{\Gamma(\theta+1)}
		\Bigl(1+\frac{c_\theta^{\text{sq}}}{\log x}+\mathcal{O}\Bigl(\frac{\log \log x}{(\log x)^\eta}\Bigr)\Bigr).
	\end{aligned}
	$$
	Since $\mathbb P(D_\theta\le 1) = \displaystyle e^{-\gamma\theta}/\Gamma(\theta+1)$, we obtain
	$$
	\mathbb P(S_n^{\text{sq}}\le 1) = \mathbb P(D_\theta\le 1) + \frac{C_\theta^{\text{sq}}}{\log p_n} + \mathcal{O}\Bigl(\frac{\log \log p_n}{(\log p_n)^\eta}\Bigr),
	$$
	where $C_\theta^{\text{sq}} = c_\theta^{\text{sq}} e^{-\gamma\theta}/\Gamma(\theta+1)$. 
	
	Finally, note from \eqref{eq:Gextend} that
	$$
	\frac{G'(s; \theta)}{G(s; \theta)} = \frac{\partial}{\partial s} \log G(s; \theta) = \sum_p \left[ \frac{-\theta p^{-s} \log p}{1 + \theta p^{-s}} + \frac{\theta p^{-s} \log p}{1 - p^{-s}} \right] = \sum_p \frac{\theta(\theta+1) p^{-2s} \log p}{(1 - p^{-s})(1 + \theta p^{-s})},
	$$ 
	yielding $\frac{G'(1; \theta)}{G(1; \theta)}>0$. Recalling that $c_\theta^{\text{sq}} = \theta \Bigl( \frac{G'(1; \theta)}{G(1; \theta)} + \theta\gamma\Bigr)$, it follows now that  $ c_\theta^{\text{sq}} \ge \theta^2 \gamma$, completing the proof.

\section{Proofs of Lemmas \ref{lem:palm_mean_bounds} and \ref{lem:laplace_Sn}} \label{app:cuelb}

\begin{proof}[Proof of Lemma \ref{lem:palm_mean_bounds}]
		For $k \in \{1,2\}$, let $\rho_k$ and $\rho_k^{(0)}$ denote the $k$-point correlation functions of the point process of eigenphases with respect to the Lebesgue measure on $(0, 2\pi]$, under the unconditional and reduced Palm measures, respectively. Since the eigenphase process is determinantal \cite{Dyson} and its reduced Palm measure remains determinantal \cite[Theorem 1.7]{ST2003}, all these correlation functions exist. We evaluate $\mu(s) = \int_0^{2\pi s} \rho_1^{(0)}(x) \, \dd x$ for $0<s <1$ by estimating $\rho_1^{(0)}(x)$. An interesting relation between the 1-point function and its Palm version (see e.g.\ \cite[Lemma 6.4]{ST2003}) is that
		\begin{equation}\label{eq:rho2}
				\rho_1(0) \rho_1^{(0)}(x) = \rho_2(0,x), \quad \text{for a.e.\ }x \in (0,2\pi).
		\end{equation}
		These functions are explicitly known; in particular (see e.g.\ \cite[Section 2.3, Page 13]{BA}), we have $\rho_1(0) = K(0) = n/2\pi$, where $K(x) = \frac{1}{2\pi} \frac{\sin(nx/2)}{\sin(x/2)}$, and
		$\rho_2(0, x) = K(0)^2 - K(x)^2$.
		From \eqref{eq:rho2}, we thus obtain
		\begin{equation}\label{eq:rho1palm}
		    \rho_1^{(0)}(x) = \frac{\rho_2(0,x)}{ \rho_1(0) } = \frac{K(0)^2 - K(x)^2}{K(0)} =  \frac{n}{2\pi} \left(1-\frac{K(x)^2}{K(0)^2}\right)\le \frac{n}{2\pi} \quad \text{for a.e. } x \in (0,2\pi).
		\end{equation}
		Integrating this expression immediately establishes the first bound in Lemma \ref{lem:palm_mean_bounds}, since
		$$
			\mu(s) = \int_0^{2\pi s} \rho_1^{(0)}(x)  \; \dd x \le n s, \quad s \in (0,1).
		$$

		When $s$ is small, the microscopic behaviour of $\mu(s)$ is different due to the repulsive effect of $0$ under the Palm measure.  By \cite[Section 2.3, Page 15]{BA}, there exists a universal constant $c \in (0,\infty)$ such that for $x \in (0,2\pi)$,
		$$
		|K(x)| \le c n, \quad \text{and} \quad |K''(x)| \le c n^3.
		$$
		Noting that $K'(0) = 0$ due to $K$ being even, Taylor's theorem yields for some $\xi \in (0, x)$ that $|K(0) - K(x)| = \frac{1}{2} |K''(\xi)|\, x^2 \le \frac{c}{2} n^3 x^2$ for $x \in (0,2\pi)$. Therefore, we can bound for $x \in (0,2\pi)$,
		$$
		0 \le \rho_1^{(0)}(x) =\frac{K(0)^2 - K(x)^2}{K(0)} = \frac{2\pi}{n} \Big( K(0) - K(x) \Big) \Big( K(0) + K(x) \Big) \le \frac{2\pi}{n} \left(\frac{c}{2} n^3x^2\right) (2c n) = C n^3 x^2,
		$$
		where $C = 2\pi c^2$. Thus integrating yields
		$$
			\mu(s) = \int_0^{2\pi s} \rho_1^{(0)}(x)  \; \dd x \le C' n^3 s^3, \quad s \in (0,1),
		$$
		where $C'= \frac{C(2\pi)^3}{3}$, establishing the second bound in \eqref{eq:meso-micro}, concluding the proof.
	\end{proof}

    \begin{proof}[Proof of Lemma \ref{lem:laplace_Sn}]
	The Laplace transform of the Dickman distribution has the exact representation $L_D := \E[e^{-D_\theta}] = \exp\left( - \theta \int_0^1 g(x) \dd x \right)$, where $g(x) := \frac{1 - e^{-x}}{x}$ (see e.g.\ \cite{Pi16b}). Note that $g(x)$ is positive and bounded on $(0,1]$.

	Given $\mathbf{W}$, the variables $X_k$ are conditionally independent, and each $X_k$ is a geometric random variable with mean $\lambda_k = \frac{\theta}{n W_{n,k}}$. Applying the standard geometric Laplace transform $\E[e^{-t X_k}] = (1 + \lambda_k(1 - e^{-t}))^{-1}$ evaluated at $t = W_{n,k}$ yields
	\begin{equation}\label{eq:Esnlb}
		\E[e^{-S_n} \mid \mathbf{W}] = \prod_{k=1}^{n-1} \left( 1 + \lambda_k(1 - e^{-W_{n,k}}) \right)^{-1} = \exp\left( - \sum_{k=1}^{n-1} \log\left( 1 + \frac{\theta}{n} g(W_{n,k}) \right) \right) = : e^{-Z_n}.
	\end{equation}
	Using the simple inequality $\log(1+x) \le x$ for $x \ge 0$, we have 
	\begin{equation}\label{eq:EZn}
			Z_n = \sum_{k=1}^{n-1} \log\left( 1 + \frac{\theta}{n} g(W_{n,k}) \right) \le \frac{\theta}{n} \sum_{k=1}^{n-1} g(W_{n,k}).
	\end{equation}
	
    By the translation invariance of the CUE eigenphases, we can upper bound $\E Z_n$ using the 1-point correlation function of the unnormalized eigenphases under the reduced Palm measure $\P_0$. Recall from \eqref{eq:rho1palm} in the proof of Lemma \ref{lem:palm_mean_bounds}  that
    $$
    \rho_1^{(0)}(x) = \frac{n}{2\pi}\left(1 - \frac{K(x)^2}{K(0)^2}\right) \quad \text{for a.e. } x \in (0,2\pi).
    $$
    Because the weights $W_{n,k} = \frac{1}{2\pi}(\theta_{k+1}-\theta_1)$ are normalized to the interval $(0,1)$, for any integrable function $f$ on $(0,1)$, we thus have 
    \begin{align*}
	\E\left[ \frac{1}{n} \sum_{k=1}^{n-1} f(W_{n,k}) \right] &= \frac{1}{n}\int_0^{2\pi} f\left(\frac{z}{2\pi}\right) \rho_1^{(0)}(z) \,\dd z \\
    &= \frac{2\pi}{n} \int_0^1 f(x) \rho_1^{(0)}(2\pi x) \, \dd x = \int_0^1 f(x) \,\dd x - \int_0^1 f(x) \bar K(x)\, \dd x,
	\end{align*}
    where $\bar K(x) := \left(\frac{K(2\pi x)}{K(0)}\right)^2 = \frac{(2\pi)^2}{n^2} K(2\pi x)^2$. Applying this identity to the bound for $Z_n$ in \eqref{eq:EZn} yields
	$$
	\E[Z_n] \le \theta \E\left[ \frac{1}{n} \sum_{k=1}^{n-1} g(W_{n,k}) \right] = \theta \int_0^1 g(x)\, \dd x - \theta \int_0^1 g(x) \bar K(x) \,\dd x.
	$$
	Because the function $g(x) = \frac{1-e^{-x}}{x}$ is decreasing and $g(1) = 1 - e^{-1} > 0$, by \eqref{eq:Esnlb} and Jensen's inequality, we therefore have
	\begin{align*}
	\E[e^{-S_n}] = \E [e^{-Z_n}] \ge e^{- \E Z_n} &\ge \exp\left( - \theta \int_0^1 g(x) \, \dd x \right) \exp\left( \theta \int_0^1 g(x) \bar K(x) \,\dd x \right) \\
    &\ge L_D \; \exp\left( \frac{C_n}{n} \right) \ge L_D\left(1+\frac{C_n}{n}\right),
	\end{align*}
	where $C_n = \theta(1-e^{-1}) \int_0^1 n\bar K(x) \, \dd x$.  Note that the scaled function $n \bar K(x) = \frac{1}{n}\left(\frac{\sin(n\pi x)}{\sin(\pi x)}\right)^2$ is the standard Fejér kernel, which integrates exactly to $1$ over $[0,1]$ (see e.g.\ \cite[Lemma 5.1]{FA}). Therefore,
	$$
	C_n = \theta (1 - e^{-1}) > 0,
	$$
	which concludes the proof.
\end{proof}

\section{Proofs of lower bound in Corollary \ref{lem:beta_shift} and Corollary \ref{cor:unified_weights}}\label{app:C}

	\begin{proof}[Proof of $d_{1,1}$ lower bound in Corollary \ref{lem:beta_shift}]
		Let $h(x) = e^{-x}$ for $x \ge 0$, so that $h \in \mathcal{H}_{1,1}$. Because $X_k \sim \text{Poi}(\theta/k), k \in [n]$ are mutually independent, we can evaluate
		\begin{align*}
			\E[h(S_n)] &= \prod_{k=1}^n \E\left[ \exp\left(-\frac{k}{n} X_k\right) \right] = \prod_{k=1}^n \exp\left( \frac{\theta}{k} (e^{-k/n} - 1) \right) = \exp\left( -  \theta \sum_{k=1}^n \frac{1}{n} g\left(\frac{k}{n}\right) \right),
		\end{align*}
		where we define $g(x) = \frac{1- e^{-x}}{x}$ for $x > 0$, with the continuous limit $g(0) = 1$. On the other hand, as noted in the proof of Lemma \ref{lem:laplace_Sn}, we have
		$$
		\E[h(D_\theta)] = \exp\left( - \theta \int_0^1 g(x) \, \dd x \right) = : L_D.
		$$
		Therefore, denoting $R_n = \int_0^1 g(x) \dd x - \frac{1}{n}\sum_{k=1}^n g(k/n)$, we can bound
		\begin{equation}\label{eq:rnbd}
			d_{1,1}(S_n, D_\theta) \ge \big|\E[h(S_n^0)] - \E[h(D_\theta)]\big| = L_D \Big| \exp(\theta R_n) - 1 \Big| \ge \theta L_D \, R_n,
		\end{equation}
		where for the final step, we used that $R_n \ge 0$, which follows by noting that $g'(x) = \frac{(x+1)e^{-x} - 1}{x^2}<0$ on $(0,1)$ so that $g$ is  decreasing on $[0,1]$. 
		Because $g'(x)$ is continuous and strictly negative on the compact set $[0,1]$, we have $c = \min_{x \in [0,1]} |g'(x)| > 0$. We can thus lower bound $R_n$ as
		$$
		R_n = \sum_{k=1}^n \int_{\frac{k-1}{n}}^{\frac{k}{n}} \left( g(x) - g(k/n) \right) \dd x \ge \sum_{k=1}^n \int_{\frac{k-1}{n}}^{\frac{k}{n}} c \left( \frac{k}{n} - x \right) \dd x = \sum_{k=1}^n c \left( \frac{1}{2n^2} \right) = \frac{c}{2n}.
		$$
		The result now follows from \eqref{eq:rnbd}.
	\end{proof}

\begin{proof}[Proof of Corollary \ref{cor:unified_weights}] The asserted Kolmogorov bound follows immediately from Theorem \ref{thm:unified_weights}, hence we only prove the $d_{1,1}$ bounds below.
	
	For Model A, let $M_k = \sum_{i=1}^k (Y_i - 1)$, $k \in [n]$ and decompose the sum as $S_n = S_n^0 + V_n$, where $S_n^0 = \sum_{k=1}^n \frac{k}{n}X_k$ and $V_n = \frac{1}{n} \sum_{k=1}^n M_k X_k$. 
	Because $\E[Y_i] = 1$, we conditionally have $\E[V_n \mid X_1, \hdots,X_n] = 0$. For any test function $h \in \mathcal{H}_{1,1}$, A Taylor series expansion $h(S_n)$ around $S_n^0$ thus yields
	\begin{equation}\label{eq:snsn0}
		\E [h(S_n) - h(S_n^0)] = \E[h'(S_n^0) V_n] + \E [E_n] = \E\big[h'(S_n^0) \E[V_n \mid X_1, \hdots,X_n]\big] + \E [E_n] = \E [E_n],
	\end{equation}
	where the remainder term satisfies $|E_n| \le \frac{1}{2} \|h''\|_\infty V_n^2 \le \frac{1}{2} V_n^2$. Rewriting $V_n$ as $V_n = \frac{1}{n}\sum_{i=1}^n (Y_i-1) \sum_{k=i}^n X_k$, and noting that $\sum_{k=i}^n X_i$ is Poisson distributed with mean $\theta \sum_{k=i}^n 1/k = \mathcal{O}\big(\log (n/i)\big)$, the independence of the increments yields
	$$
	2 \E [E_n] \le \E[V_n^2] = \frac{1}{n^2} \sum_{i=1}^n \Var(Y_i) \, \E\left[ \Big(\sum_{k=i}^n X_k\Big)^2 \right] = \frac{\sigma_{\max}^2}{n^2} \mathcal{O}\Big(\sum_{i=1}^n \log^2(n/i)\Big) = \mathcal{O}\left(1/n\right).
	$$
	The final step above follows by estimating the sum as
	\begin{multline}\label{eq:on}
		\sum_{i=2}^n \log^2(n/i) \le \int_1^n \log^2(n/x) \, \dd x = n \int_{1/n}^1 \log^2(1/u) \, \dd u = n \int_{1/n}^1 \log^2 u \, \dd u \\= n [u\log^2(u) - 2u\log(u) + 2u]_{1/n}^1 
		= n \left(2 - \frac{\log^2 n + 2\log n + 2}{n}\right),
	\end{multline}
	yielding $\sum_{i=1}^n \log^2(n/i) \le 2n$. Thus, from \eqref{eq:snsn0}, taking a supremum over all $h \in \mathcal{H}_{1,1}$, we obtain $d_{1,1} (S_n^0, S_n) = \mathcal{O}(1/n)$. Since $d_{1,1}(S_n^0, D_\theta) = \mathcal{O}(1/n)$ by Corollary \ref{lem:beta_shift}, the triangle inequality guarantees $d_{1,1}(S_n, D_\theta) = \mathcal{O}(1/n)$, yielding the conclusion for model A.
	
	For Model B, since $W_{n,k} = \frac{k + M_k}{S_n^Y}$, we have
	$$
	W_{n,k} - \frac{k}{n} = \frac{k + M_k}{S_n^Y} - \frac{k}{n} = \frac{n M_k - k M_n}{n \, S_n^Y} = \left(\frac{M_k - \frac{k}{n}M_n}{n}\right) \frac{n}{S_n^Y}.
	$$
	Writing $\frac{n}{S_n^Y} = 1 - \frac{M_n}{S_n^Y}$, we can thus decompose $S_n - S_n^0 = V_n = U_n + R_n$, where
	$$
	U_n = \frac{1}{n} \sum_{k=1}^n \left(M_k - \frac{k}{n} M_n \right) X_k \quad \text{and} \quad R_n = - \frac{M_n}{S_n^Y} U_n.
	$$
	Because $(M_k)_{k \in \N}$ is centered, we have $\E[U_n \mid X_1, \hdots,X_n] = 0$. The Taylor expansion therefore yields for $h \in \mathcal{H}_{1,1}$ that
	\begin{align}\label{eq:errr}
		\E[h(S_n) - h(S_n^0)] &= \E[h'(S_n^0) R_n] + \E[E_n] \le  \E [|R_n|] + \E [|E_n|],
	\end{align}
	where $|E_n| \le \frac{1}{2}\E[(U_n+R_n)^2] = \frac{1}{2}\E \left[ \left( nU_n/S_n^Y\right)^2 \right]$. 
	Exchanging the sums in $U_n$, we can rewrite it as
	$$U_n = \frac{1}{n} \sum_{i=1}^n (Y_i -1) \left( \sum_{k=i}^n X_k - \sum_{k=1}^n \frac{k}{n} X_k \right) =  \frac{1}{n} \sum_{i=1}^n (Y_i -1) \left( \sum_{k=i}^n X_k - S_n^0 \right).$$
	Thus by independence and using \eqref{eq:on}, we have
	\begin{multline}\label{eq:unbd}
		\E [U_n^2] = \frac{1}{n^2} \sum_{i=1}^n \Var(Y_i) \, \E\left[ \Big(\sum_{k=i}^n X_k -S_n^0 \Big)^2 \right] \le  \frac{2 \sigma_{\max}^2}{n^2} \E\left[ \Big(\sum_{k=i}^n X_k\Big)^2 + (S_n^0)^2\right]\\
		= \frac{2 \sigma_{\max}^2}{n^2} \left[ \mathcal{O}\Big(\sum_{i=1}^n \log^2(n/i)\Big) + \E [(S_n^0)^2]\right] = \frac{2 \sigma_{\max}^2}{n^2} \left[ \mathcal{O}(n) + \mathcal{O}(1)\right] = \mathcal{O}\left(1/n\right),
	\end{multline}
	where in the penultimate step we have used that
	$$
	\E [(S_n^0)^2] = \operatorname{Var}(S_n^0) + \theta^2 = \sum_{k=1}^n \frac{\theta k}{n^2}+ \theta^2  = \mathcal{O}(1).
	$$
	
	Recall the event $\mathcal{E} = \{S_n^Y > n/2\}$ introduced in the proof of Lemma \ref{lem:expected_error_unified}. Since $n/S_n^Y<2$ on $\mathcal{E}$, by the Cauchy-Schwarz inequality and \eqref{eq:unbd}, we have
	$$
	\E[|R_n| \ind_{\mathcal{E}}] \le 2 \, \E \left| U_n \frac{M_n}{n} \right| \le 2 \|U_n\|_2 \|M_n/n\|_2 = \mathcal{O}\left(1/n\right),
	$$
	where we used $\E[(M_n/n)^2] = \frac{1}{n^2}\sum_{i=1}^n \sigma_i^2 = \mathcal{O}(1/n)$. Similarly, recalling from \eqref{eq:errr} that $|E_n| \le (1/2) \E \left[ \left( nU_n/S_n^Y\right)^2 \right]$, using  \eqref{eq:unbd} we have 
	\begin{equation}\label{eq:Enbd}
		\E[|E_n| \ind_{\mathcal{E}}] \le \E \left[\frac{1}{2} (nU_n/S_n^Y)^2 \ind_{\mathcal{E}} \right] \le 2\E[U_n^2] = \mathcal{O}(1/n).
	\end{equation}
	
	On $\mathcal{E}^c$, we have $\P(\mathcal{E}^c) \le \frac{4}{n^2}\sum_{i=1}^n \sigma_i^2 = \mathcal{O}(1/n)$ from \eqref{eq:tailbd}. Since $\|h'\|_\infty \le 1$, note from the definition of $E_n$ in \eqref{eq:errr} that we can bound $|E_n|$ as
	$$
	|E_n| = \big|h(S_n) - h(S_n^0) - h'(S_n^0)V_n\big| \le |h(S_n) - h(S_n^0)| + \|h'\|_\infty |V_n| \le 2|V_n|.
	$$
	Moreoever, noting that  $W_{n,k} \le 1$ in Model B for $k \in [n]$, we have $|V_n| = |S_n - S_n^0| \le 2\sum_{k=1}^n X_k$ almost surely. Thus,
	\begin{equation}\label{eq:evbd}
		\E[|E_n| \ind_{\mathcal{E}^c}] \le 2 \E[|V_n| \ind_{\mathcal{E}^c}] \le 4\E\bigg[\sum_{k=1}^n X_k\bigg] \P(\mathcal{E}^c) = \mathcal{O}(\log n) \cdot \mathcal{O}\left(\frac{1}{n}\right) = \mathcal{O}\left(\frac{\log n}{n}\right).
	\end{equation}
	This combined with \eqref{eq:Enbd} yields $\E [|E_n|] = \mathcal{O}(\log n/n)$. Since $R_n = V_n - U_n$, we also have 
	$$
	\E|R_n| \le \E[|R_n| \ind_{\mathcal{E}}] + \E[|V_n| \ind_{\mathcal{E}^c}] + \E[|U_n| \ind_{\mathcal{E}^c}].
	$$ 
	Noting from \eqref{eq:unbd} and \eqref{eq:tailbd} that $\E[|U_n| \ind_{\mathcal{E}^c}] \le \sqrt{\E [U_n^2] \P(\mathcal{E}^c)} = \mathcal{O}(1/n)$, we combine the bounds via \eqref{eq:errr} and take a supremum over all $h \in \mathcal{H}_{1,1}$ to obtain $d_{1,1}(S_n,S_n^0) = \mathcal{O}(\log n/n)$. Recalling that $d_{1,1}(S_n^0, D_\theta) = \mathcal{O}(1/n)$ by Corollary \ref{lem:beta_shift}, this yields the result for Model B by the triangle inequality.

	Note that if we have $\max_{i \in [n]} \E [Y_i^{2+\epsilon}]<\infty$ for some $\epsilon>0$, then one can improve the tail bound in \eqref{eq:tailbd} using the Marcinkiewicz–Zygmund inequality as
	$$
	\P(\mathcal{E}^c) \le \P(|S_n^Y - n| \ge n/2) \le \frac{\E\Big[ \Big| \sum_{i=1}^n (Y_i - 1) \Big|^{2+\epsilon} \Big]}{n^{2+\epsilon}} \le \frac{\mathcal{O}(n^{1 + \epsilon/2})}{n^{2+\epsilon}} = \mathcal{O}\left( \frac{1}{n^{1 + \epsilon/2}} \right).
	$$
	This immediately improves \eqref{eq:evbd} to 
	$$
	\E[|E_n| \ind_{\mathcal{E}^c}] \le 2 \E[|V_n| \ind_{\mathcal{E}^c}] \le 4\E\bigg[\sum_{k=1}^n X_k\bigg] \P(\mathcal{E}^c) = \mathcal{O}(\log n) \cdot \mathcal{O}\left( \frac{1}{n^{1 + \epsilon/2}} \right) = \mathcal{O}\left(\frac{1}{n}\right).
	$$
	allowing one to obtain a bound of the order $1/n$ on $d_{1,1}(S_n, D_\theta)$ in Model B, concluding the proof.
\end{proof}

	\end{document}